\documentclass[reqno]{amsart}
\usepackage{amssymb,amsmath,hyperref}
\usepackage{amsrefs}
\usepackage[foot]{amsaddr}
\usepackage{bbold,stackrel}
 
\newtheorem{thm}{Theorem}

\newtheorem{cor}[thm]{Corollary}
\newtheorem{defi}[thm]{Definition}
\newtheorem{rem}[thm]{Remark}
\newtheorem{nota}[thm]{Notation}

\newtheorem{ack}[thm]{Acknowledgement}

\newtheorem*{tempo*}{Template}
\newtheorem{theorem}[thm]{Theorem}

\newtheorem{lemma}[thm]{Lemma}
\newtheorem{definition}[thm]{Definition}
\newtheorem{corollary}[thm]{Corollary}

\newtheorem{convention}[thm]{Convention}
\newtheorem{proposition}[thm]{Proposition}

\newcommand\be{\begin{equation}}
\newcommand\ee{\end{equation}} 

\usepackage{amsmath,amsfonts} 

\usepackage[applemac]{inputenc}

\def\bdefi{\begin{defi}\rm}
\def\edefi{\end{defi}}
\def\bnota{\begin{nota}\rm}
\def\enota{\end{nota}}
\def\FIVE{\Pi_{1}^{1}\text{-\textup{\textsf{CA}}}_{0}}
\def\FIVEK{\Pi_{k}^{1}\text{-\textup{\textsf{CA}}}_{0}}
\def\FIVEo{\Pi_{1}^{1}\text{-\textup{\textsf{CA}}}_{0}^{\omega}}
\def\FIVEFIVE{\Delta_{2}^{1}\text{-\textsf{\textup{CA}}}_{0}}
\def\SIX{\Pi_{2}^{1}\text{-\textsf{\textup{CA}}}_{0}}

\def\ATR{\textup{\textsf{ATR}}}
\def\ind{\textup{\textsf{ind}}}

\def\ZFC{\textup{\textsf{ZFC}}}
\def\IST{\textup{\textsf{IST}}}

\def\MUO{\textup{\textsf{MUO}}}

\def\STP{\textup{\textsf{STP}}}

\def\DNR{\textup{\textsf{DNR}}}

\def\ns{\textup{\textsf{ns}}}
\def\RCA{\textup{\textsf{RCA}}}
\def\({\textup{(}}
\def\){\textup{)}}
\def\WO{\textup{\textsf{WO}}}
\def\RCAo{\textup{\textsf{RCA}}_{0}^{\omega}}
\def\ACAo{\textup{\textsf{ACA}}_{0}^{\omega}}
\def\WKL{\textup{\textsf{WKL}}}

\def\WWKL{\textup{\textsf{WWKL}}}
\def\bye{\end{document}}
\def\P{\textup{\textsf{P}}}
\def\N{{\mathbb  N}}

\def\R{{\mathbb  R}}

\def\MUC{\textup{\textsf{MUC}}}
\def\st{\textup{st}}
\def\di{\rightarrow}
\def\asa{\leftrightarrow}
\def\ACA{\textup{\textsf{ACA}}}
\def\paai{\Pi_{1}^{0}\textup{-\textsf{TRANS}}}
\def\Paai{\Pi_{1}^{1}\textup{-\textsf{TRANS}}}
\def\QFAC{\textup{\textsf{QF-AC}}}

\def\HBU{\textup{\textsf{HBU}}}
\def\CK{\textup{\textsf{CK}}}
\def\TRO{\textup{-\textsf{TR}}_{0}}

\def\POS{\textup{\textsf{POS}}}

\def\UATR{\textup{\textsf{UATR}}}

\def\mSEP{\textup{\textsf{-SEP}}}

\def\con{\textup{\textsf{con}}}

\def\SU{\textup{\textsf{SU}}}

\def\EPA{\textup{\textsf{E-PA}}}
\def\EPRA{\textup{\textsf{E-PRA}}}

\def\ECF{\textup{\textsf{ECF}}}

\def\LMP{\textup{\textsf{LMP}}}
\def\SCF{\textup{\textsf{SCF}}}
\def\SFF{\textup{\textsf{SFF}}}
\def\WFF{\textup{\textsf{WFF}}}

\def\SOT{\textup{\textsf{SOT}}}

\def\WCF{\textup{\textsf{WCF}}}
\def\HAC{\textup{\textsf{HAC}}}
\def\INT{\textup{\textsf{int}}}

\newcommand{\T}{\mathcal{T}}

\newcommand{\m}{{\bf m}}

\usepackage{graphicx}
\usepackage{tikz}
\usetikzlibrary{matrix, shapes.misc}

\numberwithin{equation}{section}
\numberwithin{thm}{section}

\usepackage{comment}

\begin{document}
\title[Compactness in Computability Theory and Nonstandard Analysis]{The strength of compactness in Computability Theory and Nonstandard Analysis}
\author{Dag Normann}
\address{Department of Mathematics, The University of Oslo}
\email{dnormann@math.uio.no}
\author{Sam Sanders}
\address{TU Darmstadt \& University of Leeds}
\email{sasander@me.com}

\begin{abstract}
Compactness is one of {the} core notions of analysis: it connects local properties to global ones and makes limits well-behaved.  
We study the computational properties of the compactness of Cantor space $2^{\N}$ \emph{for uncountable covers}.  The most basic question is: \emph{how hard is it to {compute} a finite sub-cover from such a cover of $2^{\N}$}?   
Another natural question is: \emph{how hard is it to {compute} a sequence that covers $2^\N$ minus a measure zero set from such a cover}?  
The \emph{special} and \emph{weak} fan functionals respectively compute such finite sub-covers and sequences.  
In this paper, we establish the connection between these new fan functionals on one hand, and various well-known comprehension axioms on the other hand, including arithmetical comprehension, transfinite recursion, and the Suslin functional.
In the spirit of \emph{Reverse Mathematics}, we also analyse the logical strength of compactness in \emph{Nonstandard Analysis}.  Perhaps surprisingly, the results in the latter mirror (often perfectly) the computational properties of the special and weak fan functionals.
In particular, we show that compactness (nonstandard or otherwise) readily brings us to the outer edges of Reverse Mathematics (namely $\SIX$), and even into Schweber's higher-order framework (namely $\Sigma_{1}^{2}$-separation).  
\end{abstract}

\newpage
\maketitle
\thispagestyle{empty}

\section{Introduction}
The importance of (open-cover) \emph{compactness} can hardly be overstated, as it allows one to treat uncountable sets like Cantor space as `almost finite' by connecting local properties to global ones. 
A famous example is \emph{Heine's theorem}, i.e.\ the \emph{local} property of continuity implies the \emph{global} property of \emph{uniform} continuity on the unit interval.  
In general, Tao writes:  
\begin{quote}
Compactness is a powerful property of spaces, and is used in many ways in many
different areas of mathematics. One is via appeal to local-to-global principles; one
establishes local control on some function or other quantity, and then uses compactness to boost the local control to global control. \cite{taokejes}*{p.\ 168}
\end{quote} 
Compactness already has a long history: the \emph{Cousin lemma} (\cite{cousin1}*{p.\ 22}) on the open-cover compactness of subsets of $\R^{2}$, 
dates back\footnote{The collected works of Pincherle contain a footnote (see \cite{tepelpinch}*{p.\ 67}) which states that the associated \emph{Teorema} from 1882 corresponds to the Heine-Borel theorem.  This claim is repeated in \cite{werkskes}. Moreover, Weierstrass proves the Heine-Borel theorem (without explicitly formulating it) in 1880 in \cite{amaimennewekker}*{p.\ 204}.   A detailed motivation for these claims may be found in \cite{medvet}*{p. 96-97}.} 135 years. 
Despite its basic nature, its central role in analysis, and a long history, little is known about the logical and computational properties of compactness. 
The main aim of this paper is to study these computational properties.
In particular, we are interested in the following most basic and natural question (and its variations): 
\begin{center}
\emph{Given an uncountable cover of $2^{\N}$, how hard is it to compute a finite sub-cover}?  
\end{center}
To answer this question, we continue the project initiated in \cite{dagsam}, namely we study the computational properties of \emph{special fan functionals} (and their variations).  The latter compute the aforementioned finite sub-covers, as detailed in \textsf{(T.1)} below. 
In the spirit of \emph{Reverse Mathematics}, we also analyse the logical strength of compactness in \emph{Nonstandard Analysis} as in \textsf{(T.2)} below.  As it happens, the results in Nonstandard Analysis mirror (often perfectly) the results in {Computability Theory}. 
We assume basic familiarity with the aforementioned fields, in particular the program {Reverse Mathematics} founded by Friedman (RM hereafter; see \cite{simpson2,stillebron,simpson1} or \cite{dagsam}*{\S2.2}).   
In Section~\ref{weako}, we provide an overview of the results in \cite{dagsam}, and a list of the questions to be answered, all pertaining to the following two topics. 
Many questions left open in, or raised by, \cite{dagsam} are in fact answered in this paper.  We refer to \cite{dagsam, dagsamIII} for an introduction and overview to the project this paper is part of.  
In this paper, we explore the following topics:

\medskip

Topic \textsf{(T.1)}: We study two new classes of functionals, namely the \emph{special fan functionals}, an instance of which is denoted $\Theta$, and the (computationally weaker) \emph{weak fan functionals}, an instance of which is denoted $\Lambda$.  Intuitively speaking, any $\Theta$ computes a finite sub-cover from an uncountable cover of Cantor space, while any $\Lambda$ provides such a cover `in the limit'.  
These functionals are quite natural mathematical objects:
The special fan functionals emerge \emph{naturally and directly} from Tao's \emph{metastability} (\cite{samflo}) while the existence of $\Theta$  
 is equivalent to \emph{Cousin's lemma} (\cite{dagsamIII}*{\S3.3}), and to many basic properties of the \emph{gauge integral}; the latter in turn provides a unique/direct\footnote{There are a number of different approaches to the formalisation of Feynman’s path integral.  However, 
if one requires the formalisation to be close to Feynman’s original formulation, then the gauge integral
is the only approach (see \cite{dagsamIII}*{\S3.3} for a discussion).  Another argument in favour of the gauge integral is that this formalism gives rise to so-called {physical} solutions, i.e.\ in line with the observations from physics  (see \cite{pouly,nopouly,nopouly2,nopouly3}), in particular the absence of `imaginary time'.} formalisation of Feyman's path integral (\cite{mullingitover}).
From the perspective of higher-order computability theory, these new fan functionals are interesting as they fall \emph{outside} the well-studied classes, like e.g.\ the continuous functionals or the so-called normal functionals. 
In this paper, we establish the connection between $\Lambda$ and $\Theta$ on one hand, and arithmetical comprehension, transfinite recursion, and the Suslin functional on the other hand.  
The new fan functionals will be seen to exhibit rather surprising behaviour.  
 
\medskip 
 
Topic \textsf{(T.2)}: We study the \emph{nonstandard counterparts} of the `Big Five' systems $\WKL_{0}$, $\ACA_{0}$, and $\FIVE$ of RM.  These counterparts are respectively: the nonstandard compactness of Cantor space $\STP$, the \emph{Transfer} axiom limited to $\Pi_{1}^{0}$-formulas $\paai$, and the \emph{Transfer} axiom limited to $\Pi_{1}^{1}$-formulas $\Paai$.  While the original Big Five systems are linearly ordered  as follows
\[
\FIVE\di\ATR_{0} \di\ACA_{0}\di \WKL_{0}\di \RCA_{0},
\] 
the non-implications $\paai\not\di \STP\not\!\leftarrow \Paai$ hold for the respective nonstandard counterparts, as proved in \cite{dagsam}.  
In this paper, we study the strength of $\Paai+\STP$ which (indirectly) yields results about the strength of the combination of $\Theta$ and the Suslin functional.  
We study Schweber's third-order framework (\cite{schtreber,schtreberphd}) via Nonstandard Analysis and obtain some results 
involving \emph{compactness of function spaces}.  While interesting in its own right, the aforementioned compactness is essential to the \emph{gauge integral} over function spaces, which in turn formalises the \emph{Feynman path integral}.

\medskip

As it turns out, topics \textsf{(T.1)} and \textsf{(T.2)} are intimately connected: (non-) computability results in \textsf{(T.1)} are obtained \emph{directly} from (non-) implications in \textsf{(T.2)}, and vice versa.  In fact, $\Theta$ first arose from nonstandard compactness as in $\STP$ when studying the computational content of Nonstandard Analysis (\cite{samGH}), while instances of the axiom \emph{Transfer} give rise to (well-known) comprehension and choice functionals.  
As it happens, the connection between $\Theta$ and metastability was first proved \emph{via Nonstandard Analysis} (\cite{samflo}).  It should be noted that our definition of these new fan functionals, to be found in Section \ref{weakopi}, is \emph{different} from the (original) definition used in e.g.\ \cite{samGH}.   The definitions are equivalent as shown in Section \ref{peqingduck}.  

\medskip 

We now sketch the main results of this paper as follows.  A detailed discussion may be found in Section \ref{qopen}.  Feferman's $\mu^{2}$ is introduced in Section \ref{knowledge} and constitutes a form of arithmetical comprehension.
\begin{enumerate}
  \renewcommand{\theenumi}{\roman{enumi}}
\item The Suslin functional is not computable from $\Theta+\mu^{2}$ (Section \ref{prelim2}).
The combination $\Theta+\mu^{2}$ (directly) computes a realiser for $\ATR_{0}$ (Section \ref{atrsec}).  
\item  The combination $\Paai+\STP$ exists at the level of $\SIX$ (Section~\ref{SIXTUS}).  This result yields results not involving Nonstandard Analysis.    
\item We identify a weak fan functional $\Lambda_{1}$ and show in Section~\ref{essdag} that $\Lambda_{1}+\mu^{2}$ computes the same objects as $\mu^{2}$.  
This shows that we cannot in general compute a special fan functional from a weak one.  
\item We show that some of our results, Theorem \ref{mikeh} in particular, generalise to Schweber's third-order arithmetic \cites{schtreber,schtreberphd} (Section \ref{schweber}).  
\end{enumerate}
Finally, this paper connects Computability Theory and Nonstandard Analysis.  The first author contributed most results in the former, while the second author did so for the latter.  However, many questions were answered by translating them from one field to the other, 
solving them, and translating everything back, i.e.\ both authors contributed somehow to most of the paper.  
As suggested by the above, this paper is part of a series of papers by the authors, as follows.  
In our first two papers (\cite{dagsam} and this paper) we link Nonstandard Analysis and higher order Computability Theory, while in the other three (\cites{dagsamIII, dagsamV, dagsamVI}) we focus on the logical and computational content of classical theorems in mathematical analysis.

\section{Previous work and open questions}\label{weako}
We introduce the weak and special fan functionals and discuss their connection to nonstandard compactness.
We discuss the associated results in Computability Theory and Nonstandard Analysis from \cite{dagsam} and list the open questions to be answered below.  
We first make our notion of `computability' precise as follows.  
\begin{enumerate}
\item[(I)] We adopt $\ZFC$ set theory as the official metatheory for all results, unless explicitly stated otherwise.
\item[(II)] We adopt Kleene's notion of \emph{higher-order computation} as given by his nine clauses S1-S9 (see \cites{longmann, Sacks.high}) as our official notion of `computable'.
\end{enumerate}
In Section \ref{dagsec}, we provide the basic definitions of Computability Theory needed for (II), but do assume some familiarity with Computability Theory as a whole.
We refer to \cite{dagsam}*{\S2} or \cite{SB} for an introduction to Nelson's system $\IST$ and the fragments $\P$ and $\P_{0}$ which are conservative extensions of Peano arithmetic and $\RCA_{0}$.
For completeness, the systems $\P$ and $\P_{0}$ can be found in Appendix \ref{tomzeiker}. 

\medskip

Finally, to improve readability, we often omit types if they can be gleaned from context; we sometimes make use of set theoretical notation.  For instance, `$\alpha^{1}\leq1$' expresses that $\alpha$ is a binary sequence, but could also be written $\alpha\leq 1$ or $\alpha\in 2^{\N}$ or $\alpha \in C$.
Details regarding the former notation may be found in Notation \ref{equ}.
\subsection{The special and weak fan functionals}\label{weakopi}
First of all, we define two new classes of functionals.  The \emph{special fan functionals} intuitively output a finite sub-cover on input an uncountable cover of $2^{\N}$.  The (computationally weaker) \emph{weak fan functionals} take an additional input $k\in \N$ and output a finite sub-cover for a subset of $2^{\N}$ of measure at least $1-\frac{1}{2^{k}}$.  We usually simplify the type of these fan functionals to `$3$'.  We reserve the symbols $\Theta$  and $\Lambda$ to denote instances of the special and weak fan functionals.  It goes without saying these functionals are not unique: one can always add extra binary sequences to the finite sub-cover.  

\medskip

We now introduce the class of special fan functionals. 
We write `$f\in [\sigma]$' for $\overline{f}|\sigma|=_{0^{*}}\sigma$, where $\tau^{*}$ is the type of finite sequences of type $\tau$ objects.  For $w^{\tau^{*}}=\langle t_{0}, \dots, t_{k}\rangle$, we write $|w|=k+1$ and $w(i)=t_{i}$ for $i<|w|$. 
These `finite sequence' notations are discussed in detail in Notation \ref{skim}.
\bdefi[Special fan functionals]\label{dodier}
$\SFF(\Theta)$ is as follows for $\Theta^{2\di 1^{*}}$:
\be\label{kijkma}
(\forall G^{2})(\forall f^{1}\leq1)(\exists g\in \Theta(G))(f\in [\overline{g}G(g)]), 
\ee
Any functional $\Theta$ satisfying $\SFF(\Theta)$ is referred to as a \emph{special fan functional}.
\edefi
Intuitively, any functional $G^{2}$ gives rise to a `canonical cover' $\cup_{f\in 2^{\N}}[\overline{f}G(f)]$ of Cantor space, and $\Theta(G)$ is a finite sub-cover thereof, i.e.\ $\cup_{g\in \Theta(G)}[\overline{g}G(g)]$ also covers $2^{\N}$.
Note that Cousin (\cite{cousin1}) and Lindel\"of (\cite{blindeloef}) make use of such canonical covers (for $\R^{n}$) rather than the modern/general notion of cover.
In light of \eqref{kijkma}, special fan functionals may be called `realisers for the Heine-Borel theorem or Cousin lemma for $C$'.
As it happens, $\Theta$ actually arises from the \emph{nonstandard compactness of $C$} as in \emph{Robinson's theorem} (\cite{loeb1}*{p.\ 42}), as discussed in Sections \ref{pampson} and \ref{peqingduck}.     

\medskip

Secondly, we introduce the class of \emph{weak fan functionals} $\Lambda$, which are strictly weaker than $\Theta$ in general.  
As will become clear below, $\Lambda$ is not just `more of the same' but occupies an important place relative to $\Theta$.  
Where $\Theta(G)$ provides a finite sub-cover of $C$, $\Lambda(G,k)$ only yields a finite sub-cover of a subset of $C$ with measure at least $1-\frac{1}{2^{k}}$, i.e.\ we have the following:
\be\label{forp}\textstyle
\m(\{f\in C: (\exists g\in \Lambda(G, k))(f\in [\overline{g}G(g)])\} )\geq 1-\frac{1}{2^{k}},
\ee
where $\m$ is the usual coin-toss measure on $2^{\N}$.
It is straightforward, but cumbersome, to formally express \eqref{forp} in our formal language.
\bdefi[Weak fan functionals]\label{wdodier}
$\WFF(\Lambda)$ is as follows for $\Lambda^{(2\times 0)\di 1^{*}}$:
\be\textstyle\label{stylez}
(\forall G^{2}, k^{0})\big[\m(\{f\in C: (\exists g\in \Lambda(G, k))(f\in [\overline{g}G(g)])\} )\geq 1-\frac{1}{2^{k}}
\big].
\ee
Any functional $\Lambda$ satisfying $\WFF(\Lambda)$ is referred to as a \emph{weak fan functional}.
\edefi
Weak fan functionals are not realisers of theorems from the literature, but these functionals do capture the core complexity of several theorems concerning measure-theoretic approximations, like the Vitali Covering Theorem (\cite{vitaliorg}). This is investigated further in \cite{dagsamVI}.
As it happens, weak fan functionals also arise from \emph{nonstandard compactness}, as discussed in Sections \ref{pampson} and \ref{peqingduck}.     

\medskip

Finally, $\Theta$ appears similar in name and behaviour to Tait's `classical' fan functional (esp.\ on the continuous functionals).  
However, $\Theta$ and $\Lambda$ behave quite differently in that 
they cannot be computed by \emph{any} type two functional (see Section~\ref{knowledge}).  

\subsection{Nonstandard compactness and related notions}\label{pampson}
In this section, we introduce some axioms of Nonstandard Analysis.  We will observe that the special and weak fan functionals emerge from the \emph{nonstandard compactness of Cantor space}. 

\medskip
  
First of all, we mention the crucial theorem which connects $\P$ and Peano arithmetic.  Definitions may be found in \cite{brie}, \cite{dagsam}*{\S2}, \cite{SB}*{Appendix}, or Appendix \ref{tomzeiker}
\begin{thm}[Term extraction]\label{consresultcor}
If $\Delta_{\INT}$ is a collection of internal formulas and $\psi$ is internal, and
\be\label{bog}
\P + \Delta_{\INT} \vdash (\forall^{\st}\underline{x})(\exists^{\st}\underline{y})\psi(\underline{x},\underline{y}, \underline{a}), 
\ee
then one can extract from the proof a sequence of closed terms $t$ in $\mathcal{T}^{*}$ such that
\be\label{effewachten}
\textup{\textsf{E-PA}}^{\omega*} + \Delta_{\INT} \vdash (\forall \underline{x})(\exists \underline{y}\in t(\underline{x}))\psi(\underline{x},\underline{y},\underline{a}).
\ee
\end{thm}
\begin{proof}
See \cite{samGH}*{\S2} or \cite{SB}*{Appendix}.  The route from \eqref{bog} to \eqref{effewachten} involves a functional interpretation called `$S_{\st}$', introduced in \cite{brie}. 
\end{proof}
The system $\RCAo\equiv \textsf{E-PRA}^{\omega}+\QFAC^{1,0}$ is Kohlenbach's \emph{base theory of higher-order Reverse Mathematics} as introduced in \cite{kohlenbach2}*{\S2}.  
We permit ourselves a slight abuse of notation by also referring to the system $\textsf{E-PRA}^{\omega*}+\QFAC^{1,0}$ as $\RCAo$.
\begin{cor}\label{consresultcor2}
The previous theorem and corollary go through for $\P$ and $\textsf{\textup{E-PA}}^{\omega*}$ replaced by $\P_{0}\equiv \textsf{\textup{E-PRA}}^{\omega*}+\T_{\st}^{*} +\HAC_{\INT} +\textsf{\textup{I}}+\QFAC^{1,0}$ and $\RCAo$.  
\end{cor}
From now on, the notion `normal form' refers to a formula as in \eqref{bog}, i.e.\ of the form $(\forall^{\st}x)(\exists^{\st}y)\varphi(x,y)$ for $\varphi$ internal.  
We now provide a general template how term extraction is used below, as this will shorten a number of proofs.  
\begin{rem}[Using term extraction]\label{doeisnormaal}\rm
First of all, term extraction as in Theorem~\ref{consresultcor} is restricted to normal forms.  We now show that normals forms are `closed under implication', as follows. 
Let $\varphi, \psi$ be internal and consider the following implication between normal forms:
\be\label{nora}
(\forall^{\st}x)(\exists^{\st}y)\varphi(x, y)\di (\forall^{\st}z)(\exists^{\st}w)\psi(z, w).  
\ee
Since standard functionals have standard output for standard input, \eqref{nora} implies
\be\label{nora2}
(\forall^{\st}\zeta)\big[(\forall^{\st}x)\varphi(x, \zeta(x))\di (\forall^{\st}z)(\exists^{\st}w)\psi(z, w)\big].  
\ee
Bringing all standard quantifiers outside, we obtain the following normal form:
\be\label{nora3}
(\forall^{\st}\zeta, z)(\exists^{\st} w, x)\big[\varphi(x, \zeta(x))\di \psi(z, w)\big],
\ee
as the formula in square brackets is internal.  Now, \eqref{nora3} is equivalent to \eqref{nora2}, but one usually weakens the latter as follows:  
\be\label{nora4}
(\forall^{\st}\zeta, z)(\exists^{\st} w)\big[(\forall x)\varphi(x, \zeta(x))\di \psi(z, w)\big],
\ee
as \eqref{nora4} is closer to the usual mathematical definitions.  

\medskip
\noindent
Secondly, assuming \eqref{nora} is provable in $\P$, so is \eqref{nora4} and we obtain a term $t$ with
\be\label{nora5}
(\forall \zeta, z)(\exists w\in t(\zeta, z))\big[(\forall x)\varphi(x, \zeta(x))\di \psi(z, w)\big]
\ee
being provable in $\EPA^{\omega*}$.  We now omit the term $t$ and bring all quantifiers inside again, yielding that $\EPA^{\omega*}$ proves:
\be\label{nora6}
(\exists \zeta)(\forall x)\varphi(x, \zeta(x))\di  (\forall z)(\exists w)\psi(z, w).
\ee
Finally, we shall often shorten the below proofs by just providing normal forms and jumping straight from \eqref{nora} to \eqref{nora6} whenever possible.  
 \end{rem}  
Secondly, $\P$ does not involve Nelson's axiom \emph{Transfer}, as `small' fragments are already quite strong.  
Indeed, \emph{Transfer} restricted to $\Pi_{1}^{0}$-formulas as follows
\be\tag{$\paai$}
(\forall^{\st}f^{1})\big[  (\forall^{\st}n)f(n)\ne0 \di (\forall m)f(m)\ne0  \big]
\ee
is the nonstandard counterpart of arithmetical\footnote{Similar to how one `bootstraps' $\Pi_{1}^{0}$-comprehension to the latter, the system $\P_{0}+\paai$ proves $\varphi\asa \varphi^{\st}$ for any internal arithmetical formula (only involving standard parameters).} comprehension as in $\ACA_{0}$.    
Furthermore, the fragment\footnote{The `bootstrapping' trick for $\paai$ does not work for $\Paai$ (or $\FIVE$) as the latter is restricted to type one objects (like $g^{1}$ in $\Paai$) occurring as `call by value'. } of \emph{Transfer} for $\Pi_{1}^{1}$-formulas as follows 
\be\tag{$\Paai$}
(\forall^{\st}f^{1})\big[ (\exists g^{1})(\forall n^{0})(f(\overline{g}n)=0)\di (\exists^{\st}g^{1})(\forall n^{0})(f(\overline{g}n)=0)\big]
\ee
is the nonstandard counterpart of $\FIVE$.   It is an interesting exercise to show that if the antecedent of \eqref{nora} is $\paai$ (resp.\ $\Paai$), the antecedent of \eqref{nora6} is $(\mu^{2})$ (resp.\ $(\mu_{1})$), to be introduced in Section \ref{knowledge}.

\medskip

The following fragment of \emph{Standard Part} is the nonstandard counterpart of weak K\"onig's lemma (\cite{keisler1}):
\be\tag{$\STP$}
(\forall \alpha^{1}\leq1)(\exists^{\st}\beta^{1}\leq1)(\alpha\approx_{1}\beta),
\ee  
where $\alpha\approx_{1}\beta$ is $(\forall^{\st}n)(\alpha(n)=_{0}\beta(n))$.  Note that $\STP$ expresses the \emph{nonstandard compactness of $2^{\N}$} as in \emph{Robinson's theorem} (\cite{loeb1}*{p.\ 42}),  
The following fragment of \emph{Standard Part} is the nonstandard counterpart of weak weak K\"onig's lemma (\cite{pimpson}).
We reserve the variable `$T^{1}$' for trees and `$T^{1}\leq1$' means that $T$ is a binary tree.  
\be\tag{$\LMP$}
(\forall T \leq1)\big[ \mu(T)\gg0\di (\exists^{\st} \beta\leq1)(\forall^{\st} m)(\overline{\beta}m\in T) \big],
\ee
where `$\mu(T)\gg 0$' is just the formula $(\exists^{\st} k^{0})(\forall^{\st} n^{0})\Big(\frac{\{\sigma \in T: |\sigma|=n    \}}{2^{n}}\geq_{0} \frac{1}{2^{k}}\Big)$.  

\medskip

Note that there is no deep philosophical meaning to be found in the words `nonstandard counterpart':  this is just what the principles $\STP$, $\LMP$, $\paai$, and $\Paai$ are called in the literature (\cite{pimpson, sambon, keisler1}).  The following theorems from \cite{dagsam} provide normal forms for $\STP$ and $\LMP$ and establish the latter's relationships with the special and weak fan functionals.  In particular, the latter emerge from $\STP$ and $\LMP$ when applying Theorem \ref{consresultcor}.  
Recall the `finite sequence' notations from Notation \ref{skim}.
\begin{thm}\label{lapdog}
In $\P_{0}$, $\STP$ is equivalent to the following:
\begin{align}\label{frukkklk}
(\forall^{\st}g^{2})(\exists^{\st}w^{1^{*}}\leq 1, k^{0})(\forall T^{1}\leq1)\big[ & (\forall \alpha^{1} \in w)(\overline{\alpha}g(\alpha)\not\in T)\\
&\di(\forall \beta\leq1)(\exists i\leq k)(\overline{\beta}i\not\in T) \big], \notag
\end{align}  
and is equivalent to $(\forall^{\st}G^{2})(\exists^{\st}w^{1^{*}})(\forall f^{1}\leq{1})(\exists g\in w)({f}\in [\overline{g}G(g)])$, and to:
\begin{align}\label{fanns}
(\forall T^{1}\leq1)\big[(\forall^{\st}n^{0})(\exists \beta^{0^{*}})&(|\beta|=n \wedge \beta\in T ) \di (\exists^{\st}\alpha^{1}\leq1)(\forall^{\st}n^{0})(\overline{\alpha}n\in T)   \big].
\end{align}
Furthermore, $\P_{0}$ proves $(\exists^{\st}\Theta)\SFF(\Theta)\di \STP$.
\end{thm}
\begin{proof}
All results are established in \cite{dagsam}, except the following equivalence: 
\be\label{kuuit}
\STP\asa(\forall^{\st}G^{2})(\exists^{\st}w^{1^{*}})(\forall f^{1}\leq{1})(\exists g\in w)({f}\in [\overline{g}G(g)]). 
\ee 
To establish \eqref{kuuit}, use $\HAC_{\INT}$ to establish that $(\exists^{\st} g\leq1)(\forall^{\st}k^{0})(\overline{f}k=_{0}\overline{g}k)$ is equivalent to  
$(\forall^{\st}G^{2})(\exists^{\st} g\leq1)(\overline{f}G(g)=_{0^{*}}\overline{g}G(g))$ (by considering the negations of the latter two formulas).  Now prepend `$(\forall f^{1}\leq1)$' to the latter formula and use \emph{Idealisation} to pull the `$(\exists^{\st}g\leq1)$' to the front as in \eqref{kuuit}.  
\end{proof}
By \eqref{fanns} in the theorem, $\STP$ is just $\WKL^{\st}$ with the leading `st' dropped; this observation explains why $\STP$ deserves the monicker `nonstandard counterpart of $\WKL$'.  
The following theorem follows in the same way. 
\begin{thm}\label{lapdoc}
In $\P_{0}$, the principle $\LMP$ is equivalent to:
\[
(\forall^{\st}g^{2},k^{0})(\exists^{\st}w^{1^{*}}\leq 1, n^{0})
\textstyle(\forall T\leq1)\big[(\forall \alpha\in w)(\overline{\alpha}g(\alpha)\not\in T)\di \frac{|\{\sigma \in T: |\sigma|=n    \}|}{2^{n}}\leq\frac{1}{2^{k}}\big].
\]  
Furthermore, $\P_{0}$ proves $(\exists^{\st} \Lambda)\WFF(\Lambda)\di \LMP$.  
\end{thm}
Despite $\STP$ and $\LMP$ being the nonstandard counterparts of $\WKL$ and $\WWKL$, the former behaves \emph{quite} differently from the latter (and \eqref{linord}) in that the former does not follow from $\paai$ or $\Paai$, i.e.\ the nonstandard counterparts of $\ACA_{0}$ and $\FIVE$, as discussed in Section \ref{knowledgebase}.    

\medskip

Finally, we discuss the exact connection between our systems of Nonstandard Analysis and Computability theory provided by Theorem \ref{consresultcor}.  
The crucial point here is that in the syntactic theory of Nonstandard Analysis, the usual quantifiers $\exists$ and $\forall$ play the role of `uniform quantifiers' (see \cite{uhberger}) which are \emph{ignored} by the functional interpretation $S_{\st}$ used in the proof of Theorem \ref{consresultcor}, while the standard quantifiers $\exists^{\st}$ and $\forall^{\st}$ are given computational meaning.  
Indeed, the functional interpretation $S_{\st}$ applied to the proof of \eqref{bog} yields a term $t(\underline{x})$ in which the $(\forall^{\st}\underline{x})$ quantifier in \eqref{bog} describes the input variables, while the $(\exists^{\st}\underline{y})$ quantifier describes the output variables. 
This gives each of the nonstandard axioms a clear computational meaning entirely independent of Nonstandard Analysis per se, which may be of comfort to some who find Nonstandard Analysis alien. 
Those interested in this kind of development should consult \cite{SB}.

\subsection{Known results in Computability Theory}\label{knowledge}
A substantial number of results regarding the special and weak fan functionals were obtained in \cite{dagsam, dagsamIII, samflo}, some of which we list in this section as they are needed below or give rise to open questions.  We recall an oft-made observation regarding $\WWKL_{0}$ and the `Big Five' of RM,  namely that these six systems satisfy the strict implications:
\be\label{linord}
\FIVE\di \ATR_{0}\di \ACA_{0}\di\WKL_{0}\di \WWKL_{0}\di \RCA_{0}.
\ee
We mention \eqref{linord} as our results show that the situation is quite different in a higher-order or nonstandard setting.  
More results of this nature are in \cite{dagsamIII, dagsamV, dagsamVI}.

\medskip

First of all, it turns out that the fan functionals $\Theta$ and $\Lambda$ are hard to compute. 
\begin{thm}\label{import2}
Let $\varphi^{2}$ be a type two functional.  
There is no functional $\Theta^3$ as in $\SFF(\Theta)$ and no functional $\Lambda^3$ as in $\WFF(\Lambda)$ computable in $\varphi$.  
\end{thm}
\begin{proof}
Immediate from \cite{dagsam}*{Cor.\ 3.8 and Theorem 3.14}.
\end{proof}
We now list some well-known type two functionals which will be needed below.  
\emph{Feferman's search operator} as in $(\mu^{2})$ (see e.g.\ \cite{avi2}*{\S8}) is equivalent to $(\exists^{2})$ over Kohlenbach's system $\RCAo$ by \cite{kooltje}*{\S3}:
\be\tag{$\mu^{2}$}
(\exists \mu^{2})\big[(\forall f^{1})\big( (\exists n^{0})(f(n)=0)\di f(\mu(f))=0   \big)  \big],
\ee
\be\tag{$\exists^{2}$}
(\exists \varphi^{2})\big[(\forall f^{1})\big( (\exists n^{0})(f(n)=0)\asa \varphi(f)=0   \big)  \big],
\ee
Furthermore, $\ACAo\equiv \RCAo+(\mu^{2})$ is a $\Pi_{2}^{1}$-conservative extension of $\ACA_{0}$ (\cite{yamayamaharehare}*{Theorem 2.2}). 
The \emph{Suslin functional} $(S^{2})$ and the related $(\mu_{1})$ (see \cite{avi2}*{\S8.4.1}, \cite{kohlenbach2}*{\S1}, and \cite{yamayamaharehare}*{\S3}) give rise to $\FIVE$: 
\be\tag{$\mu_{1}$}
(\exists \mu_{1}^{1\di 1})(\forall f^{1})\big[  (\exists g^{1})(\forall x^{0})(f(\overline{g}x)=0)\di (\forall x^{0})(f(\overline{\mu_{1}(f)}x)=0)  \big].
\ee
\be\tag{$S^{2}$}
(\exists S^{2})(\forall f^{1})\big[  (\exists g^{1})(\forall n^{0})(f(\overline{g}n)=0)\asa S(f)=0  \big].
\ee
In fact, $\FIVEo\equiv \RCAo+(\mu_{1})$ is a $\Pi_{3}^{1}$-conservative extension of $\FIVE$ (\cite{yamayamaharehare}*{Theorem 2.2}).
We let $\SU(S)$ and $\MUO(\mu_{1})$ be $(S^{3})$ and $(\mu_{1})$ without the leading existential quantifiers.  
Similarly, we introduce $\FIVEK^{\omega}\equiv \RCAo+(S_{k}^{2})$, where $(S_{k}^{2})$ states the existence of a type two function $S_{k}^{2}$ which decides $\Pi_{k}^{1}$-formulas; note that $S_{1}$ is the Suslin functional.    
The higher-order version of second-order arithmetic $\textsf{Z}_{2}$ is $\textsf{Z}_{2}^{\Omega}\equiv\RCAo+(\exists^{3})$, where the latter is  
\be\tag{$\exists^3$}\label{hah}
(\exists \xi^{3})(\forall Y^{2})\big[  (\exists f^{1})(Y(f)=0)\asa \xi(Y)=0  \big].
\ee
Note that $\textsf{Z}_{2}^{\Omega}$ and $\textsf{Z}_{2}$ prove the same sentences by \cite{hunterphd}*{\S2}.
We reserve `$\exists^3$' for the unique functional $\xi^{3}$ from $(\exists^3)$.  
We do the same for other functionals, like $\mu^{2}, \mu_{1}, S^{2}, \dots$ introduced above.  
\begin{thm}\label{import3}
A functional $\Theta^{3}$ as in $\SFF(\Theta)$ can be computed from $\exists^3$.  
\end{thm}
\begin{proof}
Immediate from \cite{dagsam}*{Theorem 3.9}.
\end{proof}
By the following theorem, the exotic properties of $\Theta$ are not due to its high type.  
As discussed in \cite{dagsamIII}, $\HBU$ is essentially \emph{Cousin's lemma}, dating as far back as 1882.  
\begin{thm}\label{nolapdog}
 $\ACAo+\QFAC^{2,1}$ proves $(\exists \Theta)\SFF(\Theta)\asa \HBU$; the latter is
\be\label{zosimpelistnie}\tag{$\HBU$}\textstyle
(\forall \Psi^{2}:\R\di \R^{+})(\exists w^{1^{*}}){(\forall x^{1}\in [0,1])(\exists y\in w)(x\in I_{y}^{\Psi}   )}, 
\ee  
where $I_{y}^{\Psi}$ is $(y-\Psi(y), y+\Psi(y))$.
No system $\FIVEK^{\omega}$ proves $\HBU$.  
\end{thm}
\begin{proof}
Immediate from \cite{dagsamIII}*{Theorems 3.1 and 3.3}.
\end{proof}
A similar result can be obtained for $\Lambda$: the existence of the latter is equivalent 
to the fact that a finite sub-cover exists for any open cover of the Martin-L\"of random reals in Cantor space \emph{minus} some $U_{k}$, where the latter is 
the $k$-th set in the universal Martin-L\"of test (see \cite{samcie18}).  This result originates from the RM of $\WWKL$ as in \cite{avi1337}.  

\medskip

Theorem \ref{nolapdog} already deals a significant blow to the elegant picture in \eqref{linord}, but $\HBU$ can even collapse part of the latter linear order, namely as in Theorem~\ref{frigjr}.  
Now, $\ATR_{0}$ is $\ACA_{0}$ plus \emph{arithmetical transfinite recursion} as follows:
\be\tag{$\ATR_{\theta}$}
(\forall X^{1})\big[\WO(X)\di (\exists Y^{1})H_{\theta}(X, Y) \big], 
\ee
for any arithmetical $\theta$.  Here, $\WO(X)$ expresses that $X$ is a countable well-ordering and $H_{\theta}(X, Y)$ expresses that $Y$ is the result from iterating $\theta$ along $X$.  
Details and definitions may be found in \cite{simpson2}*{V.2}.  
For Theorem \ref{frigjr}, we need the following `trivially uniform' version of $\ATR_{0}$:
\be\tag{$\UATR$}
(\exists \Phi^{1\di 1})(\forall X^{1}, f^{1})\big[\WO(X)\di H_{f}(X, \Phi(X,f)) \big], 
\ee
where $H_{f}(X, Y)$ is just $H_{\theta}(X, Y)$ with $\theta(n, Z)$ defined as $(\exists m^{0})(f(n,m, \overline{Z}m)=0)$.
Note that the base theory in the following theorem is conservative over $\WKL_{0}$.
\begin{thm}\label{frigjr}
The system $\RCAo+\HBU+\QFAC^{2,1}$ proves $(\mu^{2})\asa \UATR$.  
\end{thm}
\begin{proof}
Immediate from \cite{dagsam}*{Cor.\ 6.7} and \cite{dagsamIII}*{Theorem 3.3}.
\end{proof}

The previous theorem is based on an effective result where $\Phi$ as in $\UATR_{0}$ is defined from $\Theta$ and $\mu^{2}$ via a term of G\"odel's $T$.  
This effective result in turn derives from Theorem \ref{mikeh}, i.e.\ via term extraction applied to Nonstandard Analysis.
\begin{thm}\label{frigjr22121}
$\RCAo+(\exists \Theta)\SFF(\Theta)$ is a conservative extension of $\RCA_{0}^{2}+\WKL$.
\end{thm}
\begin{proof}
Immediate from \cite{dagsam}*{Cor.\ 3.5}.
\end{proof}

Combining Theorems \ref{nolapdog} and \ref{frigjr}, it would seem that $\Theta$ produces non-hyper-arithmetical outputs, which turns out to be correct.  
By contrast, there are weak instances of $\Lambda$ which are `closed on the hyperarithmetical'.
\begin{thm}\label{cor.alt.6.8}
For any $\Theta$ such that $\SFF(\Theta)$, there is hyperarithmetical $G^{2}$ such that $\Theta(G)$ is not hyperarithmetical.  
\end{thm}  
\begin{proof}
Immediate from \cite{dagsam}*{Theorem 5.1}.
\end{proof}
\begin{thm}\label{cor.5.6}
There is a $\Lambda_{0}$ such that $\WFF(\Lambda_{0})$ and such that for any total, hyperarithmetical $G^{2}$, $\Lambda_{0}(G,k)$ is a finite list of hyperarithmetical functions.
\end{thm}
\begin{proof}
Immediate from \cite{dagsam}*{Cor.\ 5.14}.
\end{proof}
As noted in \cite{dagsam}*{Appendix}, and proved in full in Section \ref{essdag}, 
one can define a weak fan functional that is very weak as follows:
\begin{thm}\label{import24}
There exists a functional $\Lambda_{1}$ satisfying $\WFF(\Lambda_{1})$ such that all functions computable in $\Lambda_1$ and $\exists^2$ are hyperarithmetical.  
\end{thm}
\begin{proof}
The proof is given in Section \ref{essdag}.  See Theorem \ref{theorem.lambda}.
\end{proof}
\begin{cor}
There exists a functional $\Lambda_{1}$ satisfying $\WFF(\Lambda_{1})$ such that no $\Theta$ satisfying $\SFF(\Theta)$ is computable in $\Lambda_{1}$ and $\exists^{2}$.
\end{cor}
\begin{proof}
Theorems \ref{cor.alt.6.8} and \ref{import24} immediately yield the corollary.  
\end{proof}
Finally, Theorem \ref{frigjr22121} is proved using the $\ECF$\emph{-translation}, which will be needed below.  
We therefore discuss the proof of the former theorem in some detail.
\begin{rem}[$\ECF$-translation and $\Theta$]\label{ECFREM}\rm
As discussed in \cite{kohlenbach2}*{\S3}, one can modify the proofs in \cite{troelstra1}*{\S2.6} to establish that $\RCAo+(\exists \Omega^{3})\MUC(\Omega)$ is conservative over $\RCA_{0}^{2}+\WKL$, where $\Omega^{3}$ is called the \emph{intuitionistic fan functional} as follows:
\be\tag{$\MUC(\Omega)$}
(\forall Y^{2}) (\forall f^{1}, g^{1}\leq1)(\overline{f}\Omega(Y)=\overline{g}\Omega(Y)\notag \di Y(f)=Y(g)),   
\ee
In the latter reference, the so-called $\ECF$-interpretation is defined which, intuitively speaking, replaces all higher-order functionals (of type two or higher) by type one \emph{codes} (in the sense of Reverse Mathematics) which \emph{represent}  (automatically continuous) higher-type functionals.  
The $\ECF$-interpretation has 
the following convenient property (discussed in \cite{kohlenbach2}*{\S3}) for any formula in the language of finite types:
\be\label{TUG}
\textup{If $\RCAo\vdash A$, then $\RCA_{0}^{2}\vdash [A]_{\ECF}$}.
\ee
Now, the $\ECF$-interpretation of $(\exists \Omega^{3})\MUC(\Omega)$ expresses that there is a code $\alpha^{1}$ which yields a modulus of uniform continuity on Cantor space on input a code $\beta^{1}$ representing an (automatically continuous) type two functional.  
As follows from the discussion in \cite{longmann}*{p.\ 459}, we have $[(\exists \Omega^{3})\MUC(\Omega)]_{\ECF}\asa \WKL$.  
Alternatively, one can explicitly define the aforementioned code $\alpha^{1}$ and show that it has the required properties using (the contraposition of) $\WKL$, as done in \cite{troelstra1}*{2.6.6} and \cite{noortje}*{p.\ 101}.  

\medskip

Theorem \ref{frigjr22121} can now be obtained in at least two ways:  First of all, one considers $(\exists \Omega)\MUC(\Omega)\di (\exists \Theta)\SFF(\Theta)\di \WKL$ (provable in $\RCAo$), which follows from the results in \cite{dagsam}*{\S3} or \cite{samGH}*{\S3}, and applying the $\ECF$-translation and the above results yields $\WKL\di [(\exists \Theta)\SFF(\Theta)]_{\ECF}\di \WKL$.  Secondly, one can also explicitly define the code for $\Theta$ required for $[(\exists \Theta)\SFF(\Theta)]_{\ECF}$ in terms of the aforementioned code $\alpha^{1}$, as the classical fan functional trivially computes $\Theta(G)$ in case $G^{2}$ is continuous on Cantor space.  This finishes the proof of Theorem \ref{frigjr22121}.    
\end{rem}

\subsection{Known results in Nonstandard Analysis}\label{knowledgebase}
A substantial number of results regarding nonstandard compactness were obtained in \cite{dagsam}, some of which we list in this section 
as they are needed below or give rise to open questions.

\medskip

First of all,  although the Big Five and $\WWKL_{0}$ are linearly ordered as in \eqref{linord}, the nonstandard counterparts behave quite differently.  
\begin{thm}\label{forqu}
$\P+\Paai$ and $\P+\paai$ do not prove $\STP$ or $\LMP$.  
\end{thm}
\begin{proof}
Immediate from \cite{dagsam}*{Cor.\ 4.6}.
\end{proof}
Secondly, in light of the failure of $\paai\di \STP$, it is a natural question how strong the combination $\paai+\STP$ is.  
As it turns out, we readily obtain $\ATR^{\st}$ from $\paai+\STP$.    The same theorem for $\LMP$ fails.
\begin{thm}\label{mikeh}
The system $\P_{0}+\paai+\STP$ proves $\ATR^{\st}_{0}$ while $\P+\paai+\LMP$ does not.
\end{thm}
\begin{proof}
Immediate from \cite{dagsam}*{Theorems 6.3 and 6.4}.
\end{proof}
Note that $\WKL$ and $\WWKL$ (and hence $\STP$ and $\LMP$) are `very close' in the sense that there is nothing between 
them in the RM zoo (\cite{damirzoo}) or the Weihrauch degrees (\cite{bratger}).  
\begin{thm}\label{mikeh2}
The system $\P+\paai+\LMP$ does not prove $\STP$.
\end{thm}  
\begin{proof}
Immediate from Theorem \ref{mikeh}.
\end{proof}
Finally, we often use this theorem without mention.  
\begin{thm}\label{XXX}
If $\RCA_{0}$ proves $A$, then $\P_{0}$ proves $A^{\st}$.
\end{thm}
\begin{proof}
One readily verifies that $\P_{0}$ proves the axioms of $\RCA_{0}$ relative to `st'.
\end{proof}

\subsection{Open questions}\label{qopen}
The above listed theorems from \cite{dagsam} give rise to the following open questions. They will be answered in this paper.

\medskip

First of all, in light of Theorem \ref{frigjr}, it is a natural question how strong $\Theta+\mu^{2}$ is compared to well-known functionals.  
We show in Section \ref{prelim2} that $S^{2}$ is not computable from $\Theta+\mu^{2}$.  
In Section \ref{atrsec}, we also provide a direct proof (not involving Nonstandard Analysis) of the fact that $\Theta+\mu^{2}$ 
computes a realiser for $\ATR_{0}$.  

\medskip

Secondly, in light of Theorem \ref{forqu}, it is a natural question `how high' $\Paai+\STP$ actually goes.  
We show in Section \ref{SIXTUS} that the latter combination exists at the level of $\SIX$, i.e.\ strictly stronger than $\FIVE$ and $\Paai$.  
As a result, $\FIVE^{\omega}+\QFAC^{2,1}+\HBU$ proves the $\Pi_{3}^{1}$-consequences of $\SIX$. 

\medskip

Thirdly, in light of Theorem \ref{cor.5.6} and \ref{mikeh}, it is a natural question whether weak fan functionals carry non-trivial strength.  
The answer is negative, in the following sense:  we will identify a weak fan functional $\Lambda_{1}$ and show in Section~\ref{essdag} that $\Lambda_{1}+\mu^{2}$ computes the same objects as $\mu^{2}$.  
This shows that we cannot in general compute a special fan functional from a weak one. This provides mathematical evidence for the intuition that compactness up to measure is strictly weaker than full compactness.

\medskip

Fourth, in light of Theorem \ref{mikeh}, it is a natural question whether our results somehow generalise to Schweber's generalisation of $\ATR_{0}$ in third-order arithmetic \cites{schtreber,schtreberphd}.   We obtain such a generalisation for Theorem \ref{mikeh} in Section \ref{schweber}.

\subsection{Equivalent definitions}\label{peqingduck}
We show that the definition of the special and weak fan functionals from Section \ref{weakopi} is equivalent to the original definition from \cite{samGH}.

\medskip

The following definition for special fan functionals was used in \cite{samGH}.
We reserve the variable `$T^{1}$' for trees and denote by `$T^{1}\leq1$' that $T$ is a binary tree.  
\bdefi\label{dodierorg}
The formula $\SCF(\nu)$ is as follows for $\nu^{(2\di (0\times 1^{*}))}$:
\[
(\forall g^{2}, T^{1}\leq1)\big[(\forall \alpha \in \nu(g)(2))  (\overline{\alpha}g(\alpha)\not\in T)
\di(\forall \beta\leq1)(\exists i\leq \nu(g)(1))(\overline{\beta}i\not\in T) \big]. 
\]
\edefi
The provenance of the name of the specification `$\SFF(\Theta)$' for the \underline{s}pecial \underline{f}an \underline{f}unctional is obvious. 
Similarly, $\SCF(\eta)$ was initially (and incorrectly) believed to be a special case of the (classical) fan functional, explaining its name. 
We now  have the following theorem.
\begin{thm}\label{2.23}
There are terms $s, t$ of G\"odel's $T$ of lowest level such that 
\be\label{amaaaai}
(\forall \Theta)(\SFF(\Theta)\di \SCF(t(\Theta)))\wedge (\forall \nu)(\SCF(\nu)\di \SFF(s(\nu))).
\ee
\end{thm}
\begin{proof}
We first provide a proof based on Computability Theory.   Define $s(\nu) := \lambda g.\nu(g)(2)$ and define $t(\Theta) := \lambda g.(\max\{g(\alpha) \mid \alpha \in \Theta(g)\} + 1, \Theta(g)).$
Assume $\SCF(\nu)$ and for given $g$ consider $\nu(g) = (n,\{\alpha_1 , \ldots , \alpha_k\})$. If $\SFF(s(\nu))$ fails for $g$, there is a $\beta \leq 1$ that is not in any $[\overline{\alpha_j}g(\alpha_j)]$. Let $T$ be the tree of all sequences $\bar \beta m$. Then the antecedent in $\SCF(\nu)$ holds for $g$ and this $T$, but not the conclusion.

\medskip

Now assume $\SFF(\Theta)$ and let $g$ be given. We have that $\Theta(g) = \{\alpha_1 , \ldots , \alpha_k\}$ where $C = [\overline{ \alpha_1}g(\alpha_1)] \cup \cdots \cup  [\bar \alpha_kg(\alpha_k)]$. We must prove $\SCF(t(\Theta))$. Again we argue by contradiction. Let $T$ be a binary tree such that there is a $\beta$ with $\overline \beta i \in T$, 
where $i = (\max\{g(\alpha) \mid \alpha \in \Theta(g)\}+1$, i.e.\ the conclusion in $SCF(t(\Theta))$ fails for this $T$.
Then $\beta \in [\bar \alpha_jg(\alpha_j)]$ for some $1 \leq j \leq k$ and $\bar \alpha_j g(\alpha_j)$ will be a sub-sequence of $\bar \beta i$. Thus the assumption in $\SCF(t(\Theta))$ does not hold for this $T$ either.

\medskip

We also provide a proof based on Nonstandard Analysis.  Following Theorem~\ref{lapdog}, $\P_{0}$ proves that the normal form \eqref{frukkklk} is equivalent to the normal form 
\be\label{obvious}
(\forall^{\st}G^{2})(\exists^{\st}w^{1^{*}})(\forall f^{1}\leq{1})(\exists g\in w)({f}\in [\overline{g}G(g)]).  
\ee
Since standard functionals provide standard output for standard input, 
$(\exists^{\st}\Theta)\SFF(\Theta)$ implies \eqref{obvious}.  
Hence, $\P_{0}$ also proves the following:
\be\label{lapjes}
(\forall^{\st}\Theta)[ \SFF(\Theta) \di \eqref{frukkklk}].
\ee
Now bring outside the standard quantifiers in the consequent of \eqref{lapjes} and apply term extraction as in Corollary \ref{consresultcor2} to obtain the first conjunct of \eqref{amaaaai}.
The second conjunct follows in the same way. 
\end{proof}

We now discuss the definition of the weak fan functionals similar to Definition~\ref{dodierorg}.
We first introduce \emph{weak weak K\"onig's lemma}.
\bdefi[Weak weak K\"onig's lemma]\label{leipi}~
\begin{enumerate}
 \renewcommand{\theenumi}{\roman{enumi}}
\item For $T\leq1$, define $L_{n}(T):=\frac{|\{\sigma \in T: |\sigma|=n    \}|}{2^{n}}$.
\item For $T\leq1$, define\footnote{Note that a statement of the form `$\lim_{n\di\infty}a_{n}>_{\R}b$' always makes sense as a formula of second-order arithmetic, namely $(\exists N^{0})(\exists k^0)(\forall n^{0}\geq N)(a_{n}>_{\R}b+\frac{1}{2^{k}})$, even if limit at hand cannot be proved to exist in a weak system, like the base theory $\RCA_{0}$.} `$\mu(T)>_{\R}0$' as `$\lim_{n\di\infty} L_{n}(T)>_{\R}0$'. 
\item We define $\WWKL$ as $(\forall T \leq1)\big[ \mu(T)>_{\R}0\di (\exists \beta\leq1)(\forall m)(\overline{\beta}m\in T) \big]$.
\end{enumerate}     
\edefi
As noted right after Definition \ref{dodier}, special fan functionals intuitively provide a finite sub-cover on input an uncountable cover of $2^{\N}$.  
Similarly, weak fan functionals provide an enumerated set of neighbourhoods covering a set of measure one.  Again similar to the special ones, the weak fan functionals originate from a weak version of the nonstandard compactness of Cantor space, as discussed in Section \ref{pampson}.   
\bdefi \label{fadier}
The formula $\WCF(\eta)$ is as follows for $\eta^{(2\di (1\times 1^{*}))}$:
\[\textstyle
(\forall k^{0},g^{2}, T^{1}\leq1)\big[(\forall \alpha \in \eta(g,k)(2))  (\overline{\alpha}g(\alpha)\not\in T)
\di   L_{\eta(g,k)(1)}(T)\leq\frac{1}{{2^{k}}}\big].
\]
\edefi
In contrast to $\nu$, $\eta$ only outputs (via the function $\lambda k.\eta(g, k)(1)$) a modulus for $\mu(T)=0$ rather than a finite upper bound for $T$.  
The antecedent in the definition of $\eta$ is similar to that of $\nu$:  a finite sequence of paths not in $T$ is provided (via $\eta(g, k)(2)$).  
Thus, there is a trivial term of G\"odel's $T$ computing $\eta$ in terms of $\nu$.  

\medskip

Similar to Theorem \ref{amaaaai}, we have the following equivalence. 
\begin{thm}\label{makker}
There are terms $s, t$ of G\"odel's $T$ of lowest level such that 
\be\label{amaaaai1}
(\forall \Lambda)(\WFF(\Lambda)\di \WCF(t(\Lambda)))\wedge (\forall \eta)(\WCF(\eta)\di \SFF(s(\eta))).
\ee
\end{thm}
\noindent
The first proof of Theorem \ref{2.23} is easily adjusted to a proof of Theorem \ref{makker}.

\section{Uniform computability for $\Theta$, $\Lambda$, and $\mu^{2}$}\label{dagsec}
In this section, we investigate uniform Kleene-computability for respectively  special and weak fan functionals $\Theta$ and $\Lambda$, combined with Feferman's $\mu$. 
In Section~\ref{prelim1} we discuss some preliminary results and notation.  
In Section~\ref{prelim2}, we show that \emph{only} hyperarithmetical functions can be \emph{uniformly} computed by $\Theta$ and $\mu$; as a result, the latter combination does not compute the Suslin functional.  In Section~\ref{atrsec}, we provide a \emph{direct} proof that $\ATR_{0}$ can be obtained from $\Theta$ and $\mu^{2}$, which was established \emph{indirectly} (using term extraction from Nonstandard Analysis) in \cite{dagsam}*{\S6}.  Thus, the combination $\Theta$ plus $\mu^{2}$ can compute \emph{non-hyperarithmetical} functions, \emph{but only non-uniformly}.    
By contrast,  in Section~\ref{essdag}, we construct $\Lambda_{1}$, a weak fan functional such that only hyperarithmetical functions are computable in $\Lambda_{1}$ and $\mu$.  As a consequence, special fan functionals are in general not computable from a weak fan functional $\Lambda$ combined with $\mu$.

\subsection{Preliminaries}\label{prelim1}
\subsubsection{Introduction}
In this section, we introduce the Kleene schemes S1-S9 and consider  some minor modifications due to the need for notational simplicity.  We are primarily interested in the computational power of special fan functionals $\Theta$ or weak fan functionals $\Lambda$, in conjunction with Feferman's $\mu$.
We establish our results with respect to  full Kleene computability.  
For this, it does not matter if we consider Kleene's $\exists^2$ or Feferman's $\mu$, but in case  we restrict ourselves to primitive recursion, $\mu$ is no longer computable in $\exists^2$. Thus, for studying the computational power of sub-classes of S1-S9 like fragments of G\"odel's $T$, it is better to use $\mu$. 

\medskip

 For the reader unacquainted with (higher-order) computability theory, we point to some well-known facts that we will use without further reference:
 \begin{enumerate}
   \renewcommand{\theenumi}{\roman{enumi}}
 \item For subsets of $\N$ or $\N^\N$, the \emph{hyperarithmetical} sets are exactly those computable in $\mu$, or equivalently in $\exists^2$, and exactly the $\Delta^1_1$-sets.
 \item The $\Pi^1_1$-sets are exactly the sets \emph{semi-computable} in $\mu$ (or  $\exists^2$), i.e.\ the domains of functions partially computable in $\mu$.
 \item The ordinal $\omega_1^{\textup{\textsf{CK}}}$ (`\textsf{CK}' for \emph{Church-Kleene}) is the least ordinal without a computable code. G\"odel's  $L_{\omega_1^{\textup{\textsf{CK}}}}$, the fragment of the universe of the constructible sets up to $\omega_1^{\textup{\textsf{CK}}}$, is the least $\Sigma_1$-admissible structure\footnote{\label{admi}A structure is $\Sigma_1$-admissible if it satisfies the Kripke-Platek axioms $\Delta_1$-comprehension and $\Sigma_1$-replacement. We say that an ordinal $\alpha$ is admissible if the corresponding fragment of L is admissible.}. 
 \end{enumerate}

\subsubsection{The functionals $\Theta$ and $\Lambda$}\label{3.1.2}
We will investigate uniform Kleene-computability for respectively $\Theta$ and $\Lambda$ combined with $\mu$.  We now provide suitable alternative definitions of these fan functionals to be used below. 

\medskip

According to the specification $\SFF(\Theta)$, $\Theta$ is a functional of type $2 \rightarrow  1^*$ where for each $F$, the set of neighbourhoods $C_{\bar g(F(g))}$, with $g \in \Theta(F)$, is a cover of the Cantor space. For adjustment to the Kleene schemes, it is better to use an alternative presentation, coding a finite sequence from $C$ into one, as follows.

\medskip

In this section, we let $\Theta(F)$ be an element of Cantor space that is not constant zero. Each such object $f$ will code a finite sequence $\langle g_1 , \ldots , g_k\rangle$ of binary functions by letting $k$ be the least positive number such that $f(k - 1) = 1$, and then decode $g(n) = f(n+k)$ into $k$ elements using the standard $k$-partition of $\N$, i.e. $g_i(m) = g(m \cdot k + i - 1)$. 
When $s$ is a finite binary sequence, we also use $C_s$ to denote the corresponding basic neighbourhood in $C$, essentially meaning the same as the formal expression $[s]$.
We will write $\Theta(F) = \langle g_1 , \ldots , g_k\rangle$ and we will assume that $\Theta$ satisfies that for all $F$, $\{C_{\bar g_i(F(g_i))} \mid i = 1 , \ldots , k\}$ is a cover of Cantor space. The latter is equivalent to stating that for some $n \in \N$ and for all binary sequences $s$ of length $n$ there is some $i$ such that $\bar g_i(F(g_i)) $ is an initial segment of $ s$.

\medskip

Similarly, according to the specification $\WFF(\Lambda)$,
$\Lambda(F)(k)$ is a finite sequence $\langle f_1 , \ldots , f_n\rangle$ from Cantor space such that $\m(\bigcup_{i = 1}^n C_{\bar f_i(F(f_i))}) \geq 1 - \frac{1}{2^{k}},$
where $\m$ denotes the standard product measure on Cantor space $C$. 
When studying aspects of computability relative to $\Lambda$ and $\mu$, we may equivalently let $\Lambda(F)$ be a sequence $(f) = (f_i)_{i \in \N}$ such that
$\m(\bigcup_{i \in \N}C_{\bar f_i(F(f_i))}) = 1.$
For notational reasons, this is the form for $\Lambda$ we will use in this section. 
\subsubsection{The Kleene Schemes}
Turing's famous model of computability (\cite{tur37}) is restricted to inputs of types zero and oracles of type one.  
By way of generalisation, Kleene introduces computations taking sequences $\vec \Phi$ of higher order functionals  $ \Phi$ of pure types as arguments (\cite{kleene}). In particular, via the schemes S1-S9, that are clauses in a grand  monotone inductive definition, he defined the relation $\{e\}(\vec \Phi) = a$, i.e.\ the $e$-th (Kleene) computation with input $\vec{\Phi}$ terminates with output $a\in \N$.

\medskip

For the purpose of this section, we will introduce the Kleene schemes S1-S9 with some minor modifications, motivated by he following:
\begin{enumerate}  
  \renewcommand{\theenumi}{\roman{enumi}}
\item In all our computations, at most one functional of type 2 is used as an argument, namely Feferman's $\mu$.
\item The scheme S8 for functional application was originally designed for functionals of pure type.   However, special fan functionals are of mixed type $(\N^\N \rightarrow \N) \rightarrow (\N \rightarrow \N)$ while weak fan functionals are of type $(\N^\N \rightarrow \N) \rightarrow (\N \rightarrow (\N \rightarrow \N))$.
\end{enumerate}
Instead of coding $\Theta$ and $\Lambda$ as objects of pure type 3, we modify the schemes S1-S9 so that they make sense for the one argument $\mu$ of type 2 and for any functionals $\Theta$, and later $\Lambda$, of the relevant mixed types. 
The only motivation for this adjustment to mixed types is readability: we will let the special and weak fan functionals appear directly in the schemes, and not in coded form.
It is a matter of unpleasant routine to show that this modification yields the same notion of computation as Kleene's original schemes via the standard reductions to pure types. In \cite{longmann}*{Section 5.1.3}, Kleene's notion of computation is extended to all finite types via some form of  $\lambda$-calculus, but we prefer not to introduce the general machinery here.

\medskip

Assume that the functional $\Theta$ is of the specified type.
Let $\vec g$ be a sequence of functions and $\vec b$ be a sequence of numbers.
We now define the relation $\{e\}(\Theta,\mu,\vec g,\vec b) = a$ by induction as follows.
\begin{defi}[Modified Kleene S1-S9]\label{1.1}\hspace*{2mm}{\em 
\begin{itemize}
\item[(S1)]$\{\langle 1 \rangle\}(\Theta , \mu , \vec g, a , \vec b) = a+1$
\item[(S2)]$\{\langle 2,a \rangle\}(\Theta , \mu ,\vec g, \vec b) = a$
\item[(S3)] $\{\langle 3 \rangle\}(\Theta , \mu , \vec g, a , \vec b) = a$
\item[(S4)] If $e = \langle 4 , e_1 , e_2\rangle$ and for some $b$ we have that
\begin{itemize}
\item[(i)] $\{e_1\}(\Theta ,\mu , \vec g,\vec b) = b$
\item[(ii)] $\{e_2\}(\Theta , \mu, \vec g, b , \vec b) = a$
\end{itemize} then $\{e\}(\Theta , \mu , \vec g,\vec b) = a$
\item[(S5)] If $e = \langle 5,e_1,e_2\rangle$ then (with the obvious interpretation, in analogy with S4)
\begin{itemize}
\item[(i)] $\{e\}(\Theta , \mu , \vec g,0 , \vec b) = \{e_1\}(\Theta , \mu , \vec g , \vec b)$
\item[(ii)] $\{e\}(\Theta , \mu , \vec g , a+1 , \vec b) = \{e_2\}(\Theta , \mu , \vec g, a , \{e\}(\Theta , \mu , \vec g,a , \vec b),\vec b)$
\end{itemize}
\item[(S6)] Let $\vec g = (g_1 , \ldots , g_k)$, $\vec b = (b_1 , \ldots , b_m)$ and let $\tau_1, \tau_2$ be permutations of $\{1 , \ldots , k\}$ and $\{1 , \ldots , m\}$ respectively. If  $e = \langle 6 , e_1 , \tau_1,\tau_2\rangle$ then 
\[
\{e\}(\Theta,\mu,g_1 , \ldots , g_k, a_1 , \ldots , a_n) = \{e_1\}(\Theta , \mu , g_{\tau_1(1)}, \ldots , g_{\tau_1(k)}, b_{\tau_2(1)} , \ldots , b_{\tau_2(m)})
\]
\item[(S7)] $\{\langle 7 \rangle\}(\Theta , \mu , g,\vec g, b ,  \vec b) = g(b)$
\item[(S8.1)] If $e = \langle 8,1,e_1\rangle$ and $\{e_1\}(\Theta , \mu , \vec g, a , \vec b)$ is defined for all $a\in \N$ then
\begin{itemize}
\item[(i)] $\{e\}(\Theta , \mu , \vec g,\vec b) = 0$ if $\{e_1\}(\Theta , \mu , \vec g , a , \vec b) = 0$ for all $a$
\item[(ii)] $\{e\}(\Theta , \mu , \vec g,\vec b) = a$ for the least $a$ such that $\{e_1\}(\Theta , \mu , \vec g, a , \vec b) > 0$ otherwise
\end{itemize}
\item[(S8.2)] If $e = \langle 8,2,e_1 \rangle$, let $F(g) = \{e_1\}(\Theta , \mu , g,\vec g,\vec b)$. If $F$ is total, we let $\{e\}(\Theta,\mu,\vec g , a, \vec b) = \Theta(F)(a)$.
\item[(S9)] If $e = \langle 9,i,j\rangle$, $i \leq k$ and $j \leq m$, then 
\[
\{e\}(\Theta,\mu,g_1 , \ldots , g_k , d , b_1 , \ldots , b_m) = \{d\}(\Theta , \mu , g_1 , \ldots , g_i , b_1 , \ldots , b_j)
\]

\end{itemize}
If we leave out S9 in the previous definition, we have the schemes for Kleene primitive recursion.
Furthermore, a definition of the relation $\{e\}(\Lambda,\mu,\vec g , \vec b) = a$, where $\Lambda$ is a weak fan functional, is obtained by replacing $\Theta$ with $\Lambda$ everywhere in S1 - S7, S9 and S8.1, and replacing S8.2 with the following formula:
\begin{itemize}
\item[(S8.3)] If $e = \langle 8,3,e_1\rangle$, put $F(g) = \{e_1\}(\Lambda , \mu , g,\vec g,\vec b)$. If $F$ is total, define the value $\{e\}(\Lambda , \mu , \vec g , i,a ,\vec b) $ as $ \Lambda(F)(i)(a)$.
\end{itemize}}
\end{defi}
All these schemes are viewed as clauses in a strictly positive inductive definition. If we leave out S9, then the definition may be viewed as a recursion on $e$.  The set of indices, together with the relevant arities, can then be defined by standard primitive recursion over $\N$.  
Moreover, in this case all `computations' will terminate, as partiality is only introduced via S9.

\subsection{Uniform computability in $\Theta$}\label{prelim2}
In this section, we will introduce the notion of a \emph{$\Theta$-structure} (see Definition \ref{zetas}) and use the associated model theory to prove two crucial theorems (Theorems \ref{1.2} and \ref{1.3}) regarding computability in $\mu$ and $\Theta$. 
As a corollary, we obtain that $\Theta+\mu$ does not compute $S^{2}$. 
The proof in this section can be viewed as an elaboration on the proof of \cite{longmann}*{Theorem 5.2.25}.

\medskip

First of all, as to notation, recall that $\omega_1^f$ is the least ordinal not represented by any well-ordering Turing-computable in $f$ (see \cite{Sacks.high}*{X.2.9}).  Also, throughout this section, the quantifier `$\forall \Theta$' is to be understood as `for all special fan functionals $\Theta$', i.e.\ $(\forall \Theta)(\SFF(\Theta)\di \dots)$, which we omit for reasons of space.  
\begin{thm}\label{1.2} There is a special fan functional $\Theta$ such that for all functions $f$ computable in $\Theta$ and $\mu$ we have that
$\omega_1^f = \omega_1^{\rm CK}.$
\end{thm}
\begin{thm}\label{1.3}
The set $\{(e,\vec y,a) \mid \forall \Theta.\{e\}(\Theta , \mu , \vec y) = a\}$
is $\Pi^1_1$, where $\vec y$ ranges over all finite sequences of non-negative integers.
\end{thm}
The following corollary implies that $\Theta$ and $\mu$ cannot uniformly compute $S^{2}$.  
\begin{corollary} Let $f$ be a function such that for some $e$,
$\{e\}(\Theta , \mu , n) = f(n)$ for all $n$ and all special fan functionals  $\Theta$. Then $f$ is hyperarithmetical. 
\end{corollary}
We could, in Theorem \ref{1.3}, let $\vec y$ range over all sequences of objects of type zero and one, but we have not found any use for this observation.
The proof of Theorem \ref{1.3} will essentially be an application of the L\"owenheim-Skolem theorem, establishing the fact that the following statements are equivalent:
\begin{enumerate}
\item[(i)] For all $\Theta$, $\{e\}(\Theta , \mu , \vec g,\vec b) = a$.
\item[(ii)] For all countable models $\mathcal{M}$ containing $\vec g$ and a special fan functional $\Theta_{\mathcal{M}}$ (in the sense of the model as indicated), 
${\mathcal{M}} \models  \{e\}(\Theta_{\mathcal{M}} , \mu , \vec g,\vec b) = a$.
\end{enumerate}
We must, however, show some care in what we mean by `a model' and what we then mean by `${\mathcal{M}} \models  \{e\}(\Theta_{\mathcal{M}} , \mu , \vec g,\vec b) = a$'. 
For instance, we cannot use the usual inductive definition involved in Kleene computability directly, because the least fixed point of the Kleene schemes, even when restricted to a countable structure, is $\Pi^1_1$ itself. Moreover, the L\"owenheim-Skolem argument does not work for second-order concepts, so we need to replace Kleene's definition with something first-order.  It turns out that it suffices to consider \emph{all} fixed points of the Kleene schemes. 
Also, the proof of Theorem \ref{1.3} yields Theorem \ref{1.2} `almost for free'.

\medskip
\noindent
We introduce the notion of a $\Theta$-structure as follows.  
\begin{definition}\label{zetas}{\em A \emph{$\Theta$-structure} is a tuple ${\mathcal{M}} = \langle \N,M_1,M_2,\Theta_{\mathcal{M}},\mu,R\rangle$ such that
\begin{enumerate}
 \renewcommand{\theenumi}{\roman{enumi}}
\item $M_1$ is a set of functions $f:\N\di\N$  and $M_2$ is a set of functions $G: M_1\di \N$.
\item $\mu \in M_2$ satisfies the usual definition of $\mu$.
\item $\Theta_{\mathcal{M}}:M_2 \rightarrow \{0,1\}^\N$ and satisfies the modified $\SFF(\Theta)$ relative to $M_1, M_2$, see Section \ref{3.1.2}.
\item $R$ stands for a relation  $[e]_R(\Theta_{\mathcal{M}},\mu,\vec g,\vec b) = a,$ where $\vec g$ is a finite sequence from $M_1$ and $\vec b$ is a finite sequence from $\N$,  that satisfies:
\begin{enumerate}
\item For each $\vec g,\vec b$ there is at most one $a$ such that $[e]_R(\Theta_{\mathcal{M}},\mu,\vec g , \vec b) = a$.
\item If for some $\vec g,\vec b,e$, we have $(\forall b \in \N)(\exists a)([e]_R(\Theta_{\mathcal{M}},\mu, \vec g , b,\vec b) = a)$, then there is an $f \in M_1$ such that 
$(\forall b\in \N)( [e]_R(\Theta_{\mathcal{M}},\mu,\vec g , b ,\vec b) = f(b)).$
\item If for some $\vec g,\vec b,e$ we have $(\forall g \in M_1)( \exists a)( [e]_R(\Theta_{\mathcal{M}},\mu,g,\vec g,\vec b) = a)$, then there is a $G \in M_2$ such that $(\forall g \in M_1) [e]_R(\Theta_{\mathcal{M}},\mu,g, \vec g,\vec b) = G(g).$
\item The relation $R$ is a fixed point of the  Kleene schemes from Definition~\ref{1.1} interpreted over $\mathcal{M}$.
\end{enumerate}
\end{enumerate}}\end{definition}
We will not distinguish in notation between  $\mu$ in the structure $\mathcal{M}$ and $\mu$ in the full universe. 
For the below proofs, we need to code countable $\Theta$-structures as objects of type~1.   Clearly, the set of codes for countable $\Theta$-structures will be arithmetical:
\begin{definition}{\em 
Let ${\mathcal{M}} = \langle \N , M_1 , M_2 , \mu , \Theta_{\mathcal{M}} , R\rangle$ be a countable $\Theta$-structure.
\newline
A \emph{code} for $\mathcal{M}$ is a function $f = \langle f_1,f_2,f_3,f_4\rangle$ such that
\begin{enumerate}
 \renewcommand{\theenumi}{\roman{enumi}}
\item $f_1 = \langle f_{1,i}\rangle_{i \in \N}$ enumerates $M_1$ in a 1-1-way.
\item Let $f_2 = \langle f_{2,j}\rangle_{j \in \N}$ and let $F_j(f_{1,i}) = f_{2,j}(i)$. Then $\{F_j\}_{j \in \N}$ enumerates $M_2$ in a 1-1-way.
\item $f_3(\langle j,a\rangle) = \Theta_{\mathcal{M}}(F_j)(a)$ for all $j$ and $a$.
\item $f_4(\langle e, \langle i_1 , \ldots , i_k\rangle , \langle b_1 , \ldots , b_m\rangle , a \rangle) = 0$ if and only if \newline $ [e]_R(\Theta_{\mathcal{M}} , \mu , f_{i_1} , \ldots , f_{i_k} , b_1 , \ldots , b_m) = a$.
\end{enumerate}
}\end{definition}
It is essential for the below argument that the set of codes for $\Theta$-structures is arithmetical (or at least hyperarithmetical). 
The crucial part here is the first-order definition of special fan functionals.
The same result can be obtained for some other (classes of) functionals, but e.g.\ not for the Superjump or the Suslin functional. For those interested in such a generalisation, note that replacing $\SFF$ with another class of functionals $\Gamma$ requires that one can relativise $\Gamma$ to type structures $\mathcal M$.
\begin{definition}{\em Let $\mathcal{M}$ be a $\Theta$-structure. An \emph{extension} of $\Theta_{\mathcal{M}}$ is a special fan functional $\Theta_1$  such that whenever $F$ of type 2 is an extension of $G \in M_2$ then $\Theta_1(F) = \Theta_{\mathcal{M}}(G)$.}
\end{definition}
\begin{lemma}\label{1.9} 
For any $\Theta$-structure $\mathcal{M}$, the functional $\Theta_{\mathcal{M}}$ has an extension $\Theta_1$. 
\end{lemma}
\begin{proof}
Let $\Theta_0$ be any special fan functional, for instance the one constructed in \cite{dagsam}*{\S5}.  We define 
\[
\Theta_1(F) :=
\begin{cases}
 \Theta_{\mathcal{M}}(G) & \textup{ if $G \in M_2$ and $F$ extends $G$}\\
 \Theta_0(F) & \textup{ if there is no such $G \in M_2$}
\end{cases}.
\]
The definition of special fan functionals does not require any connection between the values of $\Theta(F_1)$ and $\Theta(F_2)$ when $F_1 \neq F_2$: we have only specified how $F$ and $\Theta(F)$ are related for each $F$. This relation will hold point-wise for each $(F,\Theta_1(F)) $ by construction, so $\Theta_1$ will also be a special fan functional.
\end{proof}
We could provide a similar construction and prove a similar lemma for other classes of  type 3 functionals, but not for all. Actually, we would always be able to find \emph{extensions} in a set-theoretical sense as above, but not necessarily in the class of functionals that we are interested in. The key property for us is that for  a given $F$ we specify, individually for that $F$, what an acceptable output of $F$ will be in such a way that we only have to know $F$ restricted to a countable (in this case, finite) set to justify that an alleged output is an acceptable one. If we, for instance, were interested in computations relative to $\exists^3$, we could not prove an extension lemma as above, since the constant zero in $\mathcal M$ will have extensions that are not constant zero, so the value of $\exists^3_{\mathcal{M}}$ cannot be preserved through extensions.

\medskip

Even though the relation $R$ does not have to represent the least fixed point of the Kleene schemes restricted to $\mathcal{M}$, 
we will see that it  will contain this least fixed point as a sub-relation. In fact, we have the following lemma, where we only make use of extensions in general, not of the fact that we deal with special fan functionals.  
\begin{lemma}\label{1.10} Let ${\mathcal{M}} = \langle \N,M_1,M_2,\Theta_{
\mathcal{M}},\mu,R\rangle$ be a $\Theta$-structure. 
Let $\Theta_1$ be an extension of $\Theta_{\mathcal{M}}$ as above. Let $\vec g$ be a sequence from $M_1$ and $\vec b$ a sequence from $\N$.
If $ \{e\}(\Theta_1 , \mu ,  \vec g,\vec b) = a$, then $[e]_R(\Theta_{\mathcal{M}},\mu,\vec g,\vec b) = a$.
\end{lemma}
\begin{proof}
We prove this by induction on the ordinal rank of the computation of $\{e\}(\Theta_1,\mu,\vec g,\vec b)$.  The proof will be given by cases following the schemes. For the schemes S1, S2, S3 and S7, the cases of initial computations, the claim follows directly from the assumption that $R$ is a fixed point of the inductive operator whose least fixed point is the true set of terminating computations. 

\medskip

For the schemes S4 (composition), S5 (primitive recursion), S6 (permutation of arguments) and S9 (enumeration), the claim follows by the induction hypothesis and the assumption on $R$. This leaves us with the two special instances of S8:
\begin{itemize}
\item $ \{e\}(\Theta_1 , \mu, \vec g,\vec b) = \mu(\lambda x^0.\{d\}(\Theta_1 , \mu ,\vec g , x , \vec b))$. By the induction hypothesis and the closure properties of $\mathcal{M}$ we have
\[
\lambda x.\{d\}(\Theta_1,\mu,\vec g, x , \vec b) = \lambda x.[d]_R(\Theta_{\mathcal{M}},\mu,\vec g,x,\vec b) \in M_1,
\]
and the application of $\mu$ will yield the same result if we consider $\mu$ as an element of $M_2$ or as an element of full type 2. Then, since $R$ is a fixed point of the Kleene computation operator, we have that
\[
[e]_R(\Theta_{\mathcal{M}} , \mu , \vec g,\vec b) = \{e\}(\Theta_1 , \mu , \vec g,\vec b).
\]
\item $\{e\}(\Theta_1,\mu,m,\vec g,\vec b) = \Theta_1(\lambda g.\{d\}(\Theta_1 , \mu , g , \vec g,\vec b))(m)$. By the induction hypothesis and the closure properties of $\mathcal{M}$, we have that
\be\label{xyz}
\lambda g \in M_1.\{d\}(\Theta_1 , \mu , g, \vec g,\vec b) = \lambda g \in M_1 . [d]_R(\Theta_{\mathcal{M}},\mu,g,\vec g,\vec b) \in M_2.
\ee
Let $G^{2}$ be the function defined by \eqref{xyz}. Then $F = \lambda g \in \N^\N.\{e\}(\Theta_1 , \mu , g,\vec g,\vec b)$ is a total extension of $G$, so $\Theta_1(F) = \Theta_{\mathcal{M}}(G)$ by the assumption on $\Theta_1$. The induction step then follows as above.
\end{itemize}
We have now treated all nine schemes, and the proof is done.  
\end{proof}
We need one more lemma as follows.
\begin{lemma}\label{3.8} For each finite sequence $\vec f$ from $\N^\N$ and  special fan functional $\Theta_1$, there is a countable $\Theta$-structure ${\mathcal{M}} = \langle \N,M_1,M_2,\mu,\Theta_{\mathcal{M}},R\rangle$ with $\vec f$ in $M_1$ such that for all $e$, $\vec g \in M_1^*$, $\vec b \in \N^*$, and $a \in \N$, we have that
\[
[e]_R(\Theta_{\mathcal{M}}, \mu,\vec g , \vec b) = a \leftrightarrow   \{e\}(\Theta_1,\mu,\vec g , \vec b) = a.
\]
\end{lemma}
\begin{proof}
We define $M_1$ as a kind of Skolem hull, and we define $M_2$, $\Theta_{\mathcal{M}}$, and $R$ explicitly from $M_1$ and $\Theta_1$. We will need that $\Theta_1$ is a special fan functional in order to show that $\mathcal{M}$ models that $\Theta_{\mathcal{M}}$ is a special fan functional, but the rest of the proof works for all type three objects.

Thus, let $M_1$ be countable such that
\begin{itemize}
\item[(i)] Each $f_i$ from $\vec f$ is in $M_1$
\item[(ii)] If $g$ is computable in $\Theta_1$, $\mu$ and a sequence $\vec g$ from $M_1$, then $g \in M_1$
\item[(iii)] If $F$ is a partial functional of type 2 computable in $\Theta_1$, $\mu$, and some $\vec g$ from $M_1$, and there is some $g$ for which $F(g)$ is undefined, then there is some $g \in M_1$ such that $F(g)$ is undefined. (This is the main Skolem hull part, and here we need the axiom of choice in a non-trivial way.)
\end{itemize}
We then let $M_2$ consist of all restrictions of $F$ to $M_1$, where $F$ is total and computable in $\Theta_{1}$, $\mu$ and some $\vec g$ in $M_1$. If $G$ is the restriction of $F$ in this way, we put $\Theta_{\mathcal{M}}(G) := \Theta_1(F)$.
 We put $[e]_R(\Theta_{\mathcal{M}},\mu,\vec g , \vec b) = a$ if and only if $\{e\}(\Theta_1 , \mu , \vec g , \vec b) = a$ for $\vec g$ in $M_1$.
Then (iii) will ensure that totality of functionals of type 2 is absolute for $\mathcal{M}$: If ${\mathcal{M}}  \models \forall g \exists a [e]_R(\Theta , \mu , g , \vec g,\vec b) = a,$ then $F$, defined by $F(g) = \{e\}(\Theta_1,\mu,\vec g , \vec b)$, is total, and the restriction to $\mathcal{M}$ is in $M_2$.

\medskip

By a similar argument, we observe that $\Theta_{\mathcal{M}}$ will be extensional: If $F_1\neq F_2$, both are total and computable in $\Theta_1$, $\mu$ and elements from $M_1$, then the partial functional $F_3$, where $F_3(g) = 0$ when $F_1(g) = F_2(g)$ and undefined otherwise, will also be computable in $\Theta_1$, $\mu$ and elements from $M_1$, and by (iii), $M_1$ will contain a $g$ such that $F_1(g) \neq F_2(g)$. Thus, the restriction operator will be 1-1, and $\Theta_{\mathcal{M}}$ is thus extensional, that is, well defined.
Except for the construction of $M_1$, the construction of $\mathcal{M}$ is explicit. Moreover, if $F \in M_2$ and $G$ is the unique extension of $F$ computable in $\Theta_1$, $\mu$ and elements from $M_1$, we have that $\Theta_1(G) \in M_1$,  and that $\Theta_{\mathcal{M}}(F) = \Theta_1(G)$ codes a finite subset of $M_1$ that, together with $G$ (or $F$) forms a finite cover of $C$, so $\Theta_{\mathcal{M}}$ will be a special fan functional from the point of view of $\mathcal{M}$. Thus $\mathcal{M}$ will satisfy the claim of the lemma. 
\end{proof}
Finally, we can prove Theorems \ref{1.2} and \ref{1.3} as follows.
\begin{proof}(of Theorem \ref{1.2})
First of all, the functional $\Theta_0$ is defined in \cite{dagsam}*{\S5} and Lemma \ref{3.8} implies that there is at least one countable $\Theta$-structure $\mathcal{M}$, i.e.\ 
the set of codes for $\Theta$-structures is hyperarithmetical and \emph{non-empty}.
By (essentially) the Gandy basis theorem (\cite{Sacks.high}*{III.1.4}), there is then a code $f$ for a countable $\Theta$-structure ${\mathcal{M}} = \langle \N,M_1,M_2,\mu,\Theta_{\mathcal{M}},R\rangle$ such that $\omega_1^f = \omega_1^{\textup{\textsf{CK}}}$.
By Lemma \ref{1.9}, $\Theta_{\mathcal{M}}$ has an extension $\Theta_1$, and by Lemma \ref{1.10}, all functions $g$ computable in $\Theta_1$ and $\mu$ are elements of $M_1$, and thus Turing computable in $f$. Then also $\omega_1^g = \omega_1^{\textup{\textsf{CK}}}$. 
\end{proof}
The following provides a proof for Theorem \ref{1.3}.
\begin{proof}
By Lemmas \ref{1.9}, \ref{1.10} and \ref{3.8}, the following are equivalent, given $e, \vec g,\vec b, a$:
\begin{enumerate}
 \renewcommand{\theenumi}{\roman{enumi}}
\item For all special fan functionals $\Theta$, we have that $\{e\}(\Theta,\mu,\vec g,\vec b) = a$\label{corke1}
\item For all countable $\Theta$-structures $\mathcal{M} = \langle \N,M_1,M_2,\mu,\Theta_{\mathcal{M}} ,R\rangle$, we have that $[e]_R(\Theta_{\mathcal{M}},\mu ,\vec g,\vec b) = a$.\label{corke2}
\end{enumerate}
Via coding, the relation in \eqref{corke2} is $\Pi^1_1$, so the relation in \eqref{corke1} must also be $\Pi^1_1$.
\end{proof}

\subsection{Beyond the hyperarithmetical via $\Theta$ and $\mu^{2}$}\label{atrsec}
In this section, we provide a \emph{direct} proof that the combination $\Theta$ and $\mu^{2}$ computes a realiser for $\ATR_{0}$.  

\medskip

We proved in \cite{dagsam} that there is no instance $\Theta$ such that all functions computable in $\Theta$ and $\mu$ are hyperarithmetical. 
We gave two proofs: one by a direct construction of a hyperarithmetical functional $F$ such that $\Theta(F)$ can never be contained in the hyperarithmetical functions, and one by applying term extraction to 
\[
\P_{0}\vdash \paai+\STP\di [\ATR_{0}]^{\st}
\]
which (indirectly) yields a realiser for $\ATR_{0}$ in terms of $\Theta$ and $\mu^{2}$.
There are thus two proofs of essentially the same result, one explicit construction where we do not analyse the logical strength needed and one indirect, via term extraction, where the underlying logic is explicit. We consider both approaches to be of value.

\medskip

In a nutshell, the aim of this section is to prove (inside $\ACA_{0}$) that $\ATR_0$ follows from the \emph{Arithmetical Compactness of $C$}, defined as follows. 
\begin{definition}[Arithmetical Compactness of $C$]
For any arithmetically defined $F:C \rightarrow \N$, where we allow function parameters, there are $f_1 , \ldots , f_n \in C$ such that
\[
C \subseteq C_{\bar f_1(F(f_1))} \cup \cdots \cup C_{\bar f_n(F(f_n))}.
\]
\end{definition}

With the exception that we have used the symbol `$\Theta$'  for other purposes (namely to denote a special fan functional), we mostly adopt Simpson's notation regarding $\ATR_{0}$ from \cite{simpson2}*{V.2}, namely as follows.
\bnota
Let $\Gamma(n,X,Z)$ be an arithmetical formula, inducing the  operator 
\[
\hat \Gamma(X,Z) = \{n \mid \Gamma(n, X,Z)\}
\]
 seen as an inductive operator in the first set variable $X$. We assume $A \subseteq \N$ and let $<_A$ be a total ordering of $A$. We use $A$ and $<_A$ as hidden parameters, and when using the variable $Y$, we implicitly assume that $Y \subseteq  \N^2$.  We define $Y_a := \{n \mid (a,n) \in Y\}$ and $Y^a := \{(b,m) \mid (b,m) \in Y \wedge b <_A a\}$ for $a \in A$.  Finally, $H(Y,Z)$ is the arithmetical statement $(\forall a \in A) (Y_a = \hat \Gamma(Y^a,Z))$. 
\enota
\begin{theorem}\label{theorem.ATR} Given $\Gamma$ as above, there is an arithmetical function $G^{2}$ such that if $F(g) = G(g,A, <_A,Z)$ \($g$ varies over $C$\)  and $g_1 , \ldots , g_n$ are as in \emph{Arithmetical Compactness} for $F$, then we can construct \(uniformly arithmetically in $Z$, $A$, $<_A$ and $g_1 , \ldots , g_n$\) a pair $(Y,h)$ such that either $H(Y,Z)$ or $h:\N \rightarrow \N$ is a strictly $<_A$-descending sequence in $A$. The verification can be formalised in $\ACAo$.
\end{theorem}
\begin{proof}
Given $g$, we put $Y[g]: = \{(b,k) \mid g(\langle b,k \rangle) = 0\}$. We now define $G$ as follows:
 the number $G(g,A,<_A,Z) $ is defined to be
\begin{enumerate}
 \renewcommand{\theenumi}{\roman{enumi}}
\item $0$ if $H(Y[g],Z)$ or there is no $<_A$-minmal $a$ such that $(Y[g])_a \neq \hat \Gamma((Y[g])^a,Z)$.
\item $\langle a,k\rangle + 1$ if $a$ is $<_A$ minimal such that $(Y[g])_a \neq \hat \Gamma((Y[g])^a,Z)$ and $k$ is the least integer in the symmetric difference of $(Y[g])_a$ and $\hat \Gamma((Y[g])^a,Z)$.
\end{enumerate}
Let $F(g) = G(g,A,<_A,Z)$ and let $g_1 , \ldots , g_n$ be such that $C = C_{\bar g_1(G(g_1))} \cup \cdots \cup C_{\bar g_n(G(g_n))}.$
If for some $i$ we have $F(g_i) = 0$, then either $H(Y[g_i],Z)$ or this is not the case since there is no $<_A$-minimal $a$ such that $(Y[g_i])_a \neq \hat \Gamma((Y[g_i])^a,Z)$. 
We select the least such $g_i$ in the lexicographical ordering on $C$. In the first case, we let $Y = Y[g_i]$ and $h$ be the constant zero, and in the second case we may also let $Y = Y[g_i]$, but we combine $\mu$-recursion and primitive recursion and let $h$ be a strictly descending $<_A$ sequence of $a$'s such that $(Y[g_i])_a \neq \hat \Gamma((Y[g_i])^a)$. 

\medskip

The other possibility is that $F(g_i) = \langle a_i,k_i\rangle + 1$ for $i = 1 , \ldots , n$. If there are $i\neq j$ such that $Y[g_i] \cap A \times \N \neq Y[g_j]  \cap A \times \N$ and there is no $<_A$-minimal $a$ with $(Y[g_i])_a \neq (Y[g_j])_a$, we can extract an infinite descending sequence in $A$ from this information. We will show that the absence of such $i$ and $j$ will lead to a contradiction. So assume that there is no such $i$ and $j$. Without loss of generality, we may assume that $a_1 \leq_A a_2 \leq_A \cdots \leq_A a_n$. We make three observations:
\begin{enumerate}
 \renewcommand{\theenumi}{\roman{enumi}}
\item If $Y[g_i] \cap (A \times \N) = Y[g_j] \cap (A \times \N)$, then $a_i = a_j$.
\item If $Y[g_i] \cap (A \times \N )\neq Y[g_j] \cap (A \times \N)$ and $a$ is the $A$-least number where $(Y[g_i])_a$ and $(Y[g_j])_a$ differ, then $\min_A\{a_i,a_j\} \leq_A a.$
\item Given $i$, if $g$ is such that $(Y[g])^{a_i} = (Y[g_i])^{a_i}$ and $(Y[g])_{a_i} = \hat \Gamma((Y[g_i])^{a_i},Z)$, then $g$ is not covered by $C_{\bar g_i(F(g_i))}$. Moreover, if $a_i <_A a_j$, then $g_j$ will satisfy this property of $g$.
\end{enumerate}
It follows that if $g$ is such that $(Y[g])^{a_n} = (Y[g_n])^{a_n}$ and $(Y[g])_{a_n} = \hat \Gamma ((Y[g_n])^{a_n},Z)$, then $g$ is not in any of the sets $C_{\bar g_i(F(g_i))}$, so these sets do not form a cover. This is the desired contradiction. 

It is easy to see that all steps here can be formalised.
\end{proof}
This gives an alternative proof of the following corollary.  
\begin{corollary} \label{1.18} 
There is no special fan functional  $\Theta$ that, together with $\mu$, computes only hyperarithmetical functions.
\end{corollary}
\begin{proof}It is well established that there is no hyperarithmetical realiser for $\ATR_0$, see e.g.\  the proof of V.2.6 in \cite{simpson2}.\end{proof}
We also have the following corollary relativising the proof above.  
\begin{corollary}
There is an arithmetically defined function $F:C^2\rightarrow \N$ such that for no special fan functional $\Theta$, the function
$F(x) = \Theta(\lambda y.F(x,y))$ is Borel. \end{corollary}
\begin{proof}
For $X\subset \N$, there is a total ordering computable in $X$ that is not a well-ordering, but such that there is no descending sequence in the ordering hyperarithmetical in $X$.  
Hence, there is no realiser for $\ATR_0$ hyperarithmetical in any $X$, i.e.\ no realiser that is Borel. 
Since we can obtain a realiser for $\ATR_0$ by section-wise application of $\Theta$ to an arithmetical functional of two variables, 
we are done.
\end{proof}
Finally, Hunter introduces a functional in \cite{hunterphd}*{p.\ 23} that constitutes a `uniform' version of $\ATR_{0}$.  
This functional is computable from $\Theta$ plus $\mu$, as follows  
\begin{cor}
Uniformly primitive recursive in $\mu^{2}$ and $\Theta^{3}$ there is a functional
$T:\N^\N \times (2^\N \rightarrow 2^\N) \rightarrow 2^\N$ such that when $f^{1}$ codes a well-ordering $<_f$, then $T(f,F)$ satisfies the following recursion equation for $a$ in the domain of $<_f$\textup{:}
\[
\{b : \langle b,a \rangle \in T(f,F)\} = F(\{\langle c,d \rangle \in T(f,F) : d <_f a\}).
\]
\end{cor}
\begin{proof}
In the proof of Theorem \ref{theorem.ATR}, note the fact that $\Gamma$ is arithmetical is (only) used to prove that the defined functional $G$ is arithmetical. The full proof therefore relativises to any $F$ of the relevant type.
\end{proof}

\subsection{Not beyond the hyperarithmetical via $\Lambda$ and $\mu^{2}$}\label{essdag}
In this section, we introduce a functional  $\Lambda_1$ of type $(C \rightarrow \N) \rightarrow (\N \rightarrow C)$, where $C = 2^{\N}$ is the Cantor space, with the following two properties:
\begin{enumerate} 
  \renewcommand{\theenumi}{\roman{enumi}}
\item If $\Lambda_1(F) = \{f_i\}_{i \in \N}$, then $\bigcup_{i \in \N}C_{\bar f_i(F(f_i))}$ has measure 1.\label{godver}
\item Only hyperarithmetical functions are (S1-S9) computable in $\Lambda_1$ and $\mu^{2}$.\label{vaag}  
\end{enumerate}
Our motivation for introducing $\Lambda_1$ is to show that there is a weak fan functional in which no special fan functional is computable relative to $\mu^{2}$. 
In \cite{dagsamVI}, this result is linked to the RM of measure theory, the original Vitali Covering theorem (\cite{vitaliorg}) in particular, and it is also generalised  to recursion relative to the Suslin functional.

\medskip

As discussed in Section \ref{3.1.2}, item~\eqref{godver} just means that $\Lambda_1$ is a weak fan functional up to computational equivalence. The existence of a functional $\Lambda$ satisfying item~\eqref{godver} follows from the existence of $\Theta$, and we let $\Lambda_0$ be some fixed instance of $\Theta$.
We define  $\Lambda_{1}$  in equation \eqref{lambda1337} below, namely in terms of $\Lambda_{0}$ and by specifying a different value for certain $F$. Since item \eqref{godver} does not require any connection between $\Lambda_1(F)$ and $\Lambda_1(G)$ when $F \neq_{2} G$,  we have much freedom in constructing $\Lambda_1$. Of course, item \eqref{vaag} puts some clear restrictions on how we can define $\Lambda_1$. For instance, if $F$ is hyperarithmetical, i.e.\ computable in $\mu$, we must have that $\Lambda_1(F)$ is hyperarithmetical. This can be arranged using basic measure theory and the Sacks-Tanaka theorems for measure-theoretic uniformity (see below). The next challenge is presented by  `iterated' outputs like for instance
\be\label{mennejezus}
 \Lambda_1(\lambda f .\Lambda_1(\lambda g . F(f,g))(17));
\ee
these also need to be hyperarithmetical whenever $F$ is. Again, basic measure theory and the Sacks-Tanaka machinery come to our rescue: as it turns out, except for a set of $f$'s of measure zero, we can use the same value for $ \Lambda_1(\lambda g \in C F(f,g))$ \emph{independent} of $f$. However, we cannot expect to be able to use the same value as the output value in \eqref{mennejezus}: the more involved $\Lambda_1$ is in a computation $\{e\}(\Lambda_1,\mu, \vec b)$, the harder it is to find a hyperarithmetical output. Our guiding idea is that we may us the same hyperarithmetical value of $\Lambda(F)$ for almost all $F$ computable at a certain countable level. We make this precise, as follows.

\medskip

In the construction of $\Lambda_{1}$, we use the available machinery from measure theory and hyperarithmetical theory (i.e.\ the computability theory of $\mu$), to construct a well-ordered sequence of \emph{possible} values for $\Lambda_1$ indexed over  the first non-computable ordinal 
$\omega_1^{\textup{\textsf{CK}}}$ and (indirectly) a set $X$ of measure 1 so that whenever $F$ is computable in $\Lambda_1$ and elements from $X$, then we may let $\Lambda_1(F)$ be in that sequence. This may look circular, but in reality, $\Lambda_1$ and our sequence will be defined by a simultaneous transfinite recursion over $\omega_1^{\textup{\textsf{CK}}}$.  This transfinite recursion is unfortunately (and unavoidably, we believe) a rather complex one.

\medskip

Now, let us consider the machinery we need. First of all, we assume without mentioning that all sets and functions are measurable. Actually, we will only work with subsets of finite or countable products of the Cantor space $C$ that are $\Sigma^1_1$ or $\Pi^1_1$ relative to objects of type 1, so measurability will not be an issue. The Cantor space $C$ will have measure 1, so all products  will have measure 1. We use $\m$ for the measure on all such product spaces. We will let ``almost everywhere" mean that the property holds except possibly on a set of measure 0, which in our cases means that the property holds on a set of measure 1.  We write `a.a.' as short for `almost all'.  
We will rely on two facts from measure theory, where all spaces are products $E$ or $D$ of the Cantor space $C$.
\begin{proposition}\label{prop.measure}~
\begin{enumerate}
\item[(i)] A countable intersection of sets with measure 1 has measure 1.
\item[(ii)] If $X \subset E \times D$ then $\m(X) = 1$ if and only if $\m(\{e \mid (e,d) \in X\}) = 1$ for a.a.\ $d \in D$ (if and only if $\m(\{d \mid (e,d) \in X\}) = 1$ for a.a.\ $e \in E$).
\end{enumerate}
\end{proposition}
These facts can be found in any standard textbook on measure theory.  Item (ii) is actually a special case of Fubini's theorem for characteristic functions.

\medskip

By convention, we denote infinite sequences of binary functions as  $(f) := \{f_j\}_{j \in \N}$.  
\begin{definition}\label{kelow}{\em 
Let $F$ be a partial function from $C$ to $\N$ and let $(f)$ be a sequence.
\begin{enumerate}
  \renewcommand{\theenumi}{\roman{enumi}}
\item We say that $(f)$ \emph{suffices} for $F$ if $F(f_j)$ is defined for all $j \in \N$ and 
$\m(\bigcup_{j \in \N} C_{\bar f_j(F(f_j))}) = 1.$
\item We say that $(f)$ \emph{fails} $F$ if $F(f_i)$ is undefined for some $i \in \N$.
\end{enumerate}
}\end{definition}
Now, $(f)$ suffices for $F$ exactly when $(f)$ can be an acceptable value of $\Lambda_1(G)$ for all total $G$ extending $F$. In the next lemmas, we will make the following intuition precise: we can choose the same value $\Lambda_1(F)$ for large parameterised classes of $F$s, and  we have a lot of freedom in choosing this common value. 

\medskip
\noindent
All the below arguments are elementary from the point of view of measure theory.
\begin{lemma}\label{5.2} Let $F:C \rightarrow \N$ be a partial \(measurable\) functional with measurable domain.
\begin{enumerate}
  \renewcommand{\theenumi}{\roman{enumi}}
\item If the domain of $F$ has measure 1, then $\{(f)\mid (f) \;{\rm suffices\;for}~F\}$ has measure 1.\label{godverdoeme1}
\item If the domain of $F$ has measure less than 1, then $\{(f) \mid (f)\;{\rm fails}\; F\}$ has measure 1.\label{isdanouzom}
\end{enumerate}
\end{lemma}
\begin{proof}
Proof of item \eqref{godverdoeme1}: 
for $k\in \N$, we will prove that the set of $(f)$ such that $ \m(\bigcup_{j \in \N}C_{\bar f_j(F(f_j))}) > 1-2^{-k}$, has measure 1.
To this end, let $n_k$ be so large that $\m(\{f \in C \mid F(f) < n_k\}) > 1 - 2^{-k}.$
Let $s_{k,1}, \ldots , s_{k,m_k}$ be the binary sequences $s_{k,l}$ of length $n_k$ such that
$\m(\{f \in C_{s _{k,l}}\mid F(f) < n_k\}) > 0$. Let  $ r_{k,l} $ be this positive measure.
Then $\m(\bigcup_{l = 1}^{m_k}C_{s_{k,l}}) > 1-2^{-k}$ and for each $s_{k,l}$ the set of $(f)$ such that for some $f_j$, $F(f_j) < n_k$ and $f_j$ extends $s_{k,l}$, has measure~1.  Indeed the probability of not satisfying this is $\prod_{j = 0}^\infty (1-r_{k,l}) = 0$.
Since a finite intersection of sets of measure 1 still has measure 1, our claim follows; the previous generalises to countable intersections and item \eqref{godverdoeme1} holds.

\medskip

Proof of item \eqref{isdanouzom}: in this case, the probability that $f_i$ is in the domain of $F$ is smaller than 1 by a fixed value. Then the infinite product of the domain of $F$ has measure 0. Thus, $(f)$ fails $F$ when $(f)$ is in the complement of this product.
\end{proof}
\begin{lemma}\label{5.3} 
Let $F:C^2 \rightarrow \N$ be a partial, measurable functional defined on a measurable set  and put $F_g(f) = F(f,g)$. 
Then the set of $(f)\in C^{\N}$ such that  for a.a.\ $g\in C$ we have that
\begin{itemize}
\item[(i)]  If the domain of $F_g$ has measure 1, then $(f)$ suffices for $F_g$
\item[(ii)] If the domain of $F_g$ has measure $<1$ then $(f)$ fails $F_g$
\end{itemize}has measure 1.
\end{lemma}
\begin{proof}
Let $X$ be the set of $\langle g,(f)\rangle$ such that either $F_g$ is defined on a set of measure 1 and $(f)$ suffices for $F_g$ or $F_g$ is defined on a set of measure $< 1$ and $(f)$ fails $F_g$. 
By Lemma \ref{5.2}, this set has measure 1, since for all $g$ we have for almost all $(f)$ that $\langle g, (f)\rangle \in X$. Then, by Proposition \ref{prop.measure}, the set of $(f)$ such that $\langle g,(f)\rangle \in X$ for almost all $g$ has measure 1, and we are done.  
\end{proof}
Since all (finite or countable) products of $C$ we consider are isomorphic (with the exception of $C^0$), we will apply the previous lemma in other cases than for $C^2$ as well. For technical reasons, we shall need a strengthening of Lemma \ref{5.3} as follows.  
\begin{definition}{\em 
Let  $\vec c$ be a non-repeating sequence from $\N$. We define $(f)_{\vec c}$ as the sequence of $f_i$ indexed via $\vec c$.}\end{definition}
\begin{lemma}\label{1.22} Let $\vec c$ be a non-repeating sequence of length $k'$ and let $X \subseteq C^{k'+k}$ have measure 1.
Let $F:C\times C^{k' + k} \rightarrow \N$ be a partial functional that is measurable with a measurable domain and put $F_{\vec h,\vec g}(g) := F( g,\vec h , \vec g)$ for $\vec h \in C^{k'}$ and $\vec g \in C^k$. Then the set of pairs $\langle (f), \vec g\rangle$ such that $\langle (f)_{\vec c},\vec g\rangle \in X$ and
\begin{enumerate}
 \renewcommand{\theenumi}{\roman{enumi}}
\item if the domain of $F_{(f)_{\vec c},\vec g}$ has measure 1, then $(f)$ suffices for  $F_{(f)_{\vec c},\vec g}$, and
\item if the domain of $F_{(f)_{\vec c},\vec g}$ has measure $< 1$,  then $(f)$ fails $F_{\(f)_{\vec c},\vec g}$,
\end{enumerate}
has measure 1.
\end{lemma}
\begin{proof}
Combining Proposition \ref{prop.measure}.(ii) with  the arguments of Lemmas~\ref{5.2} and~\ref{5.3}, the lemma follows easily. The assumption that $\vec c$ is non-repeating is essential here, since otherwise the set of possible $(f)_{\vec c}$ will have measure 0 and not 1.
\end{proof}
Now we have established the measure-theoretical lingo we need for the construction of $\Lambda_1$.  %
 In order to prove the main technical lemma, we also need some theorems from higher computability theory. We have formulated them in the form we need. For proofs, see \cite{Sacks.high}*{Sections IV.1-2 and Section X.4}. We will actually need some of these results in relativised forms, as follows.  
 \begin{proposition}\label{prop1} ~
 \begin{enumerate}
  \renewcommand{\theenumi}{\roman{enumi}}
 \item If $A \subset C$ is computable  in $f$ and $\mu$ via index $e$, then the relation $\m(A) = 1$ is decidable in $\mu$, uniformly in $f$ and $e$.\label{godverdoeme}
 \item {\em [Sacks,Tanaka] } If $A \subset C$ is hyperarithmetical and $\m(A) > 0$, then $A$ contains a hyperarithmetical element.\label{nohard} 
 \item{\em [Gandy Selection]} If a $\Pi^1_1$-set of functions contains a hyperarithmetical element, we may find one, effectively in $\mu$.\label{links} 
 \item {\em [Sacks, Tanaka]} The set of $g \in C$ such that $\omega^g_1 = \omega_1^{\textup{\textsf{CK}}}$ has measure 1.\label{huggo} 
 \end{enumerate}
 \end{proposition}
 \begin{proof}
 The items from the theorem are proved as follows in \cite{Sacks.high}.  
Item \eqref{godverdoeme} is Theorem IV.1.3. Item \eqref{nohard} is  Theorem IV.2.2. Item \eqref{links} is proved as Theorem X.4.1 in a more general form.  Item \eqref{huggo} is Corollary IV.1.6.
 \end{proof}
We have established the general machinery needed below, and now start working towards the main result of his section.
 \begin{convention}\label{confla}\rm
From now on, we let `$\prec$' be a total, computable  ordering of $\N$  such that the well-ordered initial segment has length $\omega_1^{\textup{\textsf{CK}}}$; $\prec$ may not be a well-ordering, but we will not actively use this fact.  
 We let $W = W(\prec)$ be the elements in the well-ordered part, and for $i \in W$ we let $\alpha_i$ be the ordinal rank of $i$ in $\prec$. 
\end{convention}
 It is well-known that orderings as in Convention \ref{confla} exist.  The set of computable total orderings that in addition are well-orderings, is complete $\Pi^1_1$; the set of  computable total orderings without hyperarithmetical infinite descending sequences is $\Sigma^1_1$.  Thus there is one ordering that is of the latter kind but that is not of the former kind. Such orderings are known as \emph{computable pseudo-wellorderings}. The well-ordered initial segment is $\Pi^1_1$, but not $\Delta^1_1$ or hyperarithmetical. 
 
 \medskip
 
Let $[f]$ be a (double) sequence $\{(f_i)\}_{i \in W} = \{f_{i,j}\}_{i \in W,j \in \N}$ in $(W\times \N) \rightarrow (\N \rightarrow \{0,1\})$. Each $[f]$ like this will define a \emph{partial} approximation $\Lambda_{[f]}$ to a weak fan functional in the following sense:
\begin{definition}{\em Let $[f]$ be as above.  For $F:C \rightarrow \N$, we define 
\begin{enumerate}
  \renewcommand{\theenumi}{\roman{enumi}}
\item $\Lambda_{[f]}(F) = (f_i)$ if $i \in W$, $(f_i)$ is sufficient for $F$ and no $(f_{i'})$ is sufficient for $F$ for $i' \prec i$.
\item $\Lambda_{[f]}(F)$ is undefined if there is no such $i \in W$.
\end{enumerate}}
\end{definition}
 Our aim is to construct $[f]$ is such a way that all functions computable in any total extension $\Lambda$ of $\Lambda_{[f]}$ and $\mu^{2}$ are hyperarithmetical.  However, such a construction requires controlling the complexity of computations relative to {any} such extension. One obstacle is the requirement in Kleene's S8 that the input functional $F$ must be total. We get around this obstacle by considering a more liberal interpretation of S8, so that it works for partial inputs as well, as long as they contain the relevant information. Such interpretations are well-established; see e.g.\ \cite{longmann}*{\S6.4}.

 \begin{definition}{\em We define the relation `$\{e\}_{[f]}(\Lambda_{[f]} , \mu , \vec g , \vec b) = a$' by transfinite recursion.  
 We keep the schemes S1-S7, S8.1 and S9 from Definition \ref{1.1}, only adding $[f]$ as an index everywhere. We omit S8.2 and give a new interpretation of S8.3:
 \begin{itemize}
\item[(S8.3)] If $e = \langle 8,3,e_1 \rangle$, let $F(g) = \{e_1\}_{[f]}(\Lambda_{[f]} , \mu , g,\vec g,\vec b)$. Then $F$ is in general a partial function of type 2. Let $i \in W$ be the $\prec$-least number such that:
\begin{itemize}
\item[(i)] the value $F(f_{i' , j})$ is defined for all $i' \preceq i$ and all $j \in \N$, 
 \item[(ii)] the sequence $(f_i)$ suffices for $F$.  If there is one such $i$, then define $\{e\}_{[f]}(\Lambda_{[f]} , \mu ,\vec g,a,b,\vec b) := f_{i,a}(b)$; undefined otherwise.  
\end{itemize}
\end{itemize}}
\end{definition}
Since `$\{e\}_{[f]}(\Lambda_{[f]} , \mu , \vec g , \vec b) = a$' as in the previous definition is defined as the least fixed point of a positive inductive operator, each sequence $\langle e, \Lambda_{[f]} , \mu , \vec g , \vec b , a\rangle$ in the relation has an ordinal rank.   
Since we only require -even in the case of S8.3- a countable set of immediate sub-computations to terminate, the rank of any terminating computation modulo a given $[f]$ is countable, and actually an ordinal computable in $[f]$ and the argument list $\vec g$. 

\medskip

We will only apply this definition in the case where the map $i \mapsto (f_i)$ is a function that is partially Kleene-computable in $\mu^{2}$. In this case, the partial function $\{e\}_{[f]}(\Lambda_{[f]} , \mu , \vec g , \vec b)$
will be computable in $\mu^{2}$ as well.  Our goal is to construct $[f]$ in such a way that for all indices $e$, input arguments $\vec b$ from $\N$, and total extensions $\Lambda$ of $\Lambda_{|f]}$, we have that 
\be\label{takske}
\{e\}(\Lambda_{[f]}, \mu , \vec b) \simeq \{e\}_{[f]}(\Lambda_{[f]}, \mu, \vec b),
\ee
where `$\simeq$' means that both sides are undefined or both sides are defined and equal. If we succeed, we obviously have that all functions Kleene-computable in $\Lambda$ will be hyperarithmetical when $\Lambda$ is a total extension of $\Lambda_{|f]}$. We will not need that the total extension $\Lambda$ itself is a weak fan functional for this argument.

\medskip

After the construction of $[f]$, we will prove \eqref{takske} by induction on the ordinal rank of the true Kleene-computation $\{e\}(\Lambda, \mu , \vec b)$.  
In order to make this proof work, we have to take into account that there are sub-computations with arguments from $C$. We will see that it will be possible to construct $[f]$ such that we only have to consider argument sequences  $\vec g$ of length $k$ from a $\Sigma^1_1$-set $X_k$ of measure 1. Let us now outline the construction, and what we attempt to achieve at each step:
\begin{enumerate}
  \renewcommand{\theenumi}{\roman{enumi}}
\item We will construct $[f]$ by defining $(f_i)$ by recursion on $i \in W$
\item In parallel to defining $(f_i)$, we will for each integer $k$ construct a hyperarithmetical set $X_{i,k} \subseteq C^k$ of measure 1 such that for any set of parameters $\vec g \in X_{i,k}$, integer parameters $\vec b$, and index $e$, we will have that the intended proof by induction will work for computations $\{e\}(\Lambda , \mu , \vec g , \vec b)$ of ordinal rank bounded by $\alpha_i$, the rank of $i$ in $(W, \prec)$. 
\item  $(f_i)$ will be chosen as a hyperarithmetical sequence that is sufficient for almost all functionals that are total on a set of measure 1 via computations strictly  bounded by $\alpha_i$, and fails  almost all the others.
\item To verify the key Lemma \ref{lemma.ind}, we have to consider inputs $\vec g$   together with inputs of the form $f_{i,j}$. Thus we will consider  input sequences $\vec h,\vec g$ where the sequences $\vec h$ are 'specified' as certain $f_{i,j}$ and the sequences $\vec g$ will vary over sets of measure 1. 
\end{enumerate}
 We will point out where the  sequences $\vec h$  are needed in our technical argument. The underlying idea is that we may pick $\Lambda(F)$ `at random' and the probability of success is 1.  However, this random value must be random with respect to values of $\Lambda$ obtained while computing $F$ from $\Lambda$. Making this idea precise, the need arises to take previous values of $\Lambda$ into account.

\medskip

 As a convention, when we write $\{e\}_{\ind}(\Lambda_{\ind} , \mu  , \vec g, \vec b)$, where `$\ind$' is any index, we assume without mentioning that the length of $\vec g$ fits the expression.
 \begin{convention}
 {\em If $(f_k)$ is a sequence $\{f_{k,j}\}_{j \in \N}$ for all $k \preceq i \in W$, then we write $[f]_i$ for $\{f_{k,j}\}_{k \preceq i , j \in \N}$.
 Similarly, if $(f_k)$ is a sequence $\{f_{k,j}\}_{j \in \N}$ for all $k \prec i$, then we write $[f]_{\prec i}$ for $\{f_{k,j}\}_{k \prec i , j \in \N}$.
 }\end{convention}
 Our definition of $\Lambda_{[f]}$ readily generalises to $\Lambda_{[f]_i}$ and $\Lambda_{[f]_{\prec i}}$, and so does the recursive definition of $\{e\}_{\Lambda_{[f]}}(\Lambda_{[f]},\mu, \vec g , \vec b)$. 
 \begin{convention}\label{con.1}{\em 
 In the formulation of the next lemma, we make us of three kinds of inputs: integers, elements of the form $f_{i,j}$ that can be seen as parameters, and sequences $\vec g$ from $C^k$ that can be seen as variables. 
 As a convention, we order them $(\vec h , \vec g, \vec b)$.  There is no harm in this since we may always use S6 to permute inputs. 
 In the proof, we shall introduce a fourth category $(f_i)_{\vec c}$ in the recursion step, objects that may be in the $\vec h$-part at later stages, but whose values have not been decided before we select the one $(f_i)$ we want to use. 
 
 \medskip
 
There is a small twist to this notation: for our construction and argument it is important that the sequence $\vec c$ is non-repeating, but for our application we may want to consider computations where the same function is used in several locations in the list of arguments. Instead of building up an unbearable notation, we assume that we have one case for each way of distributing the arguments $( \vec h , (f_i)_{\vec c} , \vec g, \vec b)$ as a list of inputs. Thus, each case we treat in the proof in theory covers countably many cases.  We will inform the reader when we actually make use of this.
 }
 \end{convention}
We have formulated our next item as a lemma, but it is in reality a combination of a construction by recursion and a verification of the key properties of this construction. 
We will refer to the details of the construction in later proofs.
\begin{lemma}\label{lemma.ind}
By  transfinite recursion on $i \in W$, we can construct $[f] = \{(f_i)\}_{i \in W}$ and sets $X_{i,k} \subseteq C^k$ of measure 1 \(for each $k \in \N$ and $i \in W$\) such that an alleged computation $\{e\}_{[f]_i}\big(\Lambda_{[f]_i},\mu,\vec h, \vec g, \vec b\big)$ will terminate whenever the parameters satisfy the following:
\begin{enumerate}
  \renewcommand{\theenumi}{\roman{enumi}}
\item $i \in W$ has norm $\alpha_i$, $e$ is a Kleene-index, $\vec b\in \textup{seq}$ and $\vec g \in X_{i,k}$, 
\item $\vec h $ is a  sequence from $\{f_{i',j} \mid i' \preceq i \wedge j \in \N\}$
\item there is some extension $[f']$ of ${[f]_i}$ such that $\{e\}_{[f']}(\Lambda_{[f']},\mu, \vec h, \vec g, \vec b)\!\!\downarrow$ with a computation of ordinal rank at most $\alpha_i$. 
\end{enumerate} 
\end{lemma}
\begin{proof}
We will show how to construct $(f_i)$ and $X_{i,k}$ from $[f]_{\prec i}$ and $\{X_{i',l} \mid i' \prec i, l \in \N\}$.  The key steps in our construction are:
\begin{enumerate}
  \renewcommand{\theenumi}{\alph{enumi}}
\item For each $k$,   find a hyperanalytical set $Z_{i,k} \subseteq C^\N \times C^k $  with measure 1, such that the induction step works for all $\langle(f),\vec g\rangle \in Z_{i,k}$ if we use $(f)$ as our $(f_i)$.\label{gkaf}  
\item We let  $Y_i$ be the set of $(f)$ such that for each $k\in \N$, we have ${\bf m}(\{\vec g \mid \langle (f) , \vec g\rangle \in Z_{i,k}\}) = 1$.
Then $Y_i$ has measure 1 and is hyperarithmetical by Proposition \ref{prop1}.(i). 
\item We then select $(f_i) \in Y_i$ computably in $\mu$ by the Sacks-Tanaka basis theorem (see Proposition \ref{prop1}.(ii)) and Gandy Selection (see Proposition \ref{prop1}.(iii)) computably in $\mu$.
\item 
Finally, we define $X_{i,k} := \{\vec g \mid \langle (f_i),\vec g \rangle \in Z_{i,k}\}$.
\end{enumerate}
The hard work will be to carry out step \eqref{gkaf}: the remaining steps then all follow by our general machinery.

\medskip

Now assume that $[f]_{\prec i}$ and each $X_{i',l}$, for $i' \prec i$ and $l \in \N$,  are constructed satisfying the claim of the lemma. 
We define $ X_{\prec i , k} = \bigcap_{i' \prec i} X_{i',k}$, noting that if $i_0$ is the $\prec$-least integer, then $X_{\prec i_0,k}$ is $C^k$. 
The \emph{induction hypothesis} is that ${\bf m}(X_{\prec i , k}) = 1$ and that for each $e$, each $\vec b$, each $\vec h$ from 
$\{f_{i',j }\mid i' \prec i \wedge j \in \N\}$ , each $k \in \N$ and each $\vec g \in X_{\prec  i , k}$ we have that if there is any extension $[f']$ of $[f]_{\prec i}$ such that 
$\{e\}_{[f']}(\Lambda_{[f']},\mu , \vec h , \vec g, \vec b) \!\!\downarrow$ with a computation of ordinal rank less than $\alpha_i $, then $\{e\}_{[f]_{\prec i}}(\Lambda_{[f]_{\prec i}},\mu, \vec h , \vec g, \vec b) \downarrow$.

\medskip

Since at the end, we use the recursion theorem for $\mu$, we also assume that $X_{\prec i , k}$ is hyperartithmetical, with an index computable from $\mu$, $i$ and $k$.

\medskip

Firstly, we construct sets of measure 1 dealing with each of the following cases:
\[
\{e\}(\Lambda,\mu,  \vec h,(f_i)_{\vec c}, \vec g, \vec b),
\]
where $e$ is a fixed Kleene-index, $\vec b$ is a fixed input of integers, $\vec h$ is as above , $\vec c$ is a  sequence  of length $k'$ and $\vec g \in C^k$.  Recall that each such case covers countably many cases by Convention \ref{con.1}.
Then we let $Z_{i,k}$ be the intersection of the sets constructed for each of the cases.
The purpose of $\vec c$ is to specify which elements in the sequence $(f_i)$ we are about to construct, will be used as arguments in the computation \emph{without specifying $(f_i)$ itself}.  
All together, there are only \emph{countably} many cases, so our set $Z_{i,k}$ will also have measure 1. The constructions are quite explicit and the induction hypothesis readily implies that $Z_{i,k}$ is hyperarithmetical.

\medskip

We now show what to do in the two cases of composition and application of $\Lambda$; the rest of the cases are trivial, or they follow the pattern of `S4 - composition'. 
We first treat the scheme S4 as follows: let $\vec h$ and $\vec c$ be as above and consider the case 
\[
\{e\}(\Lambda,\mu , \vec h, (f_i)_{\vec c}, \vec g, \vec b) = \{e_1\}(\Lambda, \mu  , \vec h,(f)_{\vec c},\vec g,  \{e_2\}(\Lambda,\mu, \vec h, (f)_{\vec c}, \vec g, \vec b),\vec b) .
\]
We need to find a set of pairs $\langle (f), \vec g\rangle$ of measure 1 that guarantees that the induction step for this case goes through.
Now, by Proposition \ref{prop.measure}.(ii) and the induction hypothesis, the set of $\langle (f),\vec g\rangle$ such that $(f)_{\vec c},\vec g \in X_{\prec i , k' + k}$ has measure 1. 
Choose $(f)$ and $\vec g$ in this set, and let $[f']$ be any extension of $[f]_{\prec i}(f)$, where we add the sequence $(f)$ to the end of the double sequence $[f]_{\prec i}$.
If we have
\[
\{e\}_{[f']}(\Lambda_{[f']},{\mu},  \vec h, (f)_{\vec c}, \vec g, \vec b) =\] \[\{e_1\}_{[f']}(\Lambda_{[f']}, {\mu} , \vec h,(f)_{\vec c},\vec g,  \{e_2\}_{[f']}(\Lambda_{[f']},{\mu}, \vec h, (f)_{\vec c}, \vec g, \vec b), \vec b) = a
\]
via a computation of ordinal rank at most $\alpha_i$, then $\{e_2\}_{[f']}(\Lambda_{[f']},{\mu}, \vec h, (f)_{\vec c}, \vec g, \vec b) = c$
for some $c$, and also $\{e_1\}_{[f']}(\Lambda_{[f']}, {\mu} ,  \vec h,(f)_{\vec c}, \vec g,c,\vec b ) = a,$
both with computational ranks \emph{strictly} below $\alpha_i$.
Then, since $(f)_{\vec c},\vec g \in X_{\prec i,k'+k}$, we can apply the induction hypothesis and conclude that 
\[
\{e_2\}_{[f]_{\prec i}}(\Lambda_{[f]_{\prec i}},{\mu}, \vec h, (f)_{\vec c}, \vec g, \vec b) = c \textup{ and } \{e_1\}_{[f]_{\prec i}}(\Lambda_{[f]_{\prec i}}, {\mu} ,\vec h,(f)_{\vec c}, \vec g , c,\vec b) = a.
\]
Thus, with any choice of $(f)$ as $(f_i)$ and $\vec g$ as above, we have 
\[
\{e\}_{[f]_i}(\Lambda_{[f]_i}, {\mu} ,\vec h,(f_i)_{\vec c}, \vec g , \vec b) = a,
\]
 as required for this case.  

\medskip
\noindent
We now turn to the cases with application of $\Lambda$, i.e. computations of the form 
\[
\Lambda(\lambda g. \{e\}(\Lambda, {\mu} , \vec h , g, (f)_{\vec c}, \vec g,\vec b)).
\]
As before, we see that the set of $\langle (f),g,\vec g\rangle$ such that $g,(f)_{\vec c},\vec g \in X_{\prec i , 1+k'+k}$ has measure 1.

\medskip

We now define  $F_{(f),\vec g}(g) := \{e\}_{[f]_{\prec i}}(\Lambda_{[f]_{\prec i}} , {\mu} , \vec h , g , (f)_{\vec c} , \vec g,\vec b)$ provided this computation terminates with ordinal rank $< \alpha_i$. 
We claim  that the following three sets all have measure 1:  
\begin{enumerate}
\renewcommand{\theenumi}{\roman{enumi}}
\item The set of $\langle (f),\vec g \rangle$ such that   \label{telemachus}
\[
\text{${\bf m}(\{g \mid  (f)_{\vec c},g,\vec g \in X_{\prec i , 1+k'+k}\}) = 1$ and $(f)_{\vec c} , f_{j}, \vec g \in X_{\prec i , 1+k'+k}\; {\rm for}\;{\rm all}\; j \in \N$}.
\]
\item The set of $\langle (f), \vec g \rangle$ such that $(f)_{\vec c},\vec g \in X_{\prec i , k' + k}$.\label{teefje}
\item The set of  $\langle (f),\vec g\rangle$ such that either (a) or (b) holds, as follows:\label{lop}
\begin{enumerate}
\item[(a)] The domain of $F_{(f),\vec g}$ has measure 1 and $(f) $ is sufficient for $F_{(f),\vec g}$.
\item[(b)] The domain of $F_{(f),\vec g}$ has measure  $< 1$ and $(f)$ fails $F_{(f),\vec g}$ via an $f_j$ such that the sequence $(f)_{\vec c} ,f_{j}, \vec g$ is in $X_{\prec i,1+k'+k}$.
\end{enumerate}
\end{enumerate}
For item \eqref{telemachus}, we use the second item of Proposition \ref{prop.measure} and the fact that a countable product of sets of measure 1 will have measure 1. Item \eqref{teefje} is a consequence of the second item of Proposition \ref{prop.measure} and item \eqref{lop} is a consequence of Lemma \ref{1.22}.

\medskip

We now consider $\langle (f),\vec g\rangle$ in the intersection of these three sets.  
Let $[f']$ be an extension of $[f]_{\prec i}(f)$. Assume that $\Lambda_{[f']}\big(\lambda g \{e\}_{|f']}(\Lambda_{[f']},{\mu} , \vec h, g,(f_a)_{\vec c} , \vec g,\vec b)\big)$ terminates with a computation of ordinal rank at most $\alpha_i $. 

\medskip

First assume that $\Lambda_{[f']}(F_{(f),\vec g}) = (f_{i'})$ for some $i'  \prec i$. Then, for all $i'' \preceq i' $and all $j$ we have that
\[
\{e\}_{[f']}(\Lambda_{[f']},\mu , f_{i'',j},(f)_{\vec c} ,\vec g,\vec b) \downarrow
\]
by a computation of ordinal rank below $\alpha_i$. Since $(f)_{\vec c},\vec g \in  X_{\prec i , k' + k}$, we may apply  the induction hypothesis, and see that 
$\{e\}_{[f]_{\prec i}}(\Lambda_{[f]_{\prec i}},\mu , f_{i'',j},(f)_{\vec c} ,\vec g,\vec b) \downarrow$ by the same computation. 
Then by a computation of ordinal rank not exceeding $\alpha_i$:
\[
\Lambda_{[f]_{\prec i}(f)}(\lambda g. \{e\}_{[f]_{\prec i}}(\Lambda_{[f]_{\prec i}}, {\mu}, \vec h , g, (f)_{\vec c}, \vec g,\vec b )) = (f_j)
\]
This is the one spot where we need extra parameters from $[f]_{\prec i}$, in his case  $f_{i'',j}$, when we formulate the properties of $\vec g$ used at step $\alpha_i$. 
Since $f_{i'',j}$ may already be in $\vec h$, this is also the spot where we need Convention \ref{con.1}.

\medskip
\noindent
Secondly, suppose that the assumption from the previous paragraph is not the case. By the argument above, we then have the following:
\[
\{e\}_{[f]_{\prec i}}(\Lambda_{[f]_{\prec i}},\mu , f_{i',j},(f)_{\vec c} ,\vec g,\vec b) \downarrow
\]
 for all $i' \prec i$ and $j \in \N$.
There are two sub-cases to consider:
\begin{enumerate}
  \renewcommand{\theenumi}{\roman{enumi}}
\item If the domain of $F_{(f),\vec g}$ has measure 1, we get that $\Lambda_{[f]_{\prec i}(f)}(F_{(f), \vec g}) = (f)$ since $(f)$ is the first single sequence in the double sequence $[f]_{\prec i}(f)$ that is sufficient for $F_{(f),\vec g}$. By the definition of $F_{(f),\vec g}$, this observation verifies the induction step in this case.
\item If the domain of $F_{(f),\vec g}$ has measure $ < 1$, there is one $f_{j}$ for which $F_{(f), \vec g}$ does not terminate. In light of item \eqref{telemachus}, we have that $(f)_{\vec c},f_{j} , \vec g$ is in $X_{\prec i , 1+k'+k}$, and using the induction hypothesis negatively, we see that $\{e\}(\Lambda_{[f']},{\mu} , \vec h , f_{j},\vec g , \vec b)$ does not terminate before $\alpha_i$. Since this value is required for the $\Lambda_{[f']}$-computation in question to terminate at all, the latter cannot terminate at stage $\alpha_i$ or earlier.   
 \end{enumerate}
 We are now through all cases, i.e.\ the proof of Lemma \ref{lemma.ind} is finished. 
 \end{proof}
 We now let $[f]$ and each $X_{i,k}$ be as constructed above. For each $k\in \N$, we define the intersection $X_k = \bigcap_{i \in W} X_{i,k}$.  
 \begin{lemma}\label{lemma.sup}
 For each $k$ and $\vec g \in X_k$ we have ${\bf m}(\{g \mid g,\vec g \in X_{k+1}\}) = 1$.
 \end{lemma}
 \begin{proof} It suffices to show that ${\bf m}(\{g \mid g,\vec g \in X_{i,k+1}\}) = 1$ for co-finally many $i \in W$,  and it is the requirement in item \eqref{telemachus} from the proof of Lemma \ref{lemma.ind} in the treatment of $\Lambda$-application that does the trick. 
 Let $\vec g \in A_k$ and consider $\{e\}(\Lambda,{\mu} , \vec g) = \Lambda(\lambda g.\{e_1\}(\Lambda,{\mu} , g,\vec g))$ for any $e_1$ of suitable arity.  
 When we treat this case stepping from $X_{\prec i , k}$ to $X_{i,k}$, the aforementioned item \eqref{telemachus} restricts our attention to  $\vec g$ additionally satisfying ${\bf m}(\{g \mid g,\vec g \in X_{\prec a , k+1}\}) = 1$.
 In the limit, this required property thus holds. 
 \end{proof}
 Let $\Lambda_0$ be any weak fan functional, and let $[f]$ be as constructed in the proof of Lemma \ref{lemma.ind}. We define $\Lambda_{1}$ as follows:
 \be\label{lambda1337}
 \Lambda_{1}(F) = \left \{ \begin{array}{llc} \Lambda_{[f]}(F)& {\rm if\;defined}\\\Lambda_0(F)& {\rm otherwise}\end{array}\right.
 \ee
 and prove our main theorem as follows. 
\begin{theorem}\label{theorem.lambda}
If $f:\N\di \N$ is computable in $\Lambda_{1}+\mu$, then it is computable in ${\mu}$.
\end{theorem}
\begin{proof}
We will prove the stronger claim \eqref{calem} below by induction on the length of the computation.  We need some notation as follows.
Let $e$ be a Kleene index, let $\vec b$ be a sequence from $\N$ and let $\vec g$ of length $k$ be a sequence from $\bigcap_{i \in W}X_{i,k}$ such that $\omega_1^{\textup{\textsf{CK}}} = \omega_1^{{\CK} ,\vec g}$.  
By Proposition \ref{prop1}.(iv), the final restriction does not alter the measure of the set.
Now consider the claim:
\be\label{calem}
\{e\}(\Lambda_{1} , {\mu} , \vec g, \vec b ) = a\di (\exists i\in W) (\{e\}_{[f]_i}(\Lambda_{[f]_i},{\mu} , \vec g,\vec b ) = a).
\ee 
The theorem follows from the claim \eqref{calem} and the total instances $\lambda c. \{e\}(\Lambda_{\mu},c)$. 
 
\medskip 
 
We now prove the claim \eqref{calem} by induction on the ordinal rank of the computation $\{e\}(\Lambda_{1} , {\mu}  , \vec g , \vec b) = a.$ 
The proof is split into cases according to which Kleene scheme $e$ represents, and all cases except those for application of ${\mu}$ or $\Lambda_{1}$ are trivial. 
We will consider the two cases \eqref{cake1} and \eqref{cake2}.  First, we consider
\be\label{cake1}
\{e\}(\Lambda_{1} , {\mu} , \vec g, \vec b) = {\mu}(\lambda c.\{e_1\}(\Lambda_{1} , \mu , \vec g, c, \vec b)).
\ee 
Then, by the induction hypothesis, we have the following termination property:
\[
( \forall c \in\N) (\exists i \in W) \big[\{e_1\}_{[f]_i}(\Lambda_{[f]_i} , \vec g, c,\vec b)\!\!\downarrow \big].
\]
Since $\omega_1^{\textup{\textsf{CK}}}$ is $\Sigma_1$-admissible relative to $\vec g$ (see footnote \ref{admi}), there is a bound on how far out in $W$ we need to go, i.e.\ $(\exists i \in W) (\forall c \in \N )[\{e_1\}a_{[f]_{\prec i}}(\Lambda_{[f]_{\prec i}}, \vec g,c,\vec b)\!\!\downarrow]$, 
and $\{e\}_{[f]_i}\big(\Lambda_{[f_i]} , {\mu} , \vec h , \vec g,  \vec b \big)\!\!\downarrow$ follows.

\medskip
\noindent
For the second case, consider the following (involving a slight abuse of notation):
\be\label{cake2}
\{e\}(\Lambda_{1},{\mu} , \vec g, \vec b) = \Lambda_{1}(\lambda g. \{e_1\}(\Lambda_{1} , \mu  , g , \vec g,\vec b)).
\ee 
Since this is a classically valid Kleene computation, we have that $\lambda g. \{e_1\}(\Lambda_{1} , \mu  , g , \vec g,\vec b)$ is total. By Lemma \ref{lemma.sup} and the induction hypothesis, for almost all $g$ there is an $i_g \in W$ such that $\{e_1\}_{[f]_{{i_g}}}\big(\Lambda_{[f]_{i_g}},{\mu}  , g , \vec g,\vec b\big)\!\! \downarrow$.
Now consider the sequence 
\[
i \mapsto {\bf m}\big(\{g \mid \{e_1\}_{[f]_i}(\Lambda_{[f]_i} , {\mu}  ,g, \vec g, \vec b)\!\!\downarrow\}\big).
\]
This sequence is increasing, computable in $\vec g$ and ${\mu}$, and has limit 1, implying that
\[
(\forall k) (\exists i_k \in W)\big( {\bf m}(\{g \mid \{e_1\}_{[f]_{i_k}}(\Lambda_{[f]_{i_k}} , {\mu}  ,g, \vec g, \vec b)\!\!\downarrow\}) >1 -  2^{-k}\big).
\]
Hence, by the fact that $\omega_1^{\textup{\textsf{CK}}} = \omega_1^{\textup{\textsf{CK}},\vec g}$ and is $\Sigma_1$ -admissible in $\vec g$, we see that there must be $i \in W$ such that 
\[
{\bf m}(\{g \mid \{e_1\}_{[f]_i}(\Lambda_{[f]_i }, {\mu} ,g, \vec g , \vec b)\!\!\downarrow\}) = 1. 
\] 
In this light, our construction guarantees that $(f_i)$ is sufficient for $\lambda g. \{e_1\}(\Lambda_{1} , {\mu},  g , \vec g,\vec b )$. Unless some $(f_{i'})$ already does the job for $i' \prec i$,
we may conclude that 
\[ 
\{e\}(\Lambda_{1},{\mu}   , \vec g, \vec b) = (f_i) = \{e\}_{[f]_i}( \Lambda_{[f]_i},{\mu}  , \vec g, \vec b).
\] 
This ends the induction step, and we are done.  
\end{proof}

\section{Reverse Mathematics of the special fan functional}\label{akihiro}
We show how $\Theta$ (and its generalisations) can reach the current outer edge of RM (and its higher-order generalisation).

\medskip

First of all, in Section~\ref{SIXTUS}, we investigate the strength of the combination of $\Theta$ and $S^{2}$, which will be seen to reach the current \emph{upper limit} of RM.  
Indeed, we have shown in \cite{dagsamIII} that the combination $S^{2}+\Theta$ computes Gandy's \emph{Superjump}, a functional intimately connected to $\FIVEFIVE$.  
As a complementary result, we show in Section~\ref{SIXTUS} that the system $\FIVE^{\omega}+\QFAC^{2,1}+\HBU$ behaves as follows: (i) it implies $\FIVEFIVE$, and (ii) it proves the same $\Pi_{3}^{1}$-sentences as $\SIX$.  
To establish these results, we derive $[\SIX]^{\st}$ in $\P_{0}+\Paai+\STP$.  

\medskip

Secondly, $\Theta$, $\STP$, and $\HBU$ express the compactness of Cantor space and the unit interval (in various forms).  Since the compactness of \emph{function spaces} is 
essential to the study of the gauge integral (see e.g.\ \cite{mullingitover, secondmulling}), it is a natural question how strong such compactness properties are. 
As a first step, we study in Section \ref{schweber} the strength of such a compactness property inspired by $\STP$.
In particular, we formulate a generalisation of Theorems~\ref{mikeh} and \ref{stokeheo} to higher types suggested by \cites{schtreber,schtreberphd}. 
As a result, the compactness of function spaces seems quite strong from the point of view of RM.  

\subsection{At the limit of Reverse Mathematics}\label{SIXTUS}
In this section, we derive $[\SIX]^{\st}$ in $\P_{0}+\Paai+\STP$, which is a result similar to Theorem \ref{mikeh}.  
We obtain interesting corollaries involving $\FIVEFIVE$ and $\SIX$.  
We first discuss some of the history of $\SIX$ and related systems.

\medskip

The system $\SIX$ appears in the study of the Reverse Mathematics of topology by Mummert and Simpson (\cite{mummy}), who identify this system as the `current limit' of RM.
The coding used by Mummert and Simpson is however not unproblematic, as discussed by Hunter (\cite{hunterphd}).  
Furthermore, it is known that $\SIX$ is equivalent to $\Sigma_{2}^{1}\textsf{-DC}_{0}$ and $\Sigma_{2}^{1}\mSEP_{0}$ by \cite{simpson2}*{VII.6.9 and VII.6.14}.  

\medskip

To the best of our knowledge, $\SIX$ is also the current limit of \emph{ordinal analysis}; according to Rathjen (\cite{rathjenICM}), the strength of $\SIX$ \emph{dwarfs} that of $\FIVE$.  
By the following theorem and its corollaries, $\STP$ and $\HBU$ are however all that is needed to step from the latter system to the former (in various guises).  
\begin{thm}\label{stokeheo}
The system $\P_{0}+\Paai+\STP$ proves $[\SIX]^{\st}$.
\end{thm}
\begin{proof}
As noted in \cite{simpson2}*{VII.6.14}, $\ACA_{0}$ proves that $\SIX$ is equivalent to $\Sigma_{2}^{1}\mSEP$, where the latter is: For $\varphi_{1}, \varphi_{2}\in \Sigma_{2}^{1}$ not involving the variable $Z^{1}$,
\be\label{heffer2}
(\forall n^{0})(\neg\varphi_{1}(n)\vee \neg\varphi_{2}(n))\di (\exists Z^{1})(\forall n^{0})\big[\varphi_{1}(n)\di n\in Z\wedge \varphi_{2}(n)\di n\not\in Z \big].
\ee
We shall prove $[\Sigma_{2}^{1}\mSEP]^{\st}$ in $\P_{0}+\Paai+\STP$.  Since $\P_{0}+\paai$ proves the axioms of $\ACA_{0}$ relative to `st', we obtain $[\SIX]^{\st}$.  

\medskip

Let $\varphi_{i}(n)$ be short for the formula $(\exists g^{1}_{i})(\forall h^{1}_{i})(\exists x_{i}^{0})(f_{i}(\overline{h_{i}}x_{i}, \overline{g_{i}}x_{i}, n)=0)$ and fix standard $f_{i}^{1}$ for $i=1,2$.  Then assume $\big[(\forall n^{0})(\neg\varphi_{1}(n)\vee \neg\varphi_{2}(n))\big]^{\st}$, which is
\begin{align*}
(\forall^{\st}n^{0})
\big[
(\forall^{\st} g^{1}_{1})(\exists^{\st}h_{1}^{1})&(\forall^{\st} x_{1}^{0})(f_{1}(\overline{h_{1}}x_{1},\overline{g_{1}}x_{1}, n)\ne 0) \\
&\vee  
(\forall^{\st} g^{1}_{2})(\exists^{\st}h_{2}^{1})(\forall^{\st} x_{2}^{0})(f_{2}(\overline{h_{2}}x_{2},\overline{g_{2}}x_{2}, n)\ne 0)
\big].
\end{align*}
Using $(\mu_{1})^{\st}$, which follows\footnote{In the definition of $\Paai$, bring outside the standard quantifiers and apply $\HAC_{\INT}$. Introduce standard quantifiers in the antecedent using $\paai$ to obtain $(\mu_{1})^{\st}$.} from $\Paai$, the previous formula implies that:
\begin{align}
(\forall^{\st}n^{0}, g_{1}^{1}, g_{2}^{1})
\big[~~\(\forall^{\st}  x_{1}^{0})(f_{1}&\big(~\overline{\mu_{1}(\lambda \sigma_{1}.f_{1})}x_{1},\overline{g_{1}}x_{1}, n\big)\ne 0)\label{lahiel2} \\
&\vee  
(\forall^{\st} x_{2}^{0})(f_{2}\big(~\overline{\mu_{1}(\lambda \sigma_{2}.f_{2})}x_{2},\overline{g_{2}}x_{2}, n\big)\ne 0)\notag
\big],
\end{align}
where we suppressed parameters, as the `full' notation of $\lambda\sigma_{i}.f_{i}$ is $\lambda \sigma^{0^{*}}_{i}.f_{i}(\sigma_{i}, \overline{g_{i}}x_{i}, n)$.  Now fix nonstandard $N^{0}$ and apply $\paai$ to \eqref{lahiel2} to obtain: 
\begin{align}
(\forall^{\st}n^{0}, g_{1}^{1}, g_{2}^{1})
\big[~~\(\forall  x_{1}^{0}\leq N)(f_{1}&\big(~\overline{\mu_{1}(\lambda \sigma_{1}.f_{1})}x_{1},\overline{g_{1}}x_{1}, n\big)\ne 0)\label{lahiel22} \\
&\vee  
(\forall x_{2}^{0}\leq N)(f_{2}\big(~\overline{\mu_{1}(\lambda \sigma_{2}.f_{2})}x_{2},\overline{g_{2}}x_{2}, n\big)\ne 0)\notag~
\big].
\end{align}
Now let $A_{i}(n, g_{i})$ be the (equivalent to quantifier-free) following formula 
\[
(\forall x_{i}^{0}\leq N)(f_{i}(~\overline{\mu_{1}(\lambda \sigma_{i}.f_{i})}x_{i},\overline{g_{i}}x_{i}, n)\ne 0), 
\]
and let $A(n, g_{1}, g_{2})$ be the formula $A_{1}(n, g_{1})\vee A_{2}(n, g_{2})$, i.e.\ the formula in square brackets in \eqref{lahiel22}.  By assumption, $(\forall^{\st}n^{0}, g_{1}^{1}, g_{2}^{1})A(n,g_{1},g_{2})$.  
Now consider:
\begin{align}\label{ideaal2}
(\forall^{\st} v^{1^{*}}, x^{0^{*}})(\exists & w^{1^{*}}, y^{0^{*}})(\forall g^{1} \in v, n^{0}\in x)\\
&\big[ g\in w \wedge n\in y \wedge (\forall h_{1}, h_{2}\in w, m\in y)A(m,h_{1}, h_{2}) \big].\notag
\end{align}
Note that \eqref{ideaal2} holds by taking $w=v$ and $y=x$.  Applying \textsf{I} to \eqref{ideaal2}, we obtain  
\be\label{forgik2}
(\exists w^{1^{*}}, y^{0^{*}})(\forall^{\st} g^{1}, n^{0})\big[ g\in w \wedge n\in y \wedge (\forall h_{1}, h_{2}\in w, m\in y)A(m,h_{1}, h_{2}) \big], 
\ee
which -intuitively speaking- provides two sequences $w, y$ (of nonstandard length) encompassing all standard functions and standard numbers \emph{and} such that all of its elements satisfy $A$.  
In particular, one can view \eqref{forgik2} as obtained by applying overspill to \eqref{lahiel22} while making sure all standard functions are in $w$.  

\medskip

Next, define the set $Z_{0}^{1}$ (actually a binary sequence) as follows:  $n\in Z_{0}\asa (\exists g_{1}\in w)\neg A_{1}(n,g)$, where $w$ is the sequence from \eqref{forgik2}.  
Note that the right-hand side of the equivalence is actually `$(\exists i^{0}<|w|)\neg A_{1}(n, w(i))$', i.e.\ $Z_{0}$ is definable in $\P_{0}$.

\medskip

Let $Z^{1}$ be a standard set such that $Z_{0}\approx_{1} Z$  as provided by $\STP$.  Furthermore, since $\mu_{1}$ is standard, we have the following implications (for standard $n$):
\begin{align*}
(\exists^{\st} g^{1}_{1})(\forall^{\st}h_{1}^{1})(\exists^{\st} x_{1}^{0})&(f_{1}(\overline{h_{1}}x_{1},\overline{g_{1}}x_{1}, n)=0) \\
&\di (\exists^{\st} g^{1}_{1})(\exists^{\st} x_{1}^{0})(f_{1}(\overline{\mu_{1}(\lambda \sigma_{1}.f_{1})}x_{1},\overline{g_{1}}x_{1}, n)=0) \\
&\di (\exists g^{1}_{1}\in w)(\exists x_{1}^{0}\leq N)(f_{1}(\overline{\mu_{1}(\lambda \sigma_{1}.f_{1})}x_{1},\overline{g_{1}}x_{1}, n)=0) \\
&\di (\exists g^{1}_{1}\in w)\neg A_{1}(n,g_{1})\di n\in Z_{0}\di n\in Z.
\end{align*}
Now, since $y$ from \eqref{forgik2} contains all standard numbers, the second conjunct of \eqref{forgik2} implies (by definition)
that for standard $m$ (by the definition of $A$): 
\be\label{fronxi2}
(\forall h_{1}\in w)A_{1}(m,h_{1})\vee (\forall h_{2}\in w)A_{2}(m,h_{2}).
\ee
Similarly, consider the following series of implications (for standard $n$):
\begin{align}
(\exists^{\st} g^{1}_{2})(\forall^{\st}h_{2}^{1})(\exists^{\st} x_{2}^{0})&(f_{2}(\overline{h_{2}}x_{2},\overline{g_{2}}x_{2}, n)=0) \notag\\
&\di (\exists^{\st} g^{1}_{2})(\exists^{\st} x_{2}^{0})(f_{2}(\overline{\mu_{1}(\lambda \sigma_{2}.f_{2})}x_{2},\overline{g_{2}}x_{2}, n)=0) \notag\\
&\di (\exists g^{1}_{2}\in w)(\exists x_{2}^{0}\leq N)(f_{2}(\overline{\mu_{1}(\lambda \sigma_{2}.f_{2})}x_{2},\overline{g_{2}}x_{2}, n)=0) \notag\\
&\di (\exists g^{1}_{2}\in w)\neg A_{2}(n,g_{2})\label{hoi12}\\
&\di (\forall g^{1}_{1}\in w) A_{1}(n,g_{1})\label{hoi22}\\
&\di n\not\in Z_{0}\di n\not\in Z.\notag
\end{align}
Note that \eqref{hoi22} follows from \eqref{hoi12} by \eqref{fronxi2}.  
Thus, $Z$ is as required for [$\Sigma_{2}^{1}\mSEP]^{\st}$.  
\end{proof}
The following corollary was proved in \cite{dagsam} by using the fact that no type two functional (hence including $\mu_{1}$) can compute an instance of $\Theta$.  
Hence, we observe that the computability-theoretic approach `scales' better than our above approach via Nonstandard Analysis, but the latter may be called `conceptually simpler'.  
\begin{cor}
The system $\P_{0}+\Paai$ cannot prove $\STP$.
\end{cor}
\begin{proof}
The system $\EPRA^{\omega}+(\mu_{1})$ is a $\Pi_{3}^{1}$-conservative extension of $\FIVE$ by \cite{yamayamaharehare}*{Theorem 2.2}.  
Furthermore, let $\varphi$ be an arithmetical sentence (resp.\ not) provable in $\SIX$ (resp.\ $\FIVE$).  
Suppose $\P_{0}+\Paai\vdash \STP$ and note that $\P_{0}+\Paai\vdash \varphi$ by the theorem (and the fact that $\varphi\asa \varphi^{\st}$ given $\paai$).
Since $\Paai$ is converted into $(\mu_{1})$ by term extraction, we obtain $\RCAo+(\mu_{1})\vdash \varphi$, a contradiction with the aforementioned conservation result for $(\mu_{1})$. 
\end{proof}
To be absolutely clear, we now discuss what does, and more importantly, \emph{what does not} follow from Theorem \ref{stokeheo}.
\begin{rem}\label{exul2}\rm
First of all, one of the main consequences of the \emph{Transfer} axiom of $\IST$ is the equivalence  $\varphi\asa \varphi^{\st}$ (for any internal $\varphi$ with standard parameters).  
In the absence of the \emph{full} axiom of \emph{Transfer}, as is the case for e.g.\ the system in Theorem~\ref{stokeheo}, this equivalence may no longer hold.  
Hence, the system from Theorem ~\ref{stokeheo} does \emph{not} necessarily prove $\SIX$.  By contrast, the former system \emph{does} prove the \emph{arithmetical} consequences of $\SIX$, thanks\footnote{For internal and \emph{arithmetical} $\varphi$ with standard parameters,  $\P_{0}+\paai\vdash [\varphi\asa \varphi^{\st}]$.} to $\paai$.   

\medskip

Secondly, an interesting corollary of Theorem \ref{mikeh} is that $(\mu^{2})+(\exists \Theta)\SFF(\Theta)$ implies $\ATR_{0}$ over $\RCAo$ (see Theorem \ref{frigjr}).  
To obtain this corollary, one observes that $\ATR_{0}^{\st}$ implies (using $\paai$) the following normal form:
\be\label{desnol}
(\forall^{\st} X^{1}, f^{1})(\exists^{\st} Y^{1})\big[\WO(X)\di H_{f}(X, Y) \big].
\ee
One then applies term extraction to $\P_{0}+\paai+\STP\vdash \eqref{desnol}$; omitting the extracted term, one obtains that $[(\mu^{2})+(\exists \Theta)\SFF(\Theta)]\di \ATR_{0}$ over $\RCAo$.  
However, we can only obtain the latter implication \emph{because} $\ATR_{0}^{\st}$ implies an equivalent normal form, namely \eqref{desnol}, which is \emph{highly similar} to $\ATR_{0}$ itself.  The existence of such a `highly similar' normal form (given a relatively weak system) is exceptional in that e.g.\ $[\WKL]^{\st}$, $[\Sigma_{1}^{1}\textsf{-SEP}]^{\st}$, and $[\SIX]^{\st}$ do not\footnote{Let $\WKL_{\ns}$ be the statement that a \emph{standard} and infinite binary tree has a \emph{standard} path if the former contains sequences of arbitrary length.  
Then $\P_{0}+\WKL^{\st}$ (resp.\ $\P_{0}+\WKL_{\ns}$) has the proof-theoretic strength of $\WKL_{0}$ (resp.\ $\ACA_{0}$), i.e.\ $\WKL^{\st}\not\asa \WKL_{\ns}$ over $\P_{0}$.} have them, to the best of our knowledge.  More generally, applying a proof interpretation (on which term extraction as in Theorem \ref{consresultcor} is based) to the proof of a theorem, tends to completely warp the latter, i.e.\ $\ATR_{0}^{\st}$ is the exception, not the rule.
\end{rem}
In light of Remark \ref{exul2}, it seems the system from Theorem \ref{stokeheo} cannot prove $\SIX$; we now derive `the next best thing' $\FIVEFIVE$ from the result in Theorem~\ref{stokeheo}.  
\begin{cor}\label{csdf}
The system $\RCAo+\QFAC^{2,1}+(\mu_{1})+\HBU$ proves $\FIVEFIVE$.
\end{cor}
\begin{proof}
Note that $\Pi_{2}^{1}\textsf{-SEP}\asa \FIVEFIVE$ over $\ACA_{0}$ by \cite{simpson2}*{VII.6.14}, where $\Pi_{2}^{1}\textsf{-SEP}$ is \eqref{heffer2} for $\varphi_{1}, \varphi_{2}\in \Pi_{2}^{1}$.   
By Theorem \ref{stokeheo}, $\P_{0}+\Paai+\STP$ proves $[\Pi_{2}^{1}\textsf{-SEP}]^{\st}$.  
The antecedent of the latter has the form $(\forall^{\st} n^{0})(\exists^{\st} g^{1})(\forall^{\st} h^{1})\varphi^{\st}(n, g, h)$, where $\varphi^{\st}$ is arithmetical.  
Hence, the antecedent in $[\Pi_{2}^{1}\textsf{-SEP}]^{\st}$ may be strengthened to 
\be\label{ankorage}
(\exists^{\st}\Phi^{0\di 1^{*}})(\forall n^{0})(\exists  g^{1}\in \Phi(n))(\forall h^{1})\varphi(n, g, h)
\ee
using $\paai$.   
On the other hand, the consequent of $[\Pi_{2}^{1}\textsf{-SEP}]^{\st}$ has the form 
\be\label{karellen}
(\exists^{\st}Z^{1})(\forall^{\st} n^{0})(\exists^{\st} g^{1})\underline{(\forall^{\st} h^{1})\psi^{\st}(n, g, h, Z)}, 
\ee
where $\psi^{\st}$ is arithmetical.  Now apply $\Paai$ (which readily follows from $(\exists^{\st}\mu_{1})\MUO(\mu_{1})$) to the underlined formula in \eqref{karellen}.  In the resulting formula, apply $\HAC_{\INT}$ to obtain a standard functional $\Phi^{0\di 1^{*}}$ such that:
\be\label{karellencor}
(\exists^{\st}Z^{1})(\forall^{\st} n^{0})(\exists g^{1}\in \Phi(n))\underline{(\forall h^{1})\psi(n, g, h, Z)}, 
\ee
Now apply $\paai$ to the `$(\forall^{\st}n^{0})$' quantifier in \eqref{karellencor}; note that $(\exists^{\st}\mu_{1})\MUO(\mu_{1})$ guarantees that the formula following the `$(\forall^{\st}n^{0})$' quantifier is equivalent to a quantifier-free one.
Thus, we obtain:
\be\label{ankorage2}
(\exists^{\st}\Psi^{0\di 1^{*}}, Z^{1} )(\forall n^{0})(\exists  g^{1}\in \Psi(n))(\forall h^{1})\psi(n, g, h, Z)
\ee
using $(\exists^{\st}\mu_{1})\MUO(\mu_{1})$ and $\HAC_{\INT}$.  Now apply term extraction to 
\[
\P_{0}+(\exists^{\st}\mu_{1})\MUO(\mu_{1})+\STP\vdash [\eqref{ankorage}\di \eqref{ankorage2}]
\]
and omit all terms.  Finally note that $(\mu_{1})+\QFAC^{0,1}$ yields $\Phi^{0\di 1^{*}}$ satisfying $(\forall n^{0})(\exists  g^{1}\in \Phi(n))(\forall h^{1})\varphi(n, g, h)$ from $(\forall n^{0})(\exists  g^{1})\underline{(\forall h^{1})\varphi(n, g, h)}$, as the underlined formula may be treated as quantifier-free. 
One thus obtains that $\RCAo+\QFAC^{0,1}+(\mu_{1})+(\exists \Theta)\SFF(\Theta)$ proves $\Pi_{2}^{1}\textsf{-SEP}$, and hence $\FIVEFIVE$ as discussed above.  
\end{proof}
If one repeats the previous proof for $[\Sigma_{2}^{1}\textsf{-SEP}]^{\st}$ (instead of $[\Pi_{2}^{1}\textsf{-SEP}]^{\st}$), one will observe that \emph{Transfer} for $\Pi_{2}^{1}$-formulas
seems needed to treat the consequent of $[\Sigma_{2}^{1}\textsf{-SEP}]^{\st}$ in the same way as in the previous proof.  However, this instance of \emph{Transfer} of course yields $\SIX$ after term extraction.  
In other words, the system from the corollary does not imply $\SIX$ \emph{using the same proof}.  We do obtain the following corollary where $\Pi_{1}^{1}\TRO$ is transfinite recursion for $\Pi_{1}^{1}$-formulas (see \cite{simpson2}*{VI.7.1}), i.e.\ $\ATR_{\theta}$ from Section \ref{desnol} for any $\theta\in \Pi_{1}^{1}$.
\begin{cor}
The system $\RCAo+(\mu_{1})+(\exists \Theta)\SFF(\Theta)$ proves $\Pi_{1}^{1}\TRO$.
\end{cor}
\begin{proof}
It is known that $\SIX$ implies $\Pi_{1}^{1}\TRO$ (see e.g.\ \cite{simpson2}*{VII.7.12}).  
By Theorem \ref{stokeheo}, $\P_{0}+\Paai+\STP$ proves $[\Pi_{1}^{1}\TRO]^{\st}$.  
Similar to the second part of Remark \ref{exul2}, one observes that $\Pi_{1}^{1}\TRO^{\st}$ implies (using $\Paai$) the following normal form:
\be\label{desnol2}
(\forall^{\st} X^{1})(\exists^{\st} Y^{1})\big[\WO(X)\di H_{\theta}(X, Y) \big], 
\ee
for any fixed $\theta\in \Pi_{1}^{1}$.  
One then applies term extraction to $\P_{0}+\Paai+\STP\vdash \eqref{desnol2}$; omitting the extracted term, one obtains the corollary.
\end{proof}
We now discuss some interesting proof-theoretic corollaries.  
Let $\con(S)$ be the $\Pi_{1}^{0}$-sentence expressing the consistency of $S$ (see \cite{simpson2}*{II.8.2}).   
\begin{cor}\label{poliop}
The system $\RCAo+\QFAC^{2,1}+(\mu_{1})+\HBU$ proves $\con(\FIVE)$; the same holds for any $\Pi_3^1$-sentence provable in $\SIX$.
\end{cor}
\begin{proof}
Since $\SIX\vdash \con(\FIVE)$, the system $\P_{0}+\Paai+\STP$ proves $[\con(\FIVE)]^{\st}$ and applying $\paai$ yields $ \con(\FIVE)$.
Hence, by Theorem~\ref{lapdog}, the stronger system $\P_{0}+(\exists^{\st}\mu_{1})\MUO(\mu_{1})+(\exists^{\st}\Theta)\SFF(\Theta)$ proves $\con(\FIVE)$.  
Applying term extraction as in Theorem \ref{consresultcor}, the corollary follows.  For a $\Pi_{3}^{1}$-sentence $A\equiv (\forall X^{1})(\exists Y^{1})(\forall Z^{1})\varphi(X, Y, Z)$, 
note that $\Paai$ yields $A^{\st}\asa (\forall^{\st} X^{1})(\exists^{\st} Y^{1})(\forall Z^{1})\varphi(X, Y, Z)$.  Hence, if $\SIX\vdash A$, the same proof as for $\con[\FIVE]$ yields that $\RCAo+(\mu_{1})+(\exists \Theta^{3})\SFF(\Theta)$ proves $A$.
\end{proof}
This corollary is interesting as $(\mu_{1})$ yields a conservative extension of $\FIVE$ (see \cite{yamayamaharehare}*{Theorem 2.2}), while $\HBU$ is acceptable in intuitionistic mathematics, and finitistically reducible (in the sense of yielding a conservative extension of $\WKL_{0}$).  
\begin{cor}\label{konkkinkel}
The systems $\P+\Paai+\STP$ and  $\EPA^{\omega*}+(\mu_{1})+\QFAC^{2,1}+\HBU$ prove the consistency of $\SIX$, i.e.\ $\con(\SIX)$.
\end{cor}
\begin{proof}
Note that \cite{simpson2}*{VII.6.21} states $\SIX\equiv_{\Pi^{1}_{3}}\Sigma_{3}^{1}\textsf{-CA}_{0}$ and $\SIX+\Sigma_{3}^{1}\textsf{-IND}\vdash \con(\Sigma_{3}^{1}\textsf{-CA}_{0})$.
Since $\Sigma_{3}^{1}\textsf{-CA}_{0}\di \SIX$ by \cite{simpson2}*{VII.6.6}, the corollary follows.  
\end{proof}
Finally, by way of \emph{mathematical} applications of Corollary \ref{poliop}, the \emph{graph minor theorem} is a $\Pi_{1}^{1}$-sentence provable in $\FIVE+\textsf{BI}$ (\cite{friedrosey}); the latter system is derivable in $\SIX$, yielding the following corollary.
\begin{cor}
$\RCAo+(\mu_{1})+\QFAC^{2,1}+\HBU$ proves the graph minor theorem.  
\end{cor}

\subsection{Generalisations to higher types}\label{schweber}
In this section, we study compactness properties of function spaces.  
In particular, we study a generalisation of Theorems~\ref{mikeh} and \ref{stokeheo} to higher types inspired by \cites{schtreber,schtreberphd}.  
We first discuss the results in the latter and its relation to our results.   We discuss the mathematical naturalness of compactness properties of function spaces in Remark \ref{feynman}.

\medskip

First of all, recall that Theorem \ref{mikeh} was first proved in \cite{dagsam} by proving $[\Sigma_{1}^{1}\mSEP]^{\st}$ in $\P_{0}+\paai+\STP$, where $\Sigma_{1}^{1}\mSEP$ states that for any $\varphi_{1},\varphi_{2}\in \Sigma_{1}^{1}$:
\[
(\forall n^{0})(\neg\varphi_{1}(n)\vee \neg\varphi_{2}(n))\di (\exists Z^{1})(\forall n^{0})\big(\varphi_{1}(n)\di n\in Z\wedge \varphi_{2}(n)\di n\not\in Z \big).
\]
The equivalence $\ATR_{0}\asa \Sigma_{1}^{1}\mSEP$ in \cite{simpson2}*{V.5.1} guarantees that $\ATR_{0}^{\st}$ is provable in $\P_{0}+\paai+\STP$.  
In this section, we study the higher type generalisation of $\P_{0}+\paai+\STP \vdash [\Sigma_{1}^{1}\mSEP]^{\st}$, inspired by results in \cites{schtreberphd, schtreber}, sketched next. 

\medskip

Schweber discusses a higher-order generalisation of the RM of $\ATR_{0}$ in \cites{schtreber,schtreberphd}.  
This generalisation consists in taking theorems from second-order arithmetic and `bumping up all types with one' to obtain a theorem of third-order arithmetic.  
By way of example, compare $\Sigma_{1}^{1}\mSEP$ to the `one level up' separation principle $\Sigma_{1}^{2}\mSEP$ (which is still provable in $\textsf{ZF}$) as follows.
\begin{defi}[$\Sigma^{2}_{1}\mSEP$]
For any $\varphi_{1},\varphi_{2}\in \Sigma^{2}_{1}$, we have that
\[
(\forall f^{1})(\neg\varphi_{1}(f)\vee \neg\varphi_{2}(f))\di (\exists Z^{2})(\forall f^{1})\big(\varphi_{1}(f)\di Z(f)=1 \wedge\varphi_{2}(f)\di Z(f)=0 \big).
\]
\edefi
As noted by Schweber (\cite{schtreber}), $\Sigma_{1}^{2}\mSEP$ implies $\Delta_{1}^{2}$-comprehension, and two determinacy axioms $\Sigma_{1}^{\R}\textsf{-DET}$ and $\Delta_{1}^{\R}\textsf{-DET}$ when combined with the axiom of choice as in $\textsf{SF}(\R)$.  As noted by Hachtman in \cite{schacht, schacht2}, $\Sigma_{1}^{\R}\textsf{-DET}$ is strictly stronger than $\Sigma_{4}^{0}\textsf{-DET}$, and $\Pi_{3}^{0}\textsf{-DET}$ already goes beyond second-order arithmetic (\cite{shoma}*{Cor.~1.3}).  

\medskip

As observed in \cite{schtreber}*{\S1}, many implications in the Reverse Mathematics of $\ATR_{0}$ fail when the theorems are generalised from second-order to third-order arithmetic.  
It is thus a natural question if the implication $ [\paai+\STP]\di [\Sigma_{1}^{1}\mSEP]^{\st}$ generalises to third-order arithmetic.  We answer this question positively as follows:  
The (obvious) generalisation of the system $\P_{0}+\paai+\STP$ to third-order arithmetic is $\P_{0}+\SOT+\STP_{2}$ where the latter axioms are:  
\be\tag{$\SOT$}
(\forall^{\st} Y^{2})\big[ (\exists f^{1})(Y(f)=0)\di  (\exists^{\st} f^{1})(Y(f)=0)  \big],
\ee
\be\tag{$\STP_{2}$}
(\forall Y^{2}\leq_{2}1)(\exists^{\st}Z^{2}\leq_{2}1)(Z\approx_{2}Y), 
\ee  
which are respectively $\paai$ and $\STP$ with all types `bumped up by one'.  
Recall that `$Z\approx_{2}Y$' is $(\forall^{\st} g^{1})(Z(g)=_{0}Y(g))$.
We have the following theorem.  
\begin{thm}\label{mikeh3}
The system $\P_{0}+\SOT+\STP_{2}$ proves $[\Sigma_{1}^{2}\mSEP]^{\st}$.  
\end{thm}
\begin{proof}
Let $\varphi_{i}(f^{1})$ be short for the formula $(\exists Y^{2}_{i})(\forall f_{i}^{1})(\psi_{i}^{3}(Y_{i}, f_{i},f)=0)$ and fix standard $\psi_{i}^{3}$ for $i=1,2$.  
Then assume $\big[(\forall f^{1})(\neg\varphi_{1}(f)\vee \neg\varphi_{2}(f))\big]^{\st}$, which is:
\[
(\forall^{\st}f^{1})
\big[ 
(\forall^{\st} Y^{2}_{1})(\exists^{\st} f_{1}^{1})(\psi_{1}(Y_{1}, f_{1},f)\ne0)
\vee  
(\forall^{\st} Y^{2}_{2})(\exists^{\st} f_{2}^{1})(\psi_{2}(Y_{2}, f_{2},f)\ne0)
\big].
\]
Now fix nonstandard $u^{1^{*}}$ containing all standard sequences (which exists by \emph{Idealisation} \textsf{I}) and note that we have that for all standard $f^{1}, Y_{1}^{2}, Y_{2}^{1}$:
\be\label{lahiel}
(\exists f_{1}^{1}\in u)(\psi_{1}(Y_{1}, f_{1},f)\ne0)
\vee  
(\exists f_{2}^{1}\in u)(\psi_{2}(Y_{2}, f_{2},f)\ne0)
\ee
Let $A_{i}(f, Y_{i})$ be the (equivalent to quantifier-free) formula $(\exists f_{i}^{1}\in u)(\psi_{1}(Y_{i}, f_{i},f)\ne0)$ and let $A(f, Y_{1}, Y_{2})$ be the formula $A_{1}(f, Y_{1})\vee A_{2}(f, Y_{2})$, i.e.\ the formula in \eqref{lahiel}.  By assumption, $(\forall^{\st}f^{1}, Y_{1}^{2}, Y_{2}^{2})A(f,Y_{1},Y_{2})$.  
Now consider:
\begin{align}\label{ideaal}
(\forall^{\st} v^{2^{*}}, x^{1^{*}})(\exists  w^{2^{*}}, y^{1^{*}})&(\forall Y^{2} \in v, f^{1}\in x)\\
&\big[ Y\in w \wedge f\in y \wedge (\forall Y_{1}, Y_{2}\in w, f\in y)A(f,Y_{1}, Y_{2}) \big].\notag
\end{align}
Note that \eqref{ideaal} holds by taking $w=_{2^{*}}v$ and $y=_{1^{*}}x$.  Applying \textsf{I} to \eqref{ideaal} yields
\be\label{forgik}
(\exists w^{2^{*}}, y^{1^{*}})(\forall^{\st} Y^{2}, f^{1})\big[ Y\in w \wedge f\in y \wedge (\forall Y_{1}, Y_{2}\in w, f\in y)A(f,Y_{1}, Y_{2}) \big], 
\ee
which -intuitively speaking- provides two sequences $w, y$ (of nonstandard length) encompassing all standard functionals of type two and standard functions and such that all of its elements satisfy $A$.   In particular, one can view \eqref{forgik} as obtained by applying overspill to \eqref{lahiel} while making sure all standard functionals and functions are in $w$ and $y$.  
Next, define the functional $Z_{0}^{2}$ as follows:  $Z_{0}(f)=0 $ if $ (\exists Y_{1}\in w)\neg A_{1}(f,Y_{1})$ and $1$ otherwise, where $w^{2^{*}}$ is the sequence from \eqref{forgik}.  
Note that $ (\exists Y_{1}\in w)\neg A_{1}(f,Y_{1})$ is actually `$(\exists i^{0}<|w|)\neg A_{1}(f, w(i))$', i.e.\ $Z_{0}^{2}$ is definable in $\P_{0}$.    

\medskip

Let $Z^{2}$ be a standard functional such that $Z_{0}\approx_{2} Z$  as provided by $\STP_{2}$.  Furthermore, $\SOT$ establishes the following implications (for standard $f^{1}$):
\begin{align*}
(\exists^{\st} Y^{2}_{1})(\forall^{\st} f_{1}^{1})(\psi_{1}(Y_{1}, f_{1},f)=0)
&\di (\exists^{\st} Y^{2}_{1})(\forall f_{1}^{1})(\psi_{1}(Y_{1}, f_{1},f)=0)\\
&\di (\exists^{\st} Y^{2}_{1})(\forall f_{1}^{1}\in u)(\psi_{1}(Y_{1}, f_{1},f)=0)\\
&\di (\exists  Y^{2}_{1}\in w)(\forall f_{1}^{1}\in u)(\psi_{1}(Y_{1}, f_{1},f)=0)\\
&\di (\exists Y^{2}_{1}\in w)\neg A_{1}(f,Y_{1})\di Z_{0}(f)=0\di Z(f)=0.
\end{align*}
Note that $\SOT$ is (only) necessary to establish the first implication.  Now, since $y$ from \eqref{forgik} contains all standard functions, the second conjunct of \eqref{forgik} implies (by definition)
that for standard $h^{1}$ (by the definition of $A$): 
\be\label{fronxi}
(\forall Y^{2}_{1}\in w)A_{1}(h,Y_{1})\vee (\forall Y^{2}_{2}\in w)A_{2}(h,Y_{2}).
\ee
Similarly, consider the following series of implications (for standard $f^{1}$):
\begin{align}
(\exists^{\st} Y^{2}_{2})(\forall^{\st} f_{2}^{1})(\psi_{2}(Y_{2}, f_{2},f)=0)
&\di (\exists^{\st} Y^{2}_{2})(\forall f_{2}^{1})(\psi_{1}(Y_{2}, f_{2},f)=0)\notag\\
&\di (\exists^{\st} Y^{2}_{2})(\forall f_{2}^{1}\in u)(\psi_{2}(Y_{2}, f_{2},f)=0)\notag\\
&\di (\exists  Y^{2}_{2}\in w)(\forall f_{2}^{1}\in u)(\psi_{2}(Y_{2}, f_{2},f)=0)\notag\\
&\di (\exists Y^{2}_{2}\in w)\neg A_{2}(f,Y_{2})\label{hoi2}\\
&\di (\forall Y^{2}_{1}\in w) A_{1}(f,Y_{1})\label{hoi1}\\
&\di Z_{0}(f)=1 \di Z(f)=1.\notag
\end{align}
Note that $\SOT$ is (only) necessary to establish the first implication, while \eqref{hoi1} follows from \eqref{hoi2} by \eqref{fronxi}.  
Thus, we observe that $Z^{2}$ is as required for $\Sigma_{2}^{1}\mSEP$ relative to `st', and we are done.  
\end{proof}
Note that $\P_{0}+\SOT$ exists at the level of second-order arithmetic, while $\Sigma_{1}^{2}\mSEP$ goes beyond that.  
In other words, $\STP_{2}$ yields a non-trivial step up in strength.  
The previous proof is readily generalised as follows: $ [\Sigma_{2}^{2}\mSEP]^{\st}$ follows from $\STP_{2}$ and \emph{Transfer} for $\Sigma_{1}^{1}$-formulas.  
Finally, the axiom $\STP_{2}$ has a normal form as follows.  
\begin{thm}\label{moel}
In $\P$, $\STP_{2}$ is equivalent to 
\begin{align}\label{notsobad3}
(\forall^{\st} \Psi^{2\di 1^{*}})(\exists^{\st}W^{2^{*}})(\forall Y^{2}\leq_{2}1)(\exists Z^{2}\in W)(\forall f\in \Psi(Z))(Z(f)=_{0}Y(f)  ).
\end{align}
\end{thm}
\begin{proof}
Clearly, $\STP_{2}$ implies (as standard sequences consist of standard elements): 
\be\label{angeli}
(\forall^{\st} \Psi^{2\di 1^{*}})(\forall Y^{2}\leq_{2}1)(\exists^{\st}Z^{2}\leq_{2}1)(\forall f\in \Psi(Z))(Z(f)=_{0}Y(f)  ), 
\ee
and the implication $\eqref{angeli}\di \STP_{2}$ is established as follows:  Suppose $\neg\STP_{2}$, i.e.\ there is $Y_{0}^{2}\leq_{2}1$ such that $(\forall^{\st}Z^{2}\leq_{2}1)(\exists^{\st}f^{1})(Z(f)\ne_{0}Y(f))$.  Applying $\HAC_{\INT}$ to the latter, we obtain the negation of \eqref{angeli}, and the latter is seen to be equivalent to $\STP_{2}$.
Finally, applying \emph{Idealisation} \textsf{I} to \eqref{angeli}, we obtain exactly \eqref{notsobad3}.
\end{proof}
The normal form \eqref{notsobad3} gives rise to the (non-unique) functional $\Sigma^{(2\di 1^{*})\di2^{*}}$ defined by 
the following specification:
\be\tag{$\textsf{CFS}(\Sigma)$}
 (\forall \Psi^{2\di 1^{*}})(\forall Y^{2}\leq_{2}1)(\exists Z^{2}\in \Sigma(\Psi))(\forall f\in \Psi(Z))(Z(f)=_{0}Y(f)  ), 
\ee
Intuitively, the open cover $\cup_{Y\in \{0,1\}^{\N^{\N}}}J_{Y}^{\Psi}$ has a finite sub-cover provided by $\Sigma(\Psi)$, 
where $J^{\Psi}_Y$ is the neighbourhood of all $Z \in \{0,1\}^{\N^\N}$ which agree with $Y$ on the finite sequence $\Psi(Y)$.
In contrast\footnote{Define $\SOT(\xi)\equiv (\forall Y^{2})\big[ (\exists f^{1})(Y(f)=0)\di Y(\xi(Y))=0  \big]$.  Combining the results from \cites{samflo,dagsamIII}, any $\xi^{3}$ satisfying $\SOT(\xi)$ computes $\Theta$ via a term of G\"odel's $T$, provable in $\RCAo+(\exists^{3})$.  Note that $\exists^{3}$ introduced in Section \ref{knowledge} is a variation of such $\xi$.} to special fan functionals, the functional $\Sigma$ requires a non-trivial instance of the {axiom of choice}.  The exact properties of $\Sigma$ are beyond the scope of this paper and will be studied in a subsequent paper.  

\medskip

Finally, we discuss the mathematical naturalness of compactness properties of function spaces, and the associated gauge integrals.  
\begin{rem}\label{feynman}\rm
The Feynman path integral is a central and fundamental object in physics, especially quantum mechanics. 
The Lebesgue integral does not provide an adequate formalisation for the path integral, but the latter \emph{can} be formalised using the gauge integral (\cite{mullingitover, secondmulling}) over \emph{function spaces}.  As shown in \cite{dagsamIII}*{\S3.3}, compactness as in $\HBU$ is essential for the development of the gauge integral on the unit interval, and the compactness of function spaces is similarly essential for the formalisation of the Feynman path integral.  However, as discussed in \cite{mulkerror}*{\S7}, the compactness of function spaces can be treacherous waters.  Hence, we only study $\STP_{2}$ as above in this paper, and will establish the exact connection to the gauge integral in a later publication.  
\end{rem}
\section{Conclusion}\label{kloi}

\subsection{Summary of results}\label{klap}
In this section, we provide a summary of the results in this paper and \cites{dagsam,dagsamIII}.
Figure \ref{clap} below summarises these results concisely.  

\medskip

By way of a legend, in the right column are the linearly ordered `Big Five' systems of RM, with above them full second-order arithmetic $\textsf{Z}_{2}$ and below them the system $\WWKL_{0}\equiv \RCA_{0}+\WWKL$.  
In the middle column, we classify the functionals studied in this paper as follows: $\RCAo$ plus the existence of the pictured functional is (at least or exactly) at the level of the corresponding system on the right; (struck out) arrows denote (non) S1-S9-computability.     
In the left column, we classify the nonstandard axioms studied in this paper as follows: $\P_{0}$ plus the pictured nonstandard axioms is (at least or exactly) at the level of the corresponding system on the right; (struck out) arrows denote (non)implication over $\P_{0}$. 
Many questions regarding this diagram remain unanswered, as discussed in Section \ref{future}.  
\begin{figure}[h!]
\begin{tikzpicture}[description/.style={fill=white,inner sep=2pt}]
\matrix (m) [matrix of math nodes, row sep=1.5em,
column sep=0.3em, text height=1.5ex, text depth=0.25ex]
{ ~&~&\textsf{SOT}&~&~&~& \exists^3&~&~&\textsf{Z}_{2}\\
 (\SIX)^{\st} &~ &\Pi_{2}^{1}\textsf{-TRANS}   &~ & \Paai+\STP &~& S^{2}+\Theta&~&~&\SIX\\
(S^{2})^{\st}& ~&\Paai &~ &~ & ~ &S^{2}&~&~&\FIVE\\
\!\!\!\!\!\!\textsf{\textup{ATR}}^{\st}&~&~ &~&  \paai+\STP&~& ~&&\exists^{2}+\Theta&\ATR_{0}\\ 
&&&& \paai+\LMP&&&&\exists^{2}+\Lambda&~\\
&&\paai&&&~& \exists^{2}&~& ~&\ACA_{0}\\
~ &~&~&~& ~~ &~&~&~& ~& \\
&\STP& \WKL^{\st} &~&&~&&\Theta^{3}& &\WKL_{0} \\
~ &&&& ~ \\
&& \WWKL^{\st} && \LMP&&~&&\Lambda^{3}&\WWKL_{0}\\};

\path[-]  (m-1-10) edge[->] (m-2-10);
\path[-]  (m-2-10) edge[->] (m-3-10);
\path[-]  (m-3-10) edge[->] (m-4-10);
\path[-]  (m-4-10) edge[->] (m-6-10);
\path[-]  (m-6-10) edge[->] (m-8-10);
\path[-]  (m-8-10) edge[->] (m-10-10);

  \draw[->] (m-5-9.south) -- node[ sloped, strike out,draw,-]{} (m-8-8);
\path[-]  (m-2-7.south east) edge[->] (m-4-9.north);
\draw[<->] (m-3-7.east)  -- node[sloped, strike out,draw,-]{} (m-4-9.north west);
\draw[->] (m-3-7.south)  -- node[sloped, strike out,draw,-]{} (m-5-9.north west);
\path[-]  (m-4-9) edge[->] (m-5-9.north);
\path[-]  (m-6-7.north east) edge[<-] (m-5-9.west);
\path[-]  (m-1-7) edge[->] (m-2-7);
\path[-]  (m-2-7) edge[->] (m-3-7);

\path[-]  (m-3-7) edge[->] (m-6-7);
  \draw[->] (m-6-7) -- node[ strike out,draw,-]{} (m-8-8);
		\path[-]  (m-3-7.south) edge[->, bend right=300,looseness=0.2] (m-8-8.north);
\path[-]  (m-1-7.west) edge[->, bend left=300,looseness=0.5] (m-8-8);
\path[-]  (m-8-8) edge[->, bend right=300,looseness=0.4] (m-10-9);
  \draw[->] (m-10-9) -- node[ strike out,draw,-]{} (m-8-8);
  
 \draw (3,0.5) -- (3.5,1);

\path[-]  (m-1-3) edge[->] (m-2-3);
\path[-]  (m-2-3) edge[->] (m-3-3);
\path[-]  (m-3-3) edge[->] (m-6-3);
\path[-]  (m-1-3.north west) edge[->, bend left=300,looseness=0.4] (m-2-1);
\path[-]  (m-2-1.south) edge[->] (m-3-1);
\path[-]  (m-4-5.north) edge[<-] (m-2-5);
\path[-]  (m-2-5.north west) 
		edge[->, bend left=300,looseness=0.4] (m-2-1.north east);
\path[-]  (m-4-5.west) 
		edge[->, bend right=300,looseness=0.4] (m-4-1.east);
\path[-]  (m-1-3.west) 
		edge[->, bend left=300,looseness=0.4] (m-8-2.north);
		 (m-6-3)   edge[<- ](m-3-3); 
	\path[-](m-6-3.south)	 edge[->] (m-8-3);
		      
  \draw[->] (m-6-3.south) -- node[ strike out,draw,-]{} (m-10-5);
    \draw[->] (m-6-3.south) -- node[sloped, strike out,draw,-]{} (m-8-2);
  \draw[->] (m-3-3.south) -- node[ strike out,draw,-]{} (m-10-5);
    \draw[->] (m-3-3.south) -- node[strike out,draw,-]{} (m-5-5.north west);
        \draw[->] (m-3-3.south) -- node[]{} (m-4-1.north east);
           \draw[<-] (m-8-2.north west) -- node[sloped, strike out,draw,-]{}  (m-4-1.south west);
           
    \draw[->] (m-3-3.south) -- node[sloped, strike out,draw,-]{} (m-8-2);
\path[-]	(m-8-2) edge[->] (m-8-3);
  \draw[->] (m-8-3.south) -- node[ strike out,draw,-]{} (m-10-5);
\path[-]	(m-8-3.south) edge[->] (m-10-3);
\path[-]	(m-8-2.south) edge[->] (m-10-3);
  \draw[->] (m-10-3) -- node[sloped, strike out,draw,-]{} (m-10-5);
\path[-]  (m-8-2.south) edge[->, bend right=400,looseness=0.4] (m-10-5.north west);
  \draw[->] (m-5-5.south) -- node[sloped, strike out,draw,-]{} (m-8-2);
    \path[-]  (m-5-5.south) edge[->] (m-6-3.east);
       \path[-] (m-5-5) edge[->] (m-10-5);
       \path[-]  (m-1-3.north east) 
		edge[->, bend right=300,looseness=0.5] (m-2-5);
		\path[-]	(m-2-3) edge[->] (m-2-1);
			  \draw[<-] (m-2-5) -- node[ strike out,draw,-]{} (m-2-3);
	\path[-](m-5-5.north) edge[<-] (m-4-5.south);
		    \draw[<->] (m-3-3.south east) -- node[ strike out,draw,-]{} (m-4-5.north west);
		        \draw[->] (m-3-3.west) -- node[-]{} (m-3-1.east);
		        \draw[<-] (m-4-1.north) -- node[-]{} (m-3-1.south);
		        \draw[<-] (m-4-1.south) -- node[ strike out,draw,-]{} (m-5-5.west);
\end{tikzpicture}

\caption{Summary of results}
\label{clap}
\end{figure}
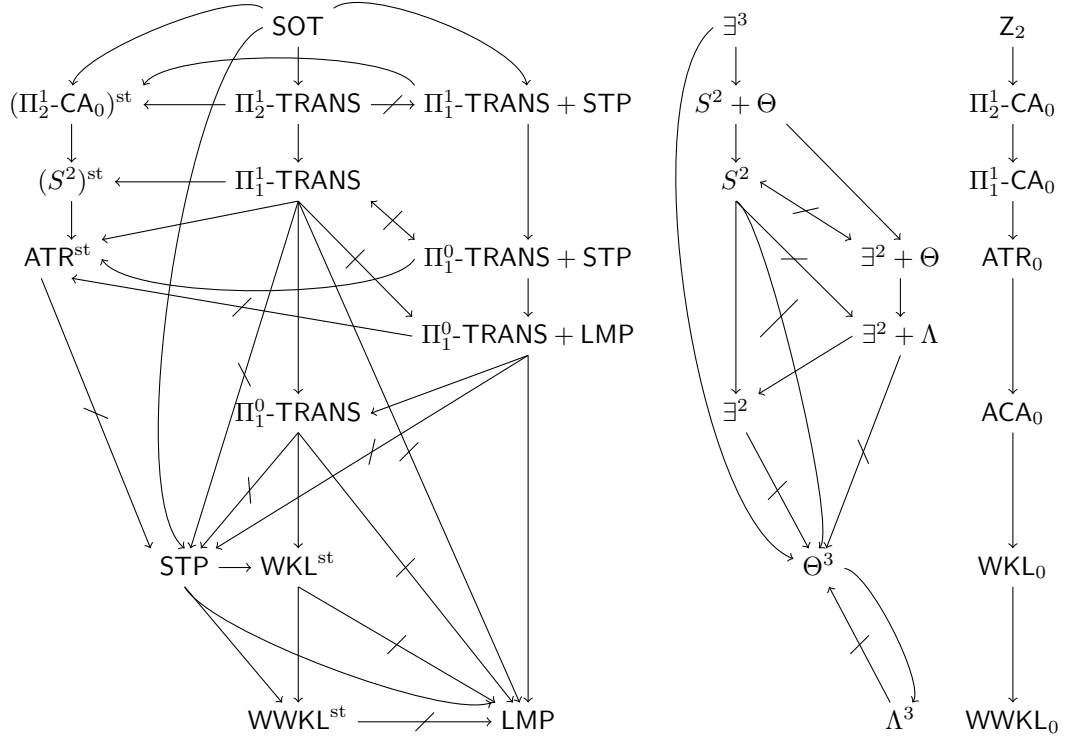
\subsection{Future research}\label{future}
We discuss some open questions and future research.
\begin{enumerate}
 \renewcommand{\theenumi}{\roman{enumi}}
\item The system $\SIX+\Pi_{3}^{1}\textsf{-TI}_{0}$ proves $\Delta_{3}^{0}$-determinacy, while $\SIX $ does not (see \cite{taka}).  
Hence, it is a natural question whether $\P+\STP+\Paai$ proves transfinite induction as in $[\Pi_{3}^{1}\textsf{\textup{-TI}}_{0}]^{\st}$.
\item What is the strength of nonstandard versions of \emph{Hindman's theorem} (\cite{dsliceke}*{\S10.3.5})?
The latter is strictly between $\ACA_{0}$ and $\ATR_{0}$.  
\item What is the strength of nonstandard versions of $\DNR$?  Can these be derived from $\Lambda$ and $\LMP$?
\item What is the strength of nonstandard versions of $\POS$ and $2$-$\WWKL$?  What is their relation to $\Lambda$ and $\LMP$?
\item Combining $\exists^{2}$ or $\mu_{1}$ with $\Theta$ results in a considerable jump in logical strength.  Which functionals yield a similar jump in strength?
\item Does the RM of $\WWKL$ give rise to interesting variations of $\Lambda$?
\item There are numerous theorems in classical analysis essentially of the form $(\forall x^2) (\exists y^{1/0}) \Phi(x,y)$, 
and each of these defines a class of realisers $\zeta^{2\di 1/0}$ such that $(\forall x^2)   \Phi(x,\zeta(x))$. A general investigation of the relative computational powers of such realisers, say modulo $\mu^{2}$ or $\exists^2$, is warranted.
We believe this study is intimately related to the RM study of the original theorems $(\forall x^2) (\exists y^{1/0}) \Phi(x,y)$, and associated theorems from Nonstandard Analysis.
\end{enumerate}
Furthermore, we have established a close link between $\Lambda$ and the \emph{Vitali covering lemma}, which we hope to develop further.  
Finally, the combination of $\Theta$ and the Suslin functional yields Gandy's \emph{Superjump} (\cite{dagsamIII}), and we have additionally 
established that the former combination goes \emph{far} beyond the latter functional. 
We hope to establish the exact (logical and computational) strength of the aforementioned combination in the future. 

\appendix

\section{Some systems of Nonstandard Analysis}\label{tomzeiker}
In this section, we introduce Nelson's axiomatic approach to Nonstandard Analysis \emph{internal set theory} (\cite{wownelly}), and it fragments based on Peano arithmetic from \cite{brie}. 
This background provides the definition for the systems $\P_{0}$ and $\P$ used above. 
\subsection{Internal set theory}\label{IIST}
In Nelson's \emph{syntactic} approach to Nonstandard Analysis (\cite{wownelly}), as opposed to Robinson's semantic one (\cite{robinson1}), a new predicate `st($x$)', read as `$x$ is standard' is added to the language of \textsf{ZFC}, the usual foundation of mathematics.  
The notations $(\forall^{\st}x)$ and $(\exists^{\st}y)$ are short for $(\forall x)(\st(x)\di \dots)$ and $(\exists y)(\st(y)\wedge \dots)$.  A formula is called \emph{internal} if it does not involve `st', and \emph{external} otherwise.   
The three external axioms \emph{Idealisation}, \emph{Standard Part}, and \emph{Transfer} govern the new predicate `st';  They are respectively defined\footnote{The superscript `fin' in \textsf{(I)} means that $x$ is finite, i.e.\ its number of elements are bounded by a natural number.} as:  
\begin{enumerate}
\item[\textsf{(I)}] $(\forall^{\st~\textup{fin}}x)(\exists y)(\forall z\in x)\varphi(z,y)\di (\exists y)(\forall^{\st}x)\varphi(x,y)$, for any internal $\varphi$.  
\item[\textsf{(S)}] $(\forall^{\st} x)(\exists^{\st}y)(\forall^{\st}z)\big((z\in x\wedge \varphi(z))\asa z\in y\big)$, for any $\varphi$.
\item[\textsf{(T)}] $(\forall^{\st}t)\big[(\forall^{\st}x)\varphi(x, t)\di (\forall x)\varphi(x, t)\big]$, where $\varphi(x,t)$ is internal, and only has free variables $t, x$.  
\end{enumerate}
The system \textsf{IST} is just \textsf{ZFC} extended with the aforementioned external axioms;  
$\IST$ is a conservative extension of \textsf{ZFC} for the internal language, as proved in \cite{wownelly}.    

\medskip

Clearly, the extension from $\ZFC$ to $\IST$ can also be done for \emph{subsystems} of the former.   
Such extensions are studied in \cite{brie} for the classical and constructive formalisations of arithmetic, i.e.\ \emph{Peano arithmetic} and \emph{Heyting} arithmetic.  
In particular, the systems studied in \cite{brie} are \textsf{E-HA}$^{\omega}$ and $\textsf{E-PA}^{\omega}$, respectively \emph{Heyting and Peano arithmetic in all finite types and the axiom of extensionality}.       
We refer to \cite{kohlenbach3}*{\S3.3} for the exact definitions of the (mainstream in mathematical logic) systems \textsf{E-HA}$^{\omega}$ and $\textsf{E-PA}^{\omega}$.  
We introduce in Section \ref{PIPI} the system $\P$, the (conservative) extension of $\textsf{E-PA}^{\omega}$ with fragments of the external axioms of $\IST$.  

\medskip

Finally, \textsf{E-PA}$^{\omega*}$ is the definitional extensions of \textsf{E-PA}$^{\omega}$ with types for finite sequences, as in \cite{brie}*{\S2}.  For the former system, we require some notation.  
\begin{nota}[Finite sequences]\label{skim}\rm
The systems $\textsf{E-PA}^{\omega*}$ and $\textsf{E-HA}^{\omega*}$ have a dedicated type for `finite sequences of objects of type $\rho$', namely $\rho^{*}$.  Since the usual coding of pairs of numbers goes through in both, we shall not always distinguish between $0$ and $0^{*}$.  
Similarly, we do not always distinguish between `$s^{\rho}$' and `$\langle s^{\rho}\rangle$', where the former is `the object $s$ of type $\rho$', and the latter is `the sequence of type $\rho^{*}$ with only element $s^{\rho}$'.  The empty sequence for the type $\rho^{*}$ is denoted by `$\langle \rangle_{\rho}$', usually with the typing omitted.  Furthermore, we denote by `$|s|=n$' the length of the finite sequence $s^{\rho^{*}}=\langle s_{0}^{\rho},s_{1}^{\rho},\dots,s_{n-1}^{\rho}\rangle$, where $|\langle\rangle|=0$, i.e.\ the empty sequence has length zero.
For sequences $s^{\rho^{*}}, t^{\rho^{*}}$, we denote by `$s*t$' the concatenation of $s$ and $t$, i.e.\ $(s*t)(i)=s(i)$ for $i<|s|$ and $(s*t)(j)=t(j-|s|)$ for $|s|\leq j< |s|+|t|$. For a sequence $s^{\rho^{*}}$, we define $\overline{s}N:=\langle s(0), s(1), \dots,  s(N)\rangle $ for $N^{0}<|s|$.  
For a sequence $\alpha^{0\di \rho}$, we also write $\overline{\alpha}N=\langle \alpha(0), \alpha(1),\dots, \alpha(N)\rangle$ for \emph{any} $N^{0}$.  By way of shorthand, $q^{\rho}\in Q^{\rho^{*}}$ abbreviates $(\exists i<|Q|)(Q(i)=_{\rho}q)$.  Finally, we shall use $\underline{x}, \underline{y},\underline{t}, \dots$ as short for tuples $x_{0}^{\sigma_{0}}, \dots x_{k}^{\sigma_{k}}$ of possibly different type $\sigma_{i}$.          
\end{nota}    
\begin{rem}[Notation]\label{equ}\rm
The system $\textsf{E-PA}^{\omega*}$ includes equality between natural numbers `$=_{0}$' as a primitive.  Equality `$=_{\tau}$' and inequality $\leq_{\tau}$ for $x^{\tau},y^{\tau}$ is:
\be\label{aparth}
[x=_{\tau}y] \equiv (\forall z_{1}^{\tau_{1}}\dots z_{k}^{\tau_{k}})[xz_{1}\dots z_{k}=_{0}yz_{1}\dots z_{k}],
\ee
\be\label{aparth1}
[x\leq_{\tau}y] \equiv (\forall z_{1}^{\tau_{1}}\dots z_{k}^{\tau_{k}})[xz_{1}\dots z_{k}\leq_{0}yz_{1}\dots z_{k}],
\ee
if the type $\tau$ is composed as $\tau\equiv(\tau_{1}\di \dots\di \tau_{k}\di 0)$.
In the spirit of Nonstandard Analysis, we define `approximate equality $\approx_{\tau}$' as follows (with the type $\tau$ as above):
\be\label{aparth2}
[x\approx_{\tau}y] \equiv (\forall^{\st} z_{1}^{\tau_{1}}\dots z_{k}^{\tau_{k}})[xz_{1}\dots z_{k}=_{0}yz_{1}\dots z_{k}]
\ee
All the above systems include the \emph{axiom of extensionality} for all $\varphi^{\rho\di \tau}$ as follows:
\be\label{EXT}\tag{\textsf{E}}  
(\forall  x^{\rho},y^{\rho}) \big[x=_{\rho} y \di \varphi(x)=_{\tau}\varphi(y)   \big].
\ee
However, as noted in \cite{brie}*{p.\ 1973}, the so-called axiom of \emph{standard} extensionality \eqref{EXT}$^{\st}$ is problematic and cannot be included in $\P$ or $\P_{0}$.  
\end{rem}

\subsection{The classical systems $\P$ and $\P_{0}$}\label{PIPI}
We first introduce the system $\P$, a conservative extension of $\textsf{E-PA}^{\omega}$ with fragments of Nelson's $\IST$.  

\medskip

To this end, we first introduce the base system $\textsf{E-PA}_{\st}^{\omega*}$.  
We use the same definition as \cite{brie}*{Def.~6.1}, where \textsf{E-PA}$^{\omega*}$ is the definitional extension of \textsf{E-PA}$^{\omega}$ with types for finite sequences as in \cite{brie}*{\S2}.  
The set $\T^{*}$ is defined as the collection of all the constants in the language of $\textsf{E-PA}^{\omega*}$.    
\bdefi\label{debsa}
The system $ \textsf{E-PA}^{\omega*}_{\st} $ is defined as $ \textsf{E-PA}^{\omega{*}} + \T^{*}_{\st} + \textsf{IA}^{\st}$, where $\T^{*}_{\st}$
consists of the following axiom schemas.
\begin{enumerate}
\item The schema\footnote{The language of $\textsf{E-PA}_{\st}^{\omega*}$ contains a symbol $\st_{\sigma}$ for each finite type $\sigma$, but the subscript is essentially always omitted.  Hence $\T^{*}_{\st}$ is an \emph{axiom schema} and not an axiom.\label{omita}} $\st(x)\wedge x=y\di\st(y)$,
\item The schema providing for each closed term $t\in \T^{*}$ the axiom $\st(t)$.
\item The schema $\st(f)\wedge \st(x)\di \st(f(x))$.
\end{enumerate}
The external induction axiom \textsf{IA}$^{\st}$ states that for any (possibly external) $\Phi$:  
\be\tag{\textsf{IA}$^{\st}$}
\Phi(0)\wedge(\forall^{\st}n^{0})(\Phi(n) \di\Phi(n+1))\di(\forall^{\st}n^{0})\Phi(n).     
\ee
\edefi
Secondly, we introduce some essential fragments of $\IST$ studied in \cite{brie}.  
\bdefi~
\begin{enumerate}
\item$\HAC_{\INT}$: For any internal formula $\varphi$, we have
\be\label{HACINT}
(\forall^{\st}x^{\rho})(\exists^{\st}y^{\tau})\varphi(x, y)\di \big(\exists^{\st}F^{\rho\di \tau^{*}}\big)(\forall^{\st}x^{\rho})(\exists y^{\tau}\in F(x))\varphi(x,y),
\ee
\item $\textsf{I}$: For any internal formula $\varphi$, we have
\[
(\forall^{\st} x^{\sigma^{*}})(\exists y^{\tau} )(\forall z^{\sigma}\in x)\varphi(z,y)\di (\exists y^{\tau})(\forall^{\st} x^{\sigma})\varphi(x,y), 
\]
\item The system $\P$ is $\textsf{E-PA}_{\st}^{\omega*}+\textsf{I}+\HAC_{\INT}$.
\end{enumerate}
\end{defi}
Note that \textsf{I} and $\HAC_{\INT}$ are fragments of Nelson's axioms \emph{Idealisation} and \emph{Standard part}.  
By definition, $F$ in \eqref{HACINT} only provides a \emph{finite sequence} of witnesses to $(\exists^{\st}y)$, explaining its name \emph{Herbrandized Axiom of Choice}.   

\medskip

The system $\P$ is connected to $\textsf{E-PA}^{\omega}$ by Theorem \ref{consresultcor} which expresses that we may obtain effective results as in \eqref{effewachten} from any theorem of Nonstandard Analysis which has the same form as in \eqref{bog}.  The scope of this theorem includes the Big Five systems of Reverse Mathematics (\cite{sambon}), the Reverse Mathematics zoo (\cite{samzooII}), and both classical and higher-order computability theory (\cite{samGH, sambon3}).  

\medskip

We now introduce the system $\P_{0}$, a conservative extension of $\RCAo$ with fragments of Nelson's $\IST$.  
Recall that the system $\RCAo\equiv \textsf{E-PRA}^{\omega}+\QFAC^{1,0}$ is Kohlenbach's \emph{base theory of higher-order Reverse Mathematics} as introduced in \cite{kohlenbach2}*{\S2}.  
The system $ \textsf{\textup{E-PRA}}^{\omega*}$ is an obvious definitional extensional as in Remark~\ref{skim}. 
Recall that we permit ourselves a slight abuse of notation by also referring to $\textsf{E-PRA}^{\omega*}+\QFAC^{1,0}$ as $\RCAo$.  
\bdefi
The system $\P_{0}$ is $\textsf{E-PRA}^{\omega*}+\QFAC^{1,0}+\T_{\st}^{*}+\textsf{I}+\HAC_{\INT}$.
\edefi
Finally, the system $\P_{0}$ is connected to $\RCAo$ by Corollary \ref{consresultcor2}.

\begin{ack}\rm
Our research was supported by FWO Flanders, the John Templeton Foundation, the Alexander von Humboldt Foundation, LMU Munich (via their Excellence Initiative), the University of Oslo, and the Japan Society for the Promotion of Science.  
The authors express their gratitude towards these institutions. 
The authors thank Tom Powell, Martin Hyland, Ulrich Kohlenbach, and Anil Nerode for their valuable advice.   The referee and editor were also instrumental in greatly improving this paper. 
\end{ack}

\bye